\documentclass[a4paper, 12pt, oneside, reqno, notitlepage]{amsart}
\usepackage{amsmath,amssymb,amsthm,graphicx,mathrsfs,bbm,url}
\usepackage[margin=3cm]{geometry}
\usepackage{amsthm}
\usepackage{enumitem}
\usepackage{mathtools}
\usepackage[usenames,dvipsnames]{color}
\usepackage[colorlinks=true,linkcolor=red,citecolor=green]{hyperref}
\usepackage[super]{nth}
\usepackage[open, openlevel=2, depth=3, atend]{bookmark}
\hypersetup{pdfstartview=XYZ}
\usepackage[font=footnotesize]{caption}
\usepackage[colorinlistoftodos]{todonotes}
\usepackage{tikz-cd}
\tikzcdset{row sep/normal=3.6em,column sep/normal=3.6em}

\RequirePackage[style=alphabetic,giveninits=true,maxnames=50]{biblatex} 
\usepackage{subcaption}
\usepackage{footmisc}

\usepackage{epstopdf}

\theoremstyle{plain}
\newtheorem{theorem}{Theorem}
\newtheorem{lemma}[theorem]{Lemma}

\newtheorem{proposition}[theorem]{Proposition}
\newtheorem{corollary}[theorem]{Corollary}

\theoremstyle{definition}
\newtheorem{definition}[theorem]{Definition}

\theoremstyle{remark}
\newtheorem{remark}[theorem]{Remark}

\newcommand{\dom}{\mathrm{dom}}
\newcommand{\nbv}{\mathrm{NBV}}
\newcommand{\sbv}{\mathrm{SBV}}
\newcommand{\bv}{\mathrm{BV}}
\newcommand{\im}{\mathrm{im}}
\newcommand{\lip}{\mathrm{Lip}}
\newcommand{\gr}{\mathrm{graph}}

\title
[Existence of spectral submanifolds in time delay systems] 
{Existence of spectral submanifolds in time delay systems}

\author{Gergely Buza}
\address{Institute for Mechanical Systems, ETH Zürich, Leonhardstrasse 21, 8092 Zurich, Switzerland}
\email{buzag@ethz.ch}

\author{George Haller}
\address{Institute for Mechanical Systems, ETH Zürich, Leonhardstrasse 21, 8092 Zurich, Switzerland}
\email{georgehaller@ethz.ch}

\begin{document}

\begin{abstract}
    Spectral submanifolds (SSMs) are invariant manifolds of a dynamical system, defined by the property of being tangent to a spectral subspace of the linearized dynamics at a steady state.
    We show existence, along with certain desirable properties such as smoothness, attractivity and conditional uniqueness, of SSMs associated to a large class of spectral subspaces in time delay systems.
    Building on these results, we generalize the criteria for existence of inertial manifolds -- defined as globally exponentially attracting Lipschitz invariant manifolds of finite dimension -- 
    and show that they need not have dimension equal to that of the physical configuration, in contrast to previous accounts.
    We then demonstrate the applicability of these results on a few simple examples.
\end{abstract}

\maketitle

\section{Introduction}

The potential existence of lower-dimensional representations of evolutionary systems is central to many application-oriented subdisciplines of mathematical physics.
Perhaps the most natural way to achieve dimensional reduction is by 
restricting attention to lower-dimensional attracting
invariant manifolds. 
Our core interest here is to uncover a certain class of such invariant manifolds in the phase space of delay differential equations (DDEs), with the objective of characterizing their internal asymptotic dynamics.
Delay  equations are abundant in the applied sciences, with examples ranging from epidemic theory \cite{diekmann2000mathematical}, population dynamics \cite{cushing2013integrodifferential} and other areas of mathematical biology \cite{glass2021nonlinear,milton2009time,zhang2018saturation} to a large variety of fields on the engineering side \cite{insperger2011semi}, including control theory \cite{richard2003time,habib2022bistability}, machine tool vibrations \cite{molnar2017analysis} and traffic dynamics \cite{orosz2010traffic}.

The evolution in a delay equation with locally Lipschitz continuous right-hand side can be characterized by a  semiflow $\{\varphi_t\}_{t \geq 0}$ acting on an open subset of the space of continuous functions, $X =C([-h,0];\mathbb{R}^n)$.
Here $n$ is the physical dimension of the space on which the delay equation is posed and $h$ is the maximal delay.
The invariant manifolds of interest to us are  spectral submanifolds (or SSMs), i.e., manifolds that act as nonlinear continuations of dominant eigenspaces of a linearized delay equation (see \cite{haller2016nonlinear} and \cite{haller2025modeling}). 
This is a local concept based at a stationary state $u_0 \in X$ (i.e., $\varphi_t(u_0) = u_0 $ for all $t \geq 0$)\footnote{More general invariant sets are of course permitted, but in this work we focus on the case of a fixed point.} of the phase space, i.e.\ SSMs need only be constructed on a neighbourhood of $u_0$.
We may suppose without loss of generality that $u_0=0$ (see Section~\ref{sect:linearization}).
Suppose the dynamics, linearized about $0$, are driven by a generator $A$. 
An SSM associates to a spectral subset $\Sigma \subset \sigma(A)$ a manifold $W^\Sigma$, that is locally invariant under $\varphi_t$ and tangent to the spectral subspace $\im (P_\Sigma)$ at $0$.\footnote{Here, $P_\Sigma$ is the spectral projection associated to $\Sigma$, assuming that it is well defined.}
The overarching objective of SSMs is to subdivide the dynamics (preserving its nonlinear structure) according to asymptotic properties of trajectories in the vicinity of the stationary state, akin to spectral projections for linear systems. 
Early versions of this concept date back to the stable, unstable and center manifold theorems
\cite{liapounoff1907probleme,hirsch1970stable,kelley1966stable}, 
where the spectral subsets are confined to either a half-space  of the complex plane or the imaginary axis.
In practical applications, especially to dissipative systems, it is useful to maintain generality and consider arbitrary spectral subsets.
Of main interest are of course pseudo-unstable and slow manifolds characterized by the rightmost part of the spectrum \cite{irwin1980new,chen1997invariant}, 
since they correspond 
to the asymptotically prevalent part of the dynamics.
There exist, however, scenarios in which an alternative choice of $\Sigma$ proves more beneficial (e.g., the buckling beam example of Haller et al.~\cite{haller2023nonlinear}).

The theoretical foundations of SSMs were laid by the parameterization method of Cabre et al.~\cite{Cabre2003a,cabre2003parameterization2,cabre2005parameterization3,haro2006parameterization,haro2006parameterizationb} (with assumptions optimalized for the stable case).
The classical linearization results of Sternberg \cite{sternberg1957local,sternberg1958structure} allow for more general hyperbolic equilibria but under more stringent nonresonance assumptions \cite{haller2025modeling}. 
In recent years, the evolution of SSM theory, as well as its equation- and data-driven implementations, were motivated by contemporary problems of increasing physical complexity.
Specifically, soft robotics inspired an extension of the theory to temporally aperiodically forced systems \cite{haller2024nonlinear,kaundinya2025data}; 
the identification of edge states in fluid dynamics \cite{kaszas2022dynamics} and the aforementioned beam buckling example motivated the introduction of fractional and mixed-mode SSMs \cite{haller2023nonlinear};
the precise identification of nonlinearities in
tabletop
gravity experiments prompted the development of oblique projection to SSMs along their stable foliation \cite{bettini2025data}  (see also \cite{szalai2020invariant} on general results on such foliations);
nonlinear structural vibration problems motivated the development of SSM theory for random dynamical systems (\cite{xu2025nonlinear}).
Advancements on the practical side include a more direct equation-driven formulation for finite element models \cite{jain2022compute,vizzaccaro2022high,opreni2023high},  and a data-driven approach for constructing SSMs \cite{cenedese2022data,axaas2023fast}, enabling treatment of large-scale problems. 
All these advances along with further ones are discussed in the book of Haller \cite{haller2025modeling}.

Given the prevalence of time delay systems in applied sciences, it is pressing to extend the theory to encompass such systems.
A pathway in this direction has already been initiated by the recent work of Szaksz et al.~\cite{szaksz2024reduction,szaksz2025spectral}, who 
demonstrate, importantly, that the setting of delay equations permit the SSM reduction to be performed in the function space (serving as the phase space) itself 
without the need for a priori discretization. 
Here, we prove existence and regularity results in order to provide the rigorous foundations facilitating such a reduction.

The idea to seek finite-dimensional approximations of delay equations is by no means novel: 
several  distinct  techniques have been developed over the years.
Numerical methods aimed at discretizing the full equation are surveyed in the work of Bellen and Zennaro \cite{bellen2013numerical}, with a particularly popular example being the semi-discretization method of Insperger and Stépán \cite{insperger2002semi,insperger2011semi}; a survey with specific emphasis on how (and whether) dynamical qualities of the full, infinite-dimensional system are preserved under such discretizations was conducted by Garay \cite{garay2005brief}.
Works on model reduction thus far have typically focused on linear systems,
achieving lower dimensionality by means of an optimal choice of linear projection that best resembles the original dynamics \cite{michiels2011krylov,beattie2009interpolatory}.
An alternative approach permitting a certain class of nonlinearities is proposed by van de Vouw et al.~\cite{van2015model}.
Techniques aimed at invariant manifold computations have focused on
the center and unstable manifolds for the most part \cite{krauskopf2003computing,green2004one,sahai2009numerical,scarciotti2015model} (see also \cite{farkas2001unstable,farkas2002small} for an early initiative providing rigorous theoretical results for the convergence of such approximations even for  center-unstable manifolds).
These works compute the invariant manifolds after having discretized the equation, as opposed to the aforementioned approach of Szaksz et al.~\cite{szaksz2025spectral} and that of Groothedde and Mireles-James \cite{groothedde2017parameterization}, the latter being the first instance in which the parameterization method was applied to invariant manifold computation in delay equations. 
Also notable are the works of He and de la Llave \cite{he2017construction,he2016construction} that extract quasiperiodic solutions of DDEs  with state-dependent delay  via the parameterization method. Subsequent work by Yang et al.~\cite{yang2021parameterization} focuses on limit cycles but also extracts their slow stable manifolds.

On the theoretical side of the delay equations literature, stable ($ \Sigma = \{ \lambda \in \sigma(A) \; | \; \mathrm{Re } \, \lambda < 0\}$), unstable ($\Sigma = \{\lambda \in \sigma(A) \; | \; \mathrm{Re } \, \lambda > 0\}$) and center ($ \Sigma =\mathrm{i} \mathbb{R} \cap \sigma(A)$)  manifolds were already subject of discussion in the monograph of Diekmann et al.~\cite{diekmann2012delay} (see also the earlier references \cite{hale1964neighborhood} and \cite{diekmann1984invariant}).
Subsequent theoretical advancements prioritized encompassing a larger variety of problems rather than types of manifolds (i.e., $\Sigma$ remained confined to the same subsets).
Particular attention has been given to the notoriously difficult problem of state-dependent delays \cite{hartung2006functional} and center(-stable) manifolds thereof \cite{krisztin2006c1,qesmi2009center}.
Moreover, center manifold theory has been extended to cover the case of infinite delays \cite{matsunaga2015center} and that of impulsive DDEs \cite{church2018smooth}.
In contrast,
Example 5.2 of
Chen et al.~\cite{chen1997invariant} considers global pseudo-unstable manifolds ($ \Sigma = \{\lambda \in \sigma(A) \; | \; \mathrm{Re } \, \lambda \geq \gamma\}$ for some $\gamma \in \mathbb{R}$) of neutral functional differential equations (thus including DDEs) in the functional setting of Hale and Lunel \cite{hale2013introduction}\footnote{The setup described in \cite{chen1997invariant} is akin to our Remark~\ref{remark:IM_with_F} -- hence a subsequent localization is needed to dispose of the global Lipschitz assumption and obtain tangency results, at the expense of constraining the domain. 
}. 
The localization of this approach, which we carry out here to obtain SSMs, is more natural to perform in the setting known as sun-star calculus due to Diekmann et al.~\cite{diekmann2012delay}, utilizing perturbation theory of dual semigroups.
The idea for this first came about in the sequence of papers
\cite{diekmann1987perturbed,clement1987perturbation,clement1988perturbation,clement1989perturbation,clement2020perturbation,diekmann2023perturbation}.

A related concept, an inertial manifold (IM) 
is defined as a globally exponentially attracting, finite dimensional, invariant Lipschitz manifold \cite{foias1988inertial,constantin2012integral,TITI1990}. 
These properties render IMs the most desirable invariant manifolds for model reduction purposes.
Their existence, however, is generally difficult (if not impossible) to prove. 
The term inertial manifold was initially coined by Foias et al.~\cite{foias1988inertial} in the context of parabolic PDEs; the question of their existence in the two-dimensional Navier-Stokes equations, however, has remained an open problem since then.
In all classical proofs, the existence of IMs hinges on the ratio of the global Lipschitz constant of the nonlinearity and the spectral gap being sufficiently small. 
In Navier-Stokes flows, this ratio cannot be controlled in an arbitrary fashion by moving down the sequence of eigenvalues of the Laplacian due to the presence of a spatial derivative in the nonlinearity. (This is unlike in  reaction-diffusion equations, for which this procedure yields IMs \cite{robinson2001}).

In contrast, existence of IMs in the small-delay regime of DDEs has been observed as early as the 60s   in a sequence of works by Ryabov \cite{ryabov1960application,ryabov1960application2,ryabov1961application,ryabov1963certain} (see also \cite{driver1968ryabov,driver1976linear,chicone2003inertial}), who termed trajectories on IMs special solutions.
Their method of proof is focused on extracting such trajectories directly, hence circumventing any mention of spectral gaps or even invariant manifolds, but it has since been shown that these trajectories actually form a $C^1$ manifold \cite{chicone2003inertial}.
Recently, these results were extended to neutral functional differential equations \cite{chen2022smooth} with improved smoothness (see also \cite{anikushin2023frequency}).
All of these works assume a small delay $h$, and the dimension of the inertial manifold is always equal to the dimension of the physical space $n$.

In this work, we show that small delays are not a necessity for inertial manifolds to exist. 
We reformulate the problem using the results of Chen et al.~\cite{chen1997invariant} to the familiar smallness condition on the ratio the global Lipschitz constant of the nonlinearity to the spectral gap.
This also has the geometric advantage of identifying the IM as an extension of an SSM;
that is, we obtain its tangency to a spectral subspace at a fixed point under appropriate conditions.
In particular, we demonstrate, through examples as well, that there exist inertial manifolds in delay equations with dimensions differing from $n$.
We then show explicitly how shrinking the delay is a suitable way to obtain arbitrarily large spectral gaps, and hence arbitrary smoothness and rate of attraction towards the IM, thus recovering the results of \cite{chen2022smooth} for the case of DDEs.
Roughly speaking, a spectral gap obtained in this way arises between eigenvalues of the system in the limit $h \to 0$ and the infinitely many new ones emanating from $-\infty$ due to the addition of the delay.
In this case, therefore, the IM has to be of dimension $n$, which is sometimes called the case of negligible delays \cite{kurzweil2006small}.

\subsection{Overview}

The ideal setting for the purposes herein is the sun-star calculus of Diekmann et al.~\cite{diekmann2012delay}, since it provides a convenient way to 
linearize about a fixed point using variation of constants formulae.
We recollect the basics of this theory in Sections \ref{sect:DDEs} and \ref{sect:linearization}.

As already noted, the convention in the delay equations literature is to take 
as phase space $X := C\big([-h,0];\mathbb{R}^n \big)$, the Banach space of continuous functions equipped with the supremum norm $| \cdot|_X :=| \cdot |_{\infty}$, for $h>0$ (the maximal delay) and $n \in \mathbb{N}$.
Let $f : O \to \mathbb{R}^n$ be a locally Lipschitz continuous function on the open set $O \subset X$.
We consider autonomous DDEs of the form
\begin{subequations} \label{eq:DDE_classical} 
\begin{align}
&\dot{x}(t) = f(x_t),  \label{eq:DDE_classical1} \\ 
& x_0 = u,  \label{eq:DDE_classical2}
\end{align}
\end{subequations}
where 
$x_t : \theta \mapsto x(t+\theta) \in X$ denotes the state variable and $u \in O$ the initial condition.
A solution to \eqref{eq:DDE_classical} is a function $x: [-h,t^*) \to \mathbb{R}^n$, $t^* \in (0,\infty) \cup \{\infty\}$, 
which satisfies \eqref{eq:DDE_classical2},
is continuously differentiable on $(0,t^*)$ such that \eqref{eq:DDE_classical1} holds; and which moreover possesses a right derivative equal to $f(x_0)$ at $0$.
Classical results assert (see, for example, \cite{hale2013introduction,diekmann2012delay}) that under the local Lipschitz assumption on $f$ there exists a unique solution to \eqref{eq:DDE_classical} for each initial condition $u \in O$.
Let us denote by $\{\varphi_t\}_{t \geq 0}$ the semiflow \eqref{eq:DDE_classical} generates.
Regularity properties of $\varphi$ will be discussed more thoroughly in Section~\ref{sect:semiflow} (see Proposition~\ref{prop:semiflow}).

The main results for a DDE of the form \eqref{eq:DDE_classical} are formulated in Section~\ref{sect:statement}.
Here, we give a brief overview of what they comprise. 
As in the Introduction, let us suppose that a stationary state of the semiflow $\varphi$ is attained at $0 \in X$, i.e., $f(0) = 0$.
Let $A$ denote the generator of the semigroup $t \mapsto D \varphi_t(0)$, and let $\Sigma \subset \sigma(A)$ be a bounded spectral subset with spectral projection $P_\Sigma$. 
For these preliminary statements we shall also assume that $f$ is $C^k$ with $k \geq 1$, although it is not strictly necessary (see the footnote to assumption \ref{A1}).

\begin{theorem} \label{thm:intro}
    Suppose that the spectral subset $\Sigma$ is of the form $ \Sigma = \{\lambda \in \sigma(A) \; | \; \mathrm{Re } \, \lambda \geq \gamma\}$ for some $\gamma \in \mathbb{R}$.
    Then, there exists a neighbourhood $U$ of the stationary state $0 \in X$ and a locally invariant (under the semiflow $\varphi$) $C^\ell$ manifold $W^\Sigma \subset U$ (an SSM) tangent to the spectral subspace $ \im(P_\Sigma)$ at $0$.
    Here the integer $1 \leq \ell \leq k$ is determined by the spectral gap condition 
    \begin{displaymath}
        \sup\mathrm{Re} \, \big( \sigma(A) \setminus \Sigma \big) <\ell \inf \mathrm{Re} \, \Sigma.
    \end{displaymath}
    Moreover, $W^\Sigma$ attracts trajectories that remain in $U$ at an exponential rate of $\gamma$, synchronized along the leaves of a $C^0$ foliation.
\end{theorem}

The precise version of this result is Theorem~\ref{thm:main1}.
We remark that the $C^1$ version of Theorem~\ref{thm:intro} is to be expected to hold based on \cite{chen1997invariant} and \cite{diekmann2012delay}. 
In Section~\ref{sect:proofmain1}, we nevertheless carry out the proof in detail.
As required in the proof, we also show the smoothness of the semiflow in the sun-star setting in Appendix~\ref{sect:semiflow_smoothness}, which has apparently been unavailable in the literature.

The following theorem is a less direct consequence of existing results:

\begin{theorem} \label{thm:intro2}
    Suppose  $ \Sigma \subset \{\lambda \in \sigma(A) \; | \; \mathrm{Re } \, \lambda < 0\}$ and that $\Sigma$ satisfies the additional non-resonance conditions specified in \ref{A4}.
    Set $\gamma : = \inf \mathrm{Re} \, \Sigma$ and choose an integer $1 \leq \ell \leq k$ such that 
    \begin{displaymath}
        \sup \{ \lambda \in \sigma(A) \; | \; \mathrm{Re} \, \lambda < \gamma \} < \ell \gamma.
    \end{displaymath}
    Then, there exists a neighbourhood $U$ of $0 \in X$ and a locally invariant (under the semiflow $\varphi$) $C^\ell$ manifold $W^\Sigma \subset U$ (an SSM) tangent to the spectral subspace $ \im(P_\Sigma)$ at $0$.
\end{theorem}

For the precise statement, see Theorem~\ref{thm:main2}. 
The proof constructs the desired manifolds by first constructing the spectral foliation \cite{szalai2020invariant,buza2025smooth} tangent to the spectral subspace corresponding to the spectral subset $\Sigma \cup \{ \lambda \in \sigma(A) \; | \; \mathrm{Re} \, \lambda < \inf \mathrm{Re} \, \Sigma \}$ within the stable manifold of $0 \in X$;
then intersecting the leaf passing through the origin with a manifold produced in Theorem~\ref{thm:intro}.
A result analogous to Theorem~\ref{thm:intro2} holds for  $ \Sigma \subset \{\lambda \in \sigma(A) \; | \; \mathrm{Re } \, \lambda > 0\}$ as well,
whose proof follows from an application of the parameterization method \cite{Cabre2003a} to the unstable manifold under the assumptions listed in Theorem~\ref{thm:main3}.

Our results on inertial manifolds are described in Section~\ref{sect:IM_statement}.
The main advancements were described in the final paragraph of the Introduction; we give a summary statement below.

\begin{theorem}
    Suppose that $f$, the right hand-side of the DDE \eqref{eq:DDE_classical}, is globally Lipschitz, and hence the semiflow $\varphi$ is globally defined.
    Set $L := \lip(f-Df(0))$.
    Suppose that there exists a large enough spectral gap $\nu$ (made precise in Theorem~\ref{thm:IM} and Remark~\ref{remark:IM_altassumption}) along the negative real axis such that
    \begin{equation}
        \frac{L}{\nu} < C
        \label{eq:introcond}
    \end{equation}
    for some  fixed  constant $C > 0$ (see \eqref{eq:nu_cond_remarked}). 
    Then there exists an inertial manifold $W^\Sigma$ tangent to $\im(P_\Sigma)$ at $0 \in X$ (uniquely characterized by the Lyapunov exponents of its trajectories, see \eqref{eq:negative_semiorbit_1}-\eqref{eq:negative_semiorbit_2}), where $\Sigma$ is defined as the subset of $\sigma(A)$ to the right of this spectral gap.
    The manifold $W^\Sigma$ attracts trajectories along a globally defined $C^0$ foliation at an exponential rate faster than $\inf \mathrm{Re} \, \Sigma$.
\end{theorem}

The corresponding precise statement is Theorem~\ref{thm:IM}.
The smoothness and rate of attraction of the  inertial manifold can be improved by means of decreasing the ratio in \eqref{eq:introcond}.
Corollary~\ref{corollary:smalldelays} then asserts that one way to exhibit a small ratio in \eqref{eq:introcond} is via shrinking the delay.

In Section~\ref{sect:examples}, we demonstrate how these results can be applied through three examples.
In Section~\ref{sect:example1}, a nonlinear version of the Cushing equation \cite{cushing1977time} (a particularly simple case of distributed delay) is treated, where we construct inertial manifolds based solely on the spectral gap condition at a moderate delay $h = 1$, thereby establishing IMs with dimension not equal to that of the physical system.
Section~\ref{sect:example2} provides an even simpler example which showcases how an arbitrarily large spectral gap can be achieved by means of shrinking the delay $h$.
The final example (Section~\ref{sect:szaksz}) revisits the equation treated by Szaksz et al.~\cite{szaksz2025spectral}, and shows how inertial manifolds can be obtained in the absence of global control over the nonlinearity, if a bound on the size of the global attractor is known.

The remainder of the paper consists of the proofs of the main theorems in Sections \ref{sect:proofmain1}, \ref{sect:proofmain2} and \ref{sect:IM} with smoothness results deferred to appendices. 
In particular, smoothness of the semiflow is discussed in Appendix~\ref{sect:semiflow_smoothness}; whereas smoothness of invariant manifolds is obtained in Appendix~\ref{sect:mfdsmoothness} (akin to Chapter IX of \cite{diekmann2012delay}).
Some preliminaries are recalled in Appendix~\ref{sect:prelim}.

\section{Delay equations}
\label{sect:DDEs}

In this section we recollect the basics of DDE theory following the book of Diekmann et al.~\cite{diekmann2012delay}.

\subsection{DDE as an evolution equation}

From the perspective of dynamical systems, the most favorable way to handle
\eqref{eq:DDE_classical} is through the lens of the evolution equation it generates on $X$.
To achieve this rigorously in a unified fashion for all DDEs of the form \eqref{eq:DDE_classical}, we follow the approach set out in \cite{diekmann2012delay}  utilizing perturbation theory of dual semigroups.

We begin by considering the simplified problem
\begin{subequations}\label{eq:shift_eq}
\begin{align}
&\dot{x}(t) = 0,  \label{eq:shift_eq1} \\ 
& x_0 = u, \qquad u \in X,  
\end{align}
\end{subequations}
which generates a strongly continuous semigroup\footnote{Appendix~\ref{sect:SGprelim} summarizes the concepts we use here from the theory of semigroups.} $t \mapsto T_0(t)$ called the shift semigroup on $X$, given explicitly as
\begin{equation}
    [T_0(t)u](\theta) = \begin{cases}
        u(t+\theta), \qquad & \text{if } -h \leq t + \theta \leq 0, \\
        u(0), & \text{if } t + \theta \geq 0.
    \end{cases}
    \label{eq:T_0(t)}
\end{equation}
Its generator is a closed densely defined operator $A_0 : \dom(A_0) \to X$, $\dom(A_0) \subset X$, given by
\begin{equation}
    \dom(A_0) = \Big\{ u \in X \; \big| \; \dot{u} \in X, \; \dot{u}(0) = 0 \Big\}, \qquad A_0u = \dot{u}.
    \label{eq:A_0}
\end{equation}
The peculiar detail about \eqref{eq:A_0} is that the update rule \eqref{eq:shift_eq1} only manifests through the definition of the domain $\dom(A_0)$, which would therefore have to be altered for each DDE \eqref{eq:DDE_classical1} on a case-by-case basis.
To circumvent the technical complications this would induce (e.g., in nonlinear problems, in particular when linearizing about a fixed point), \cite{diekmann2012delay} proceeds by exploiting the principles of duality.
The dual space of $X = C\big([-h,0];\mathbb{R}^n \big)$ is 
isomorphic to the space of \textit{normalized bounded variations}\footnote{
For the relevant definitions, see Appendix~\ref{sect:BV}.
The reader might be more used to the interpretation of $X^*$ through the modern version of Riesz's representation theorem,
which assigns to each $x^* \in X^*$ a signed finite $\mathbb{R}^n$-valued Borel measure $\mu$ on $[-h,0]$.
Historically, the literature concerning delay equations has adopted the convention to declare $\zeta: [0,h] \to \mathbb{R}^n$ via $\zeta(t)=\mu([-t,0])$, which constitutes an isomorphic correspondence between $X^*$ and $\nbv\big([0,h];\mathbb{R}^n\big)$. (In fact, Riesz's original result is posed in terms of $\bv$ functions \cite{Riesz1909Sur}.) 
},
$X^*\cong \nbv\big([0,h];\mathbb{R}^n\big)$, via the pairing
\begin{equation}
    \langle \zeta,u \rangle_{X^*,X} = \int_0^h d\zeta(t) u(-t), \qquad \zeta \in \nbv\big([0,h];\mathbb{R}^n\big).
    \label{eq:pairing}
\end{equation}
The integral in \eqref{eq:pairing} is to be interpreted in the Riemann-Stieltjes sense (see Appendix~\ref{sect:BV}), or equivalently, through the measure $\mu$ corresponding to $\zeta$.

The dual semigroup $t \mapsto T_0^*(t):=T_0(t)^*$ need not be strongly continuous; but a theorem of Phillips \cite{phillips1955adjoint,hille1996functional}
asserts that there exists a closed linear subspace $X^{\odot}$ of $X^*$ such that $T_0^*(t)X^{\odot} \subset X^{\odot}$ and $T_0^\odot(t):=T_0^*(t)|_{X^\odot}$ is strongly continuous; let us denote its generator by $A_0^\odot$.
The space $X^{\odot}$ is called the sun subspace of $X$.
Explicitly, $X^\odot = \overline{\dom(A_0^*)}$, where the closure is taken in the norm topology of $X^*$ (see also \eqref{eq:dual_domain}, Appendix~\ref{sect:SGprelim}).
Furthermore, we have\footnote{Here $\lambda$ refers to the Lebesgue measure.}
\begin{equation}
    X^\odot  \cong \mathbb{R}^n \times L^1([0,h],\lambda; \mathbb{R}^n), \qquad | (c,v) |_{X^\odot} = |c| + |v|_{L^1}. 
    \label{eq:Xodot}
\end{equation}
Here $c \in \mathbb{R}^n$ represents the value at $0$ through the isomorphism $ (c,v) \mapsto w$, $w(t) = c + \int_0^t v$ mapping $ \mathbb{R}^n \times L^1([0,h],\lambda; \mathbb{R}^n)$ to $X^\odot$.
We refer to Theorem II.5.2, \cite{diekmann2012delay}, for the details.
It is perhaps more illuminating to interpret \eqref{eq:Xodot} in terms of measures: $X^{\odot }$ consists of Borel measures $\mu$ that can be written as a sum $\mu = \nu + c\delta_0$ for $c \in \mathbb{R}^n$ and $\nu$ absolutely continuous with respect to the Lebesgue measure $\lambda$, $\nu \ll \lambda$.

We may perform the same procedure once more, now for the strongly continuous semigroup $T_0^\odot$ on $X^\odot$.
The dual space of $X^\odot$ can be characterized, using the representation \eqref{eq:Xodot}, as
\begin{equation}
    X^{\odot *} = \mathbb{R}^n \times L^\infty ([-h,0],\lambda; \mathbb{R}^n), \qquad | (b,w) |_{X^{\odot*} } = \sup\{|b|,|w|_{L^\infty}\}
    \label{eq:Xodotstar}
\end{equation}
with norm induced by the pairing 
\begin{displaymath}
    \langle (b,w),(c,v) \rangle_{X^{\odot *},X^{\odot }} = \langle b,c \rangle_{\mathbb{R}^n} + \int_0^h w(-\theta) v(\theta) \,  d \lambda (\theta).
\end{displaymath}
Seeking the sun subspace of \eqref{eq:Xodotstar} on which $T_0^{\odot *}$ is strongly continuous, one arrives at $X^{\odot \odot} = \overline{\mathrm{dom}(A_0^{\odot *})} = \imath(X)$ (Theorem II.5.5 of \cite{diekmann2012delay}), where $\imath : X \xhookrightarrow{} X^{\odot *}$ is the continuous embedding $u \mapsto (u(0),u)$ defined by the pairing between $X$ and $X^\odot$ (in the future, we shall occasionally identify $X$ with its image under $\imath$).
We remark that it is precisely the form of this embedding 
that permits the transition of the update rule \eqref{eq:DDE_classical1} from the domain into the evolution equation below in \eqref{eq:evolution}.
The property $X = X^{\odot \odot}$ is called $\odot$-reflexivity, which can also be equivalently characterized by weak compactness of the resolvent $(\lambda \, \mathrm{id} - A_0)^{-1}$ \cite{de1989characterization}.

We may now write \eqref{eq:DDE_classical} as an evolution equation, exploiting the extra flexibility provided by the enlarged space $X^{\odot *}$. 
We  separate \eqref{eq:DDE_classical} into two parts: one given by the shift semigroup $T_0$ from \eqref{eq:shift_eq} and one by the update rule in \eqref{eq:DDE_classical1}.
This allows one to  view all delay equations of the form \eqref{eq:DDE_classical} as finite rank perturbations of the shift semigroup.
Specifically, the canonical basis $(e_1,\ldots,e_n)$ of $\mathbb{R}^n$ determines $n$ elements
\begin{displaymath}
    r_i^{\odot *} = (e_i,0) \in X^{\odot *}, \qquad 1\leq i \leq n,
\end{displaymath}
which can be used to construct the finite rank nonlinear map $O \to X^{\odot *}$,
\begin{equation}
    F(u) = \sum_{i=1}^n \langle f(u),e_i\rangle_{\mathbb{R}^n} \, r_i^{\odot *}.
    \label{eq:F}
\end{equation}
We arrive at the formal\footnote{Equation \eqref{eq:evolution} is formal because \eqref{eq:evolution1} is only satisfied by solutions of \eqref{eq:DDE_classical} for $t > 0$, for general $u \in O$. If $u \in C^1([-h,0];\mathbb{R}^n)\cap O$, then \eqref{eq:evolution1} is satisfied for all $t \geq 0$, with $d/dt$ interpreted as a right derivative at $t = 0$. 
We shall hence mostly work with \eqref{eq:variationofconstants}, which corresponds to solutions of \eqref{eq:DDE_classical}.}
evolution equation 
\begin{subequations}\label{eq:evolution} 
    \begin{align}
        \imath \frac{d}{dt} x_t &= A_0^{\odot *} \imath x_t + F(x_t),  \label{eq:evolution1}\\
        x_0 &= u \in O,
    \end{align}
\end{subequations}
where, explicitly,
\begin{gather*}
    \dom(A_0^{\odot *}) = \{ (c,u) \; | \; u \text{ is Lipschitz with } u(0) = c \}, \\
    A_0^{\odot *}(c,u) =(0,\dot{u}),
\end{gather*}
where we recall that Lipschitz functions admit derivatives in the $L^\infty$ sense.
Provided $f$ is locally Lipschitz (which implies, through \eqref{eq:F}, that $F$ is locally Lipschitz),
 the variation of constants formula
\begin{equation}
    x_t = T_0(t) u + \imath^{-1} \int_{0}^t T_0^{\odot *} (t-s) F(x_s) \, ds, \qquad u \in O, \; t  \geq 0,
    \label{eq:variationofconstants}
\end{equation}
provides solutions which are in one-to-one correspondence with solutions of \eqref{eq:DDE_classical}, by Proposition VII.6.1 of \cite{diekmann2012delay}.
The integral in \eqref{eq:variationofconstants} is to be interpreted in the weak* sense (see Appendix~\ref{sect:weakstar_int} for details).

\subsection{Properties of semiflows generated by delay equations}
\label{sect:semiflow}

Delay equations, through \eqref{eq:variationofconstants}, generate semiflows -- for the general definitions and notational conventions, we refer the reader to Appendix~\ref{sect:SGprelim}; for notions of smoothness, to Appendix~\ref{sect:notions_of_smoothness}.
We gather the properties of such semiflows in the following proposition.

\begin{proposition}[Results of Sections VII.3, VII.4 and VII.6, \cite{diekmann2012delay}] \label{prop:semiflow}
    Suppose $f$ is locally Lipschitz on $O$.
    Then, there exists a unique maximal semiflow $\varphi:\mathcal{D}^\varphi \to O$ defined on the open domain $\mathcal{D}^\varphi \subset [0,\infty) \times O$ associated to the maps $\varphi_t:u \mapsto x_t$ given by \eqref{eq:variationofconstants}.
    Moreover, the semiflow $\varphi$ satisfies the following:
    \begin{enumerate}[label=\upshape{(\roman*)}]
        \item \label{sfitem:1} $\varphi$ is jointly continuous on $\mathcal{D}^\varphi$. 
        \item \label{sfitem:2} If $f$ is $C^k$, then for each $t \geq 0$ fixed, $\varphi_t : \mathcal{D}^\varphi_t \to O$ is $C^k$. 
        \item If $f$ is $C^1$, then $\varphi$ is jointly $C^1$ on $\{(t,u) \in \mathcal{D}^\varphi \; | \; t > h\}$. 
        \item \label{sfitem:4} If $f$ is bounded, $t \geq h$ and $B \subset O$ is a bounded set, then $\varphi\big( [h,t] \times (B \cup \mathcal{D}^\varphi_t) \big)$ is precompact in $X$. 
    \end{enumerate}
\end{proposition} 

\begin{proof}
    This follows from results of Sections VII.3, VII.4 and VII.6, \cite{diekmann2012delay}, with the exception of \ref{sfitem:2} for $k \geq 2$. 
    The proof of the $C^k$ case, $k \geq 2$, is given in Appendix~\ref{sect:semiflow_smoothness} (in the perturbed setting, described in Section~\ref{sect:linearization}).
\end{proof}

\section{Linearization about a stationary state}
\label{sect:linearization}

Suppose $w \in O$ is a stationary point of \eqref{eq:DDE_classical}, i.e., $\varphi_t(w) =0$ for all $t \geq 0$, which, through the lens of \eqref{eq:evolution}, implies $A_0^{\odot * }\imath w + F(w) = 0$.
Without loss of generality, by relabeling $w$ to $0 \in X$, adjusting $O$ and $\varphi$ accordingly, we may assume that the stationary state is attained at $0$ and that $O$ is a neighbourhood of $0$ (Corollary VII.5.3, \cite{diekmann2012delay}). 
In this setting, we have $F(0) = 0$.
Assume $f$ is $C^1$, which implies $F$ is also $C^1$.
Therefore, 
\begin{equation}
    F(u) =  DF(0) u +R(u), \qquad u \in O,
    \label{eq:Fexpansion}
\end{equation}
where $R:O \to X^{\odot*}$ is a $C^1$ map satisfying $R(v) = o(|v|_X)$ as $|v|_X \to 0$.
In particular, $R(0) = 0$ and $DR(0) =0$.

Hence, the formal evolution equation \eqref{eq:evolution} reads
\begin{align*}
    \imath \frac{d}{dt} x_t &= \big( A_0^{\odot *} \imath + DF(0) \big) x_t + R(x_t),\\
    x_0 &= u \in O.
\end{align*}
The linear part on $X$, given by
\begin{subequations} \label{eq:A}
    \begin{gather}
    \dom(A) = \left\{ u \in \imath^{-1} \big( \dom(A_0^{\odot *}) \big)  \; \big| \;  A_0^{\odot *} \imath u + DF(0) u \in \imath (X) \right\}, \\
    A u= \imath^{-1}\big( A_0^{\odot *} \imath u + DF(0)u \big),
\end{gather}
\end{subequations}
generates a strongly continuous semigroup on $X$, given explicitly as (Theorem III.2.4 and Corollary III.2.13, \cite{diekmann2012delay})
\begin{equation}
    T (t) u= T_0(t)u + \imath^{-1} \int_0^t T^{\odot *}_0 (t-s) DF(0) T(s)u \, ds, \qquad u \in X, \; t \geq 0.
    \label{eq:T(t)}
\end{equation}
Note that $X$ is also $\odot$-reflexive with respect to $T$ and that the $\odot$ and $\odot*$ spaces with respect to $T$ remain the same as those with respect to $T_0$ (this is a consequence of $DF(0)$ being a bounded perturbation).
The full semiflow, as in \eqref{eq:variationofconstants}, can then be given as a nonlinear perturbation of \eqref{eq:T(t)} (Proposition VII.5.4, \cite{diekmann2012delay}),
\begin{equation}
    \varphi_t(u) = T(t) u + \imath^{-1} \int_0^t T^{\odot *} (t-s) R \circ \varphi_s(u) \, ds,  \qquad (t,u) \in \mathcal{D}^\varphi.
    \label{eq:semiflow_pert}
\end{equation}
Note that $T(t) = D\varphi_t(0)$ for all $t \geq 0$ (see also Proposition VII.4.5, \cite{diekmann2012delay}).

\subsection{Spectrum of \texorpdfstring{$A$}{A}}
\label{sect:spectrum}

Since our original equation \eqref{eq:DDE_classical} was posed on the real space $\mathbb{R}^n$ for physical reasons, the operator $A : \dom(A) \to X$ acts on a real Banach space $X = C([-h,0];\mathbb{R}^n)$ -- its spectrum is hence defined via the device of complexification (see Section III.7 of \cite{diekmann2012delay}).
This is done by declaring a complex structure on $X_\mathbb{C} : =X \otimes_\mathbb{R} \mathbb{C}$ via $b(u \otimes c) = u \otimes (bc)$ for $b \in \mathbb{C}$, then equipping $X_\mathbb{C}$ with an admissible norm induced by $| \cdot |_X$ (Definition III.7.5 of \cite{diekmann2012delay}).
The resulting space, termed the complexification $X_\mathbb{C}$ of $X$, can be identified with $X \oplus \mathrm{i} X$, 
on which one may consider complexifications of operators $A$ originally acting on $X$ via
$A_\mathbb{C}(u+\mathrm{i} v) := Au + \mathrm{i} A v$, with $\dom(A_\mathbb{C}) = \dom(A)\oplus \mathrm{i} \, \dom (A)$.
Since all admissible norms are equivalent, this gives rise to a well-defined notion of spectrum via $\sigma(A;X) : = \sigma\big(A_\mathbb{C}; X_\mathbb{C}\big)$.
Note, in particular, that $X_\mathbb{C} \cong C([-h,0];\mathbb{C}^n)$ in our example.
We remark also that $A_\mathbb{C}$ generates $\big(T(t)\big)_\mathbb{C}$, which will be used implicitly in what follows.

The spectrum of delay equations is easiest to characterize via the eventual compactness of the linearization semigroup $T(t)$.  
This property can be inferred as a consequence of Proposition~\ref{prop:semiflow}, applied to the linearized equation 
\begin{subequations}\label{eq:linearized_formaleq}
    \begin{align}
    &\dot{x}(t) = Df(0)x_t,   \\ 
    & x_0 = u \in X.  
    \end{align}
\end{subequations}
Indeed, item \ref{sfitem:4} implies that $T(t)$, as defined in \eqref{eq:T(t)}, is a compact operator for each $t \geq h$ (recall that $Df(0) \in \mathcal{L}(X,\mathbb{R}^n)$). 
The spectral theory of eventually compact semigroups then implies the following.

\begin{theorem}[Theorem B.IV.2.1, \cite{arendt1986one}] \label{thm:spectrum}
    Suppose $A$ is the generator of an eventually compact semigroup $T(t)$ on a Banach space $X$.
    The spectrum of $A$ consists of a countable set, possibly empty or finite,
    of eigenvalues (i.e., $\sigma(A;X) = \sigma_p(A;X) $) of finite algebraic multiplicity. 
    The set $\{ \lambda \in \sigma(A;X) \; | \; \mathrm{Re} \, \lambda \geq \gamma \}$ is finite for all $\gamma \in \mathbb{R}$.
    Moreover, the following spectral mapping property holds
    \begin{equation}
        \sigma\big(T(t);X\big)\setminus \{ 0\} = e^{t \sigma(A;X)}.
        \label{eq:spectralmapping}
    \end{equation}
\end{theorem}

\begin{proof}
    This is just a restatement of Theorem B.IV.2.1, \cite{arendt1986one}, apart from the last statement \eqref{eq:spectralmapping}, which can be inferred from Theorem IV.2.2, \cite{diekmann2012delay}, but we give a proof here as well. 
    The spectral mapping property \eqref{eq:spectralmapping} is a consequence of $\sigma(A;X) = \sigma_p(A;X)$ and the spectral mapping property of point spectra for strongly continuous semigroups (IV.3.7, \cite{engel2000one}), which states
    \begin{displaymath}
        \sigma_p\big(T(t);X\big)\setminus \{ 0\} = e^{t \sigma_p(A;X)}.
    \end{displaymath}
    To conclude \eqref{eq:spectralmapping}, we observe, through eventual compactness of $T$ and the formula $T(t)^n = T(tn)$ (for $n \in \mathbb{N}$ large enough so that $T(nt)$ is compact) that $\sigma(T(t))^n \setminus \{0\}=\sigma(T(t)^n)\setminus\{0\}  = \sigma_p(T(t)^n) \setminus\{0\}  = \sigma_p(T(t))^n \setminus\{0\}$; hence $\sigma(T(t)) \setminus \{0\} = \sigma_p(T(t)) \setminus \{0\}$.
\end{proof}

The following lemma recalls the existence of spectral projections.

\begin{lemma}[Spectral projections] \label{lemma:proj}
    Maintaining the assumptions of Theorem~\ref{thm:spectrum},
    let $\Sigma  $ denote a bounded subset of $\sigma(A;X)$; denote by $\Sigma' = \sigma(A;X) \setminus \Sigma$.
    There exists a projection $P_\Sigma$ on $X_\mathbb{C}$ satisfying the following:\footnote{
    Here and elsewhere, $\imath_\Sigma : X_\Sigma \xhookrightarrow{} X$ denotes the inclusion and $\widetilde{P}_\Sigma : X \to X_\Sigma$ the projection with restricted range.
    We also denote by $P_{\Sigma'} = \mathrm{id}_X - P_\Sigma$.
    }
    \begin{enumerate}[label=\upshape{(\roman*)}]
        \item \label{projlemma_item1} $X_\mathbb{C}$ admits a splitting $X_\mathbb{C} = X_\Sigma \oplus X_{\Sigma'}$ of $A_\mathbb{C}$-invariant closed linear subspaces $X_\Sigma = \im (P_\Sigma)$ and $X_{\Sigma'} = \ker(P_\Sigma)$, with $X_\Sigma$ finite dimensional. 
        \item The operator $A_\Sigma : = \widetilde{P}_\Sigma A_\mathbb{C} \imath_\Sigma:X_\Sigma \to X_\Sigma$ is bounded, its spectrum is $\sigma(A_\Sigma ; X_\Sigma) = \Sigma$.
        \item \label{projlemma_item3}  The unbounded operator $A_{\Sigma'} := \widetilde{P}_{\Sigma'} A_\mathbb{C}|_{X_{\Sigma'} \cap \dom(A_\mathbb{C})} : X_{\Sigma'}\cap \dom(A_\mathbb{C}) \to X_{\Sigma'}$ has spectrum $\sigma(A_{\Sigma'} ; X_{\Sigma'}) = \Sigma'$.
        \item \label{projlemma_item4}  If $\Sigma$ is invariant under complex conjugation, then $P_\Sigma$ preserves the real subspace, i.e., $P_\Sigma X \subset X$, and all points above can be reformulated in terms of '$X$' and '$A$' instead of '$X_\mathbb{C}$' and '$A_\mathbb{C}$'.
    \end{enumerate}
\end{lemma} 

\begin{proof}
    Points \ref{projlemma_item1}-\ref{projlemma_item3} follow directly from holomorphic functional calculus, see e.g.\ Theorem 6.19, \cite{buhler2018functional}.
    Note that in order to conclude finite dimensionality of $X_\Sigma$, we used the first two assertions of Theorem~\ref{thm:spectrum}, in particular, that $\Sigma$ is hence a finite set consisting of eigenvalues of finite algebraic multiplicity.
    
    Point \ref{projlemma_item4} follows from Corollary IV.2.19 of \cite{diekmann2012delay}, which states the result only for a single pair of complex conjugate  eigenvalues, but trivially extends to the present case in view of the observation above, which implies $P_\Sigma = \sum_{\lambda \in \Sigma}P_{\{\lambda\}}$.
\end{proof}

 Since our treatment of delay equations assumes $A$ is real, its eigenvalues appear in complex conjugate pairs, i.e., $\sigma(A;X)= \sigma_p(A;X)$ is invariant under conjugation. 
 Hence, we can always select $\Sigma$ in a way such that $P_\Sigma$ preserves the real subspace by replacing $\Sigma$ with $\Sigma \cup \mathrm{conj}(\Sigma)$.
 In the future, we shall restrict our attention to projections as such.
 
 We shall also assume that $\Sigma' = \sigma(A;X) \setminus \Sigma \neq \emptyset$ for almost the entirety of this paper, with the exception of Remark~\ref{remark:IM_with_F} concerning inertial manifolds.\footnote{To extend the results herein to the case when $\Sigma' = \emptyset$, one should simply interpret $\beta$ in what follows as $-\infty$ and $-\beta+\varepsilon_2$ as some (large) negative number (specifically in Lemma~\ref{lemma:dichotomy}).
 See also Remark~\ref{remark:IM_with_F}.
 The assumption is made here mostly for notational convenience.}
 We remark that this condition does not automatically hold because $\sigma(A;X)$ may be finite in Theorem~\ref{thm:spectrum}.
 However, the only interesting example for which the condition fails is the following in Remark~\ref{remark:A0proj}. 
 It will only be used for Remark~\ref{remark:IM_with_F}, whose purpose is to simplify some assumptions for inertial manifolds.

\begin{remark} \label{remark:A0proj}
    Consider the operator $A_0 : \dom(A_0) \to X$ given in \eqref{eq:A_0}.
    Knowing that $\sigma(A_0;X) = \sigma_p(A_0;X)$ by Theorem~\ref{thm:spectrum}, one can deduce from the eigenvalue equation
    \begin{displaymath}
        A_0 u = \dot{u} = \lambda u, \qquad u \in \dom(A_0),
    \end{displaymath}
    that $\sigma(A_0;X) = \{0\}$ (recall that $\dot{u}(0) = 0$ for $u \in \dom(A_0)$). 
    Moreover, the eigenvalue $0$ has algebraic and geometric multiplicity equal to $n$; the eigenspace $X_{\{0\}} \subset X$ is spanned by basis elements of $\mathbb{R}^n$ interpreted as constant functions over $[-h,0]$. 
    The associated projection $ P_{\{0\}}$ can be identified as follows.
    Since spectral projections commute with the operator that defines them, we must have
    \begin{displaymath}
        P_{\{0\}} A_0 = A_0 P_{\{0\}} , \qquad \text{on } \dom(A_0),
    \end{displaymath}
    which implies that $P_{\{0\}} \dot{u} = 0$ for all $u \in \dom(A_0)$.
    Hence, $P_{\{0\}} u =u(0)$. 
\end{remark}

 Next, we consider the exponential dichotomy that $P_\Sigma$ produces.
 Denote the \textit{growth bound} of a strongly continuous semigroup $t \mapsto T(t)$ with generator $A$ by  
\begin{equation}
    \omega(A) := \inf \left\{ \omega \in \mathbb{R} \; \Big| \; \exists M > 0 \text{ such that } \Vert T(t) \Vert_{\mathcal{L}(X)} \leq M e^{\omega t}, \; \forall t \geq 0 \right\}.
    \label{eq:growthbound}
\end{equation}
Note that this is well defined: $\omega(A) < \infty$ for strongly continuous semigroups.
Recall that for all $t > 0$ we have (see, e.g.,  \cite{van1996asymptotic}, Proposition 1.2.2)
\begin{equation}
    \omega(A) = \frac{\log r_\sigma(T(t))}{t},
    \label{eq:growthbound_spectralrad}
\end{equation}
where $r_\sigma(T(t)) = \sup_{\lambda \in \sigma(T(t))} |\lambda|$ denotes the spectral radius.

 \begin{lemma}[Exponential dichotomy] \label{lemma:dichotomy} 
     Assume the setting of Theorem~\ref{thm:spectrum}, and let $\Sigma$ and $P_\Sigma$ be as in Lemma~\ref{lemma:proj}, with $\Sigma$ invariant under complex conjugation and $\Sigma' \neq \emptyset$.
     Set $\alpha = \inf\{ \mathrm{Re} \, \lambda \; | \; \lambda \in \Sigma \}$ and $\beta = \sup\{ \mathrm{Re} \, \lambda \; | \; \lambda \in  \Sigma' \}$.
     Then, for any $\varepsilon_1,\varepsilon_2>0$, there exists $K_1,K_2 > 0$ such that
     \begin{subequations}\label{eq:dichotomy}
     \begin{align}
         \Vert T(t) P_\Sigma \Vert_{\mathcal{L}(X)} &\leq K_1 e^{(\alpha -\varepsilon_1)t} \Vert P_{\Sigma} \Vert_{\mathcal{L}(X)}, \qquad &t \leq 0,  \label{eq:dichotomy1}\\ 
         \Vert T(t)P_{\Sigma'} \Vert_{\mathcal{L}(X)} &\leq K_2 e^{(\beta+\varepsilon_2) t} \Vert P_{\Sigma'}\Vert_{\mathcal{L}(X)}, \qquad &t \geq 0. \label{eq:dichotomy2}
     \end{align}    
     \end{subequations}
 \end{lemma}

 \begin{proof} 
    It is a consequence of \eqref{eq:generator_sg} that $T(t)$ can be factored into two semigroups according to the decomposition $X = X_\Sigma \oplus X_{\Sigma'}$.\footnote{The subspaces are equipped with the norm inherited from $X$; in this case, $\Vert\widetilde{P}_\Sigma \Vert_{\mathcal{L}(X,X_\Sigma)} = \Vert{P}_\Sigma \Vert_{\mathcal{L}(X)}$.}
    In particular, $A_\Sigma$ generates $\widetilde{P}_\Sigma T(t) \imath_\Sigma : X_\Sigma \to X_\Sigma$; similarly for $A_{\Sigma'}$.
    Since $A_\Sigma$ is bounded, $t \mapsto  \widetilde{P}_\Sigma T(t) \imath_\Sigma$ is norm continuous, hence forming a group (Lemma 7.17, \cite{buhler2018functional}). 
    In particular, $-A_\Sigma$ generates the semigroup $\{ \widetilde{P}_\Sigma T(-t) \imath_\Sigma \}_{t \geq 0}$; hence, by \eqref{eq:growthbound_spectralrad},
    \begin{align*}
        \omega(-A_\Sigma) &= \frac{1}{t}\log r_\sigma( \widetilde{P}_\Sigma T(-t) \imath_\Sigma )  \\
        &= \frac{1}{t} \sup_{\lambda \in \Sigma} \log \left| e^{-t \lambda} \right| \\
        & = -\alpha.
    \end{align*}
    By definition \eqref{eq:growthbound}, for any $\varepsilon_1>0$, there exists $K_1 > 0$ such that
    \begin{displaymath}
        \Vert \widetilde{P}_\Sigma T(-t) \imath_\Sigma \Vert_{\mathcal{L}(X_\Sigma)} \leq K_1 e^{(-\alpha +\varepsilon_1)t}, \qquad t \geq 0.
    \end{displaymath}
    Therefore,
    \begin{align*}
        \Vert T(t) P_\Sigma \Vert_{\mathcal{L}(X)}  &= \Vert \widetilde{P}_\Sigma T(t) \imath_\Sigma \widetilde{P}_\Sigma \Vert_{\mathcal{L}(X)} \\
        &\leq \Vert \widetilde{P}_\Sigma T(-t) \imath_\Sigma  \Vert_{\mathcal{L}(X_\Sigma)}  \Vert{P}_\Sigma \Vert_{\mathcal{L}(X)} \\
        &\leq K_1 e^{(\alpha-\varepsilon_1) t} \Vert{P}_\Sigma \Vert_{\mathcal{L}(X)}, \qquad t \leq 0.
    \end{align*}

    Similar considerations lead to \eqref{eq:dichotomy2}.
 \end{proof}

\subsubsection{Characteristic equation}
\label{sect:characteristiceq}

The spectrum of $A$ is usually computed via the \textit{characteristic equation} associated to the linear system \eqref{eq:linearized_formaleq}.
To derive the characteristic equation, one first 
complexifies \eqref{eq:linearized_formaleq} so that $Df(0) \in \mathcal{L}(X_\mathbb{C},\mathbb{C}^n)$, 
then transforms it into a more manageable form via the Riesz representation theorem, treating $Df(0)$ as $n$ linear functionals on $X_\mathbb{C}$.
In particular, there exists a unique $\zeta \in \nbv([0,h];\mathbb{C}^{n \times n})$ for which
\begin{equation}
    Df(0) u = \int_0^h d\zeta(\theta) u(-\theta) , \qquad u \in X_\mathbb{C},
    \label{eq:RieszDf0}
\end{equation}
where the integral is interpreted in the Riemann-Stieltjes sense (see Appendix~\ref{sect:BV}).

The \textit{characteristic matrix} is defined by
\begin{equation}
    \Delta (z) = z \, \mathrm{id}_{\mathbb{C}^{n \times n}} - \int_0^h e^{-zt} d\zeta(t), \qquad z \in \mathbb{C}, 
    \label{eq:characteristic_matrix}
\end{equation}
the corresponding characteristic equation reads 
\begin{equation}
    \det(\Delta (z)) = 0.
    \label{eq:characteristic_eq}
\end{equation}
Equation \eqref{eq:characteristic_eq} is usually obtained from \eqref{eq:linearized_formaleq} and \eqref{eq:RieszDf0} via a Laplace transform (see Section I.3, \cite{diekmann2012delay}), or by direct substitution of the model solution $c e^{z t}$, $c \in \mathbb{C}^n$.

It is almost immediate from the latter approach that roots of \eqref{eq:characteristic_eq} are contained in $\sigma_p(A)$; one only needs to check that the eigenfunction thus produced, $w : \theta \mapsto c e^{\lambda \theta}$ on $X_\mathbb{C}$ for which $\Delta(\lambda) c = 0$, is contained in $\dom(A_\mathbb{C})$, i.e., $(Df(0)w,\dot{w} )\in \imath(X_\mathbb{C})$.
Thus, we need only verify $Df(0)w = \dot{w}(0)$, which follows from the observation
\begin{displaymath}
    \dot{w}(0) - Df(0)w = \Delta(\lambda) c. 
\end{displaymath}
The reverse inclusion requires a careful assessment of the resolvent operator of $A$
(see Theorem IV.3.1, \cite{diekmann2012delay}).
The overall statement, 
given in Corollary IV.3.3, \cite{diekmann1987perturbed},
is 
\begin{equation}
    \left\{ \lambda \in \mathbb{C} \; \big| \; \det(\Delta(\lambda)) = 0 \right\} = \sigma_p(A;X) = \sigma(A;X).
    \label{eq:charact_equal_spect}
\end{equation}

Due to their significance in determining the asymptotic behaviour of linear systems,
characteristic equations of functional differential equations have received widespread attention throughout the years \cite{stepan1989retarded,michiels2007stability,cooke1963differential}.
In the special case of delay differential equations, the asymptotic stability of linear systems is completely determined by the roots of \eqref{eq:characteristic_eq}, which is a simple consequence of Lemma~\ref{lemma:dichotomy} and \eqref{eq:charact_equal_spect}.
For functional equations of arbitrary type, spectral stability is a necessary but generally not sufficient condition \cite{stepan1989retarded}.
We remark also that, while the asymptotic behaviour is completely characterized by the spectrum in the linear delayed case, there might exist \textit{small solutions} (defined as solutions decaying faster than $e^{at}$ for all $a \in \mathbb{R}$) that are not captured by spectral expansions (see Chapter V, \cite{diekmann2012delay}).

\subsubsection{Asymptotics of the spectrum, barriers to \texorpdfstring{$C^k$}{Ck} smoothness of pseudo-unstable manifolds} 
\label{sect:spectrum_asymp}

For pseudo-unstable manifolds of interest herein, $C^k$ smoothness, $k \geq 2$, of the manifolds hinges on spectral gap conditions of the form
\begin{equation}
    \beta< k \, \alpha < 0, \qquad \alpha = \inf_{\lambda \in \Sigma} \mathrm{Re} \, \lambda, \qquad \beta = \sup_{\lambda \in {\Sigma'}} \mathrm{Re} \, \lambda,
    \label{eq:spectralgap}
\end{equation}
where $\Sigma \subset \sigma(A)$ is a bounded subset of the spectrum and $\Sigma' = \sigma(A) \setminus \Sigma$, as in the statement of Lemma~\ref{lemma:dichotomy}.
In this subsection, we show that \eqref{eq:spectralgap} is violated whenever $|\alpha|$ is large enough, i.e., for all but finitely many choices of $\Sigma$, under the mild assumption $\zeta \in \sbv$ (for the relevant definition, see Appendix~\ref{sect:BV}).

For this, let us recall two results about the roots of \eqref{eq:characteristic_eq}.

\begin{proposition}[Theorem V.7.11, \cite{diekmann2012delay}] \label{prop:spectrum1}
    Suppose $\zeta \in \sbv$.
    There exist constants $C,C_0,C_1 > 0$ such that all roots $\lambda \in \mathbb{C}$ of \eqref{eq:characteristic_eq} with $|\lambda| > C$ must be contained in the strip
    \begin{displaymath}
        \left\{ z \in \mathbb{C}   \; \big| \;  C_0 e^{-h \mathrm{Re} \, z} \leq |z| \leq C_1  e^{-h \mathrm{Re} \, z}   \right\}. 
    \end{displaymath}
\end{proposition}

\begin{proposition}[Theorem 4.12, \cite{lunel1989exponential}] \label{prop:spectrum2}
    The number of roots of \eqref{eq:characteristic_eq} with modulus less than $r$, denoted by $n(r)$, satisfies\footnote{
    In \eqref{eq:spect_asymp}, $E(\det \Delta)$ refers to the exponential type of the function $z \mapsto \det (\Delta(z))$. 
    For the precise definition, see Definition 4.1, \cite{lunel1989exponential}, but for all  purposes herein, one could just interpret it as a positive constant determined by $\zeta$.
    }
    \begin{equation}
        \lim_{r \to \infty} \frac{n(r)}{r} = \frac{E(\det \Delta)}{\pi}.
        \label{eq:spect_asymp}
    \end{equation}
\end{proposition}

\begin{lemma}
    Let $\zeta \in \sbv$ and suppose $\sigma(A)$ has infinite cardinality.
    Then, \eqref{eq:spectralgap}  fails for $k = 2$ (and hence for all $k \geq 2$) for $|\alpha|$ large enough.
\end{lemma}

\begin{proof}
    Supposing the contrary, we may extract a decreasing sequence $(\alpha_j)_{j \in \mathbb{N}} \subset \mathrm{Re} \, \sigma(A)$,  $\alpha_j \to -\infty$, such that 
    \begin{equation}
        \beta_j < 2 \, \alpha_j < 0, \qquad \beta_j := \sup \{ \mathrm{Re} \, \lambda \; | \; \lambda \in \sigma(A), \; \mathrm{Re} \, \lambda < \alpha_j \}.
        \label{eq:contr_hyp_spectrlemma}
    \end{equation}

    Choose $\delta > 0 $ small enough so that $2\delta < -\alpha_1$.
    Set $r_{1,j} : = C_1e^{-(\alpha_j - \delta) h}$ and $r_{2,j} : = C_0e^{-(\beta_j + \delta) h}$, where $C_0$, $C_1$ are as in Proposition~\ref{prop:spectrum1}.
    Then, for $j$ large enough so that $r_{2,j}>r_{1,j} > C$, we have by Proposition~\ref{prop:spectrum1} that
    \begin{displaymath}
        n(r_{2,j}) - n(r_{1,j}) \leq \#  \{ z \in \sigma(A) \; | \; \mathrm{Re} \, z \in ( \beta_j+\delta, \alpha_j - \delta ] \}.
    \end{displaymath}
    By definition of $\beta_j$, we must have $n(r_{2,j}) = n(r_{1,j}) $.
    On the other hand, Proposition~\ref{prop:spectrum2} states that
    \begin{displaymath}
        \left| \frac{n(r)}{r} - \frac{E( \det \Delta)}{\pi} \right| < \varepsilon_j, \qquad \text{for all } r \geq r_{1,j},
    \end{displaymath}
    with $\varepsilon_j \to 0$ as $j \to \infty$.
    Hence, for $j$ sufficiently large as above,
    \begin{multline*}
        n(r_{2,j}) - n(r_{1,j}) > \left( \frac{E (\det \Delta)}{\pi} - \varepsilon_j \right) C_0e^{-(\beta_j + \delta) h} - \left( \frac{E (\det \Delta)}{\pi} + \varepsilon_j \right) C_1e^{-(\alpha_j - \delta) h} \\
        > e^{-(\alpha_j - \delta) h } \left[ \left( \frac{E (\det \Delta)}{\pi} - \varepsilon_j \right) C_0e^{-(\alpha_j + 2\delta) h} - \left( \frac{E (\det \Delta)}{\pi} + \varepsilon_j \right) C_1\right],
    \end{multline*}
    where the last line used \eqref{eq:contr_hyp_spectrlemma}.
    It is clear that the right hand side is positive (in fact, arbitrarily large) for $j$ large enough, which yields the desired contradiction.
\end{proof}

\section{Statement of results}
\label{sect:statement}

In this section, we formulate precisely our main results.
For this, we first group the notational conventions and hypotheses that shall be used throughout.

\begin{enumerate}[label =(A.\arabic*)] 
    \item \label{A1} (Setup) Let $X$ denote the Banach space $ C([-h,0];\mathbb{R}^n)$ equipped with the supremum norm. Let $O \subset X$ be an open open neighbourhood of $0$; and suppose $f : O \to \mathbb{R}^n$ is of class $C^1$.\footnote{For the Lipschitz version of Theorem~\ref{thm:main1}, one could drop the $C^1$ assumption provided the equation is already in the form \eqref{eq:Fexpansion} and $\lip_\rho(R) \to 0$ as $\rho \to 0$ (in the notation of Lemma~\ref{lemma:Lipschitz}).}
    Let $\varphi: \mathcal{D}^\varphi \to O$ denote the semiflow generated by \eqref{eq:DDE_classical} (see Proposition~\ref{prop:semiflow}).
    Suppose $0$ is a stationary state of $\varphi$, and $F$ and $R$ are as in \eqref{eq:Fexpansion}.
    Suppose $\Sigma \subset \sigma(A;X)$ is a nonempty bounded subset of the spectrum invariant under complex conjugation, where $A$ is defined in \eqref{eq:A}, and set $\Sigma' = \sigma(A;X) \setminus \Sigma \neq \emptyset$. Denote the associated projection by $P_\Sigma$, its image by $X_\Sigma$ and its kernel by $X_{\Sigma'}$ (as in Lemma~\ref{lemma:proj}). Set $P_{\Sigma'} = \mathrm{id}_X - P_\Sigma$. 
\end{enumerate}
For statements concerning higher-degree smoothness, we shall utilize the following assumption on $f$.
\begin{enumerate}[label =(A.\arabic*)]
    \setcounter{enumi}{1}
    \item \label{A2} (Regularity) Suppose $f : O \to \mathbb{R}^n$ from \ref{A1} is of class $C^k$, $k \in \mathbb{N}^{\geq 2} \cup \{ \infty \}$.
\end{enumerate}
For the first result, we assume the following form of $\Sigma$.
\begin{enumerate}[label =(A.\arabic*)]
    \setcounter{enumi}{2}
    \item \label{A3} (Spectral subset) Suppose $\Sigma$ from \ref{A1} is of form
    \begin{displaymath}
        \Sigma = \{ \lambda \in \sigma(A;X) \; | \; \mathrm{Re} \, \lambda \geq \gamma \}
    \end{displaymath}
    for some $\gamma \in \mathbb{R}$. Let $\alpha :=\inf_{\lambda \in \Sigma} \mathrm{Re} \, \lambda,$ and $\beta := \sup_{\lambda \in {\Sigma'}} \mathrm{Re} \, \lambda$.
\end{enumerate}

For a given set $U$, let us denote by $\mathcal{U}(U)$ the set of pairs $(t,u)$ for which a trajectory initiated at $u$ remains in $U$ for all times up to $t$, i.e.,
\begin{displaymath}
    \mathcal{U}(U) = \left\{ (t,u) \in \mathcal{D}^\varphi  \; \big| \; \varphi ([0,t] \times \{u\}) \subset U  \right\},
\end{displaymath}
and let $\mathcal{U}_u(U) = \{ t \in \mathcal{D}^\varphi_u \; | \; (t,u) \in \mathcal{U}(U) \}$.

\begin{theorem} \label{thm:main1}
    Assume \ref{A1} and \ref{A3}.
    The following hold:
    \begin{enumerate}[label=\upshape{(\roman*)}]
        \item \label{thm1st1} \textup{(Invariant manifold)} There exists an open neighbourhood $U \subset O$ of $0$ and a $C^1$ submanifold $W^\Sigma \subset U$ tangent to $X_\Sigma$ at $0$, which is locally invariant under $\varphi_t$: If $u \in W^\Sigma$, then $\varphi_t(u) \in W^\Sigma$ for all $(t,u) \in \mathcal{U}(U)$.
        $\{\varphi_t|_{W^\Sigma} \}_{t \geq 0}$ extends to a jointly $C^1$ flow on $\mathcal{U}(U \cap W^\Sigma)$.
        \item \label{thm1st2} \textup{(Smoothness)}  Suppose in addition \ref{A2}. If $\beta < \ell \alpha $ for some integer $\ell \leq k$, then $W^\Sigma$ is a $C^\ell$ submanifold of $U$. (If $\gamma > 0$, then $W^\Sigma$ is of class $C^k$.)
        In either case, the restricted semiflow $\varphi_t|_{W^\Sigma}$ extends to a jointly $C^\ell$ (resp.\ $C^k$) flow on $\mathcal{U}(U \cap W^\Sigma)$.
        \item \label{thm1st3} \textup{(Attractivity)} There exists a continuous map $\pi : U \to W^\Sigma$ such that 
        \begin{equation}
            \pi \circ \varphi_t(u) = \varphi_t\vert_{W^\Sigma} \circ \pi (u), \qquad \text{for all } (t,u) \in \mathcal{U}(U).
            \label{eq:pi}
        \end{equation}
        Moreover, for any $u \in U$, there exists $C > 0$ such that 
        \begin{equation}
            | \varphi_t(u) - \varphi_t \circ \pi (u)|_X \leq C e^{ \gamma t}, \qquad \text{for all } t \in \mathcal{U}_u(U).
            \label{eq:attractivity}
        \end{equation}
        The preimages $\{\pi^{-1}(w)\}_{w \in W^\Sigma}$ are Lipschitz submanifolds of $U$.
        \item \label{thm1st4} \textup{(Pseudo-uniqueness)} If $\gamma > 0$, $W^\Sigma$ is uniquely determined by its tangency to $X_\Sigma$ at $0$. If $\gamma < 0$, $\pi^{-1}(0)$ is uniquely determined by its tangency to $X_{\Sigma'}$ at $0$. Otherwise, if $W^\Sigma$ and $\widetilde{W^\Sigma}$ are two invariant manifolds tangent to $X_\Sigma$ at $0$, there exists a neighbourhood $V \subset O$ on which the reduced dynamics are topologically conjugate, i.e., there exists a homeomorphism $g : W^\Sigma \cap V \to \widetilde{W^\Sigma} \cap V$ such that
        \begin{displaymath}
            \varphi_t \circ g(u) = g\circ \varphi_t (u), \qquad \text{for all }  (t,u) \in \mathcal{U}(W^\Sigma \cap V ).
        \end{displaymath}
    \end{enumerate}
\end{theorem}

\begin{proof}
    See Section~\ref{sect:proofmain1}.
\end{proof}

\begin{remark}
    It is possible to replace \eqref{eq:attractivity} by the stronger statement
    \begin{displaymath}
        | \varphi_t(u) - \varphi_t \circ \pi (u)|_X \leq C e^{ \gamma t} |u-\pi(u)|_X, \qquad (t,u) \in \mathcal{U}(U),
    \end{displaymath}
    where now the constant $C>0$ is independent of $u$. 
    This follows by repeating the proof of Lemma 4.4 in \cite{buza2024spectral} for the case of delay equations, and using its result in place of \eqref{eq:decayrho} in the course of the proof.
\end{remark}

\begin{remark}[Strong-stable/unstable manifolds] \label{remark:strongmfds}
    It is well known that  $W^\Sigma$ is uniquely characterized by  Lyapunov exponents of negative semiorbits extendable for all times if $\gamma > 0$ (see statement \ref{item1:CHT} of Theorem~\ref{thm:CHT}). 
    In this case, $W^\Sigma$ is called a strong unstable manifold; it is a consequence of the above properties that it is of class $C^k$.
    If $\gamma \leq 0$, then $\pi^{-1}(0)$ represents the strong stable manifold, which is again $C^k$ smooth and can be characterized by forward Lyapunov exponents of its trajectories.
    These are known results, but for a (partial) proof of the strong unstable case, see Appendix~\ref{sect:mfdsmoothness}.
\end{remark}

Next, we  generalize  parts of this statement to spectral subsets not of the form \ref{A3}.
The following assumption (replacing \ref{A3}) specifies the form of stable spectral subsets that are admissible.

\begin{enumerate}[label =(A.\arabic*)]
    \setcounter{enumi}{3}
    \item \label{A4} (Nonresonance) Suppose $\Sigma$ is a bounded spectral subset satisfying
    \begin{displaymath}
         \Sigma \subset \{ \lambda \in \sigma(A;X) \; | \; \mathrm{Re} \, \lambda < 0 \}
    \end{displaymath}
    and invariant under complex conjugation. 
    Let 
    \begin{displaymath}
        \widetilde{\Sigma} = \{ \lambda \in \sigma(A;X) \setminus \Sigma \; | \; \inf \mathrm{Re} \, \Sigma \leq \; \mathrm{Re} \, \lambda < 0 \},
    \end{displaymath}
    and $\widetilde{\Sigma}' = \{ \lambda \in \sigma(A;X) \; | \; \mathrm{Re} \, \lambda < 0 \} \setminus \widetilde{\Sigma}$.
    Denote by $r$ the smallest integer for which 
    \begin{displaymath}
        (r+1) \sup \{ \mathrm{Re} \, \lambda \; | \; \lambda \in \sigma(A;X), \; \mathrm{Re} \, \lambda < 0 \} < \inf \{ \mathrm{Re} \, \lambda \; | \; \lambda \in  \widetilde{\Sigma} \}.
    \end{displaymath}
    Suppose that $\Sigma$ satisfies the non-resonance condition
    \begin{displaymath}
   \left( (i-j) \widetilde{\Sigma} +
    j \widetilde{\Sigma}' \right) \cap \widetilde{\Sigma} = \emptyset
    \end{displaymath}
    for all pairs $(j,i)$ of integers such that $2 \leq i \leq r$ and $1 \leq j \leq i$.
\end{enumerate}

\begin{theorem} \label{thm:main2}
    Suppose \ref{A1}, \ref{A2} and \ref{A4} hold such that $k > r$.
        Set $\gamma : = \inf \mathrm{Re} \, \Sigma$ and choose an integer $1 \leq \ell  \leq k$ such that 
    \begin{displaymath}
        \sup \{ \lambda \in \sigma(A;X) \; | \; \mathrm{Re} \, \lambda < \gamma \} < \ell \gamma.
    \end{displaymath}
    There exists an open neighbourhood $U \subset O$ of $0$ and a $C^\ell$  submanifold $W^\Sigma \subset U$ tangent to $X_\Sigma$ at $0$, which is locally invariant under $\varphi_t$: If $u \in W^\Sigma$, then $\varphi_t(u) \in W^\Sigma$ for all $(t,u) \in \mathcal{U}(U)$.
    $\{\varphi_t|_{W^\Sigma} \}_{t \geq 0}$ extends to a jointly $C^\ell$ flow on $\mathcal{U}(U \cap W^\Sigma)$.
\end{theorem}

\begin{proof}
    See Section~\ref{sect:proofmain2}.
\end{proof}

We have a similar result for the case of unstable spectral subsets.

\begin{enumerate}[label =(A.\arabic*)]
    \setcounter{enumi}{4}
    \item \label{A5} (Nonresonance) Suppose $\Sigma$ is a spectral subset satisfying
    \begin{displaymath}
         \Sigma \subset \{ \lambda \in \sigma(A;X) \; | \; \mathrm{Re} \, \lambda > 0 \}
    \end{displaymath}
    and invariant under complex conjugation. 
    Denote by $r$ the smallest integer for which 
    \begin{displaymath}
          (r+1) \inf \{ \mathrm{Re} \, \lambda \; | \; \lambda \in  {\Sigma} \} >  \sup \{ \mathrm{Re} \, \lambda \; | \; \lambda \in \sigma(A;X) \setminus \Sigma \}.
    \end{displaymath}
    Suppose that $\Sigma$ satisfies the (non-resonance) condition
    \begin{displaymath}
   j \Sigma \cap \{  \lambda \in \sigma(A;X) \setminus \Sigma \; | \;  \mathrm{Re} \, \lambda > 0 \} = \emptyset
    \end{displaymath}
    for all integers $j$ such that $2 \leq j \leq r$.
\end{enumerate}

\begin{theorem} \label{thm:main3}
    Suppose \ref{A1}, \ref{A2} and \ref{A5} hold such that $k > r$.
    There exists an open neighbourhood $U \subset O$ of $0$ and a $C^k$ submanifold $W^\Sigma \subset U$ tangent to $X_\Sigma$ at $0$, which is locally invariant under $\varphi_t$: If $u \in W^\Sigma$, then $\varphi_t(u) \in W^\Sigma$ for all $(t,u) \in \mathcal{U}(U)$.
    $\{\varphi_t|_{W^\Sigma} \}_{t \geq 0}$ extends to a jointly $C^k$ flow on $\mathcal{U}(U \cap W^\Sigma)$. 
\end{theorem}

\begin{proof}
    The proof consists of an application of Theorem 1.2 of \cite{Cabre2003a} to the time-one map of the $C^k$ flow obtained upon restricting $\varphi$ to the finite-dimensional unstable manifold of $0$.
    The details would be analogous to the proof of Theorem~\ref{thm:main2} and are hence omitted.
\end{proof}

\subsection{Inertial manifolds}
\label{sect:IM_statement}

This section concerns the existence of inertial manifolds in system \eqref{eq:DDE_classical}. 
In the small delay regime, a positive answer to this question came as early as the 60s, through a series of papers by Yu.\ A.\ Ryabov \cite{ryabov1960application,ryabov1960application2,ryabov1961application,ryabov1963certain}; see also \cite{driver1968ryabov,driver1976linear,chicone2003inertial}.
Here we derive a different condition not based on the delay, but instead on the existence of a large enough spectral gap (Theorem~\ref{thm:IM}).
We do so utilizing the results of \cite{chen1997invariant} used in the proof of Theorem~\ref{thm:main1}.
These not only show the existence of a smooth inertial manifold, but also the existence of a corresponding invariant foliation, whose leaves characterize the approach of trajectories towards the inertial manifold (c.f.\ \eqref{eq:foliationrho}).
We consequently show that large enough spectral gaps exist in systems with sufficiently small delays -- compared to the Lipschitz constant of the nonlinearity -- thus recovering also the classical results (Corollary~\ref{corollary:smalldelays} and Remark~\ref{remark:IM_with_F}).

As a prerequisite, we shall henceforth assume that $f$ is globally Lipschitz on the whole of $X$.
This is not strictly necessary if the existence of a compact global attractor is known a priori, in which case cutoff functions encompassing the attractor can be used to ease this assumption (as in \cite{foias1988inertial}).

The term \textit{inertial manifold} was initially coined by the authors of \cite{foias1988inertial}, who define inertial manifolds as follows.\footnote{Note that Definition~\ref{def:IM} only makes sense for globally defined semiflows -- but this is certainly the case in our setting when $f$ is globally Lipschitz, see Lemma~\ref{lemma:globalexist}.}

\begin{definition}[Inertial manifolds, \cite{foias1988inertial,robinson2001}] \label{def:IM}
    Consider a semiflow $\{\varphi_t\}_{t \geq 0}$ on a Banach space $X$, with domain $\mathcal{D}^\varphi = \mathbb{R}^{\geq 0} \times X$.
    A subset $\mathcal{M} \subset X$ is said to be an inertial manifold for $\varphi_t$ if it has the following three properties
    \begin{enumerate}
        \item $\mathcal{M}$ is a finite-dimensional Lipschitz manifold,
        \item $\mathcal{M}$ is positively invariant, i.e., $\varphi_t(\mathcal{M}) \subset \mathcal{M}$, for all $t \geq 0$,
        \item $\mathcal{M}$ attracts all trajectories exponentially; that is, there exists $\kappa < 0$ (the rate of attraction) such that
        \begin{displaymath}
            \mathrm{dist}_X(\varphi_t(u),\mathcal{M}) \leq C(|u|_X) e^{ \kappa t}, \qquad  u \in X, \; t \geq 0,
        \end{displaymath}
        where $C(|u|_X)$ is a constant that depends only on the norm of $u \in X$.
    \end{enumerate}
\end{definition}

The following theorem states that, for a sufficiently large spectral gap, inertial manifolds exist for delayed equations.
For the precise statement, we fix the following constants.
Let $M > 0$ and $\omega > \omega(A)$, $\omega \neq 0$, be such that (c.f.\ \eqref{eq:growthbound}) $\Vert T(t) \Vert_{\mathcal{L}(X)} \leq M e^{\omega t}$ for all $t \geq 0$.

\begin{theorem} \label{thm:IM}
    Suppose \ref{A1}, \ref{A3} with $\alpha \leq 0$ and that $f$ is globally Lipschitz on the whole of $X$.
    Let $\varepsilon_1,\varepsilon_2 > 0$ be real numbers and let $K_1$ and $K_2$ be the constants from Lemma~\ref{lemma:dichotomy} associated to $\alpha,\beta,\varepsilon_1$ and $\varepsilon_2$.
    If there exists a positive real number
    \begin{equation}
        \nu < \frac{(\alpha-\beta-\varepsilon_1-\varepsilon_2)}{\ln2} 
        \label{eq:nu}
    \end{equation}
    such that
    \begin{equation}
        \lip(R) < \frac{\omega}{2M (e^{\omega /\nu} -1)}
        \label{eq:replacement}
    \end{equation}
    and 
    \begin{equation}
        \left(\sqrt{K_1} + \sqrt{K_2} \right)^2 4M^2  \lip(R) < \omega \frac{e^{(\alpha - \varepsilon_1-\omega)/\nu}    }{e^{\omega /\nu} - 1}
        \label{eq:nu_cond_2}
    \end{equation}
    are satisfied,
    then the conclusions of Theorem~\ref{thm:main1} hold for $U = X$; where the rate of attraction in \eqref{eq:attractivity} is replaced by some value $\kappa \in (\beta,\alpha)$ (we no longer have full control over this), and moreover, assertion~\ref{thm1st4} can be replaced by genuine uniqueness, characterized by the property \eqref{eq:negative_semiorbit_1}-\eqref{eq:negative_semiorbit_2} in Theorem~\ref{thm:CHT}\ref{item1:CHT}. 
    In particular, $W^\Sigma$ is a $C^1$ inertial manifold.
\end{theorem}

\begin{proof}
    See Section~\ref{sect:IM}.
\end{proof}

\begin{remark}[Precise rate of decay] \label{remark:varrho}
    The precise value of the decay rate $\kappa$ guaranteed by Theorem~\ref{thm:IM} can be computed directly as the smaller solution of $\varrho(e^{s/\nu}) \lip(\varphi_{1/\nu}-T(1/\nu)) = 1$ within $s \in (\beta+\varepsilon_2,\alpha-\varepsilon_1)$, with $\varrho$ as in \eqref{eq:varrho}.
    In fact, since the proof of Theorem~\ref{thm:IM} only used $\alpha \leq 0$ to conclude attractivity; it is possible to construct inertial manifolds if $\beta < 0 < \alpha$ provided the above defined $\kappa$ is negative.
    This is certainly the case if $(\sqrt{K_1} e^{(\beta+\varepsilon_2)/\nu} + \sqrt{K_2} e^{(\alpha-\varepsilon_1)/\nu})/(\sqrt{K_1}+\sqrt{K_2}) < 1$, which is obtained upon examining the minimum of the above function in $(\beta,\alpha)$. 
\end{remark}

\begin{remark}
    The smoothness of the inertial manifold can be improved to $C^\ell$, provided that \ref{A2} holds with $k = \ell$ and the two solutions $\gamma_2 \leq \gamma_1$ of $\varrho(e^{s/\nu}) \lip(\varphi_{1/\nu}-T(1/\nu)) = 1$ within $s \in (\beta+\varepsilon_2,\alpha-\varepsilon_1)$ satisfy $\gamma_2 < \ell \gamma_1$.
    (This follows from Appendix~\ref{sect:mfdsmoothness} upon noting that $[\gamma_2,\gamma_1]$ here must be contained in the permissible interval $[\tilde{\eta},\overline{\eta}]$ in \eqref{eq:choice_of_rho} with $\lip(R)$ fixed; or otherwise via a slight alteration of Theorem~5 in \cite{irwin1980new}.) 
\end{remark}

\begin{remark} \label{remark:IM_altassumption}
    Supposing $\nu > |\omega|$, it suffices to check
    \begin{equation}
        \left(\sqrt{K_1} + \sqrt{K_2} \right)^2 8M^2  \lip(R)  e^{(\omega + \varepsilon_1 - \alpha)/|\omega|} < \nu
        \label{eq:nu_cond_remarked}
    \end{equation}
    in place of the condition \eqref{eq:nu_cond_2}.
    This follows by elementary properties of the exponential function ($\mathrm{sign}(x)(e^x-1) \leq 2|x|$ for $|x| < 1$).    
\end{remark}

The next corollary shows that the spectral gap assumption of Theorem~\ref{thm:IM} can be replaced by a smallness assumption on the delay. 
We thus recover, albeit with alternative assumptions, the results of \cite{driver1968ryabov}.
What we are able to conclude is slightly stronger due to the foliation result of \cite{chen1997invariant}.

The proof of Corollary~\ref{corollary:smalldelays} is based on the observation that we can exhibit an arbitrarily large spectral gap for small enough delays.
For this, we trace the dependence of the characteristic matrix $\Delta(z)$, from \eqref{eq:characteristic_matrix}, on $h$.
It will be convenient to use the representation
\begin{equation}
    \det(\Delta (z)) = z^n - \sum_{j=1}^n z^{n-j}\left( \prod_{i=1}^j    \int_0^h e^{-zt} d\eta_{ji}(t) \right),
    \label{eq:detDelta}
\end{equation}
where each $\eta_{ji}$ is a linear combination of entries of $\zeta$.
This follows directly from taking the determinant of \eqref{eq:characteristic_matrix} (Lemma I.4.3, \cite{diekmann2012delay}).
To make the statement more concise, we define\footnote{Here, $\mathrm{TV}$ refers to the total variation defined in \eqref{eq:TV}. In formula \eqref{eq:Q}, the convention $0^0 =1$ is used.}
\begin{equation}
    Q := \sum_{j=1}^n \frac{j}{n} \left( \prod_{i =1}^j \mathrm{TV}(\eta_{ji}) \right)^{\frac{n}{j}} (2(n-j))^{\frac{n}{j}-1}.
    \label{eq:Q}
\end{equation}

\begin{corollary} \label{corollary:smalldelays}
    Assume \ref{A1} and that $f$ is globally Lipschitz on the whole of $X$.
    Set  $\gamma  := -(2Q)^{\frac{1}{n}}-1$\footnote{The result could be stated for arbitrary $\gamma<-(2Q)^\frac{1}{n}$ for which the conditions hold, it is merely fixed here for concreteness and ease of use.} and produce the set $\Sigma$ as in \ref{A3}.
    If there exist $\varepsilon_1, \varepsilon_2 > 0$ such that 
    \begin{equation}
        h < \min \left\{ \frac{1}{r} \ln \left((2Q)^{-\frac{1}{n}} r \right), \frac{1}{|\gamma|} \ln \left((2Q)^{-\frac{1}{n}} |\gamma| \right)\right\}
        \label{eq:h_cond}
    \end{equation}
    with
    \begin{multline}
        r : = \max \left\{ \left(\sqrt{K_1} + \sqrt{K_2} \right)^2 8M^2  \lip(R)  e^{(\omega + \varepsilon_1 - \alpha)/|\omega|} , |\omega|, \frac{\omega}{\ln \left( \frac{\omega}{2M \lip(R)} + 1 \right)} \right\} \ln2 \\ -\gamma+ \varepsilon_1 + \varepsilon_2, \label{eq:r}
    \end{multline}
    ($K_1$, $K_2$ depending on $\varepsilon_1$, $\varepsilon_2$ as in the statement of Lemma~\ref{lemma:dichotomy}) then the conclusions of Theorem~\ref{thm:IM} hold.
\end{corollary}

\begin{proof}
    See Section~\ref{sect:IM}.
\end{proof}

For $\tau \leq h$, let us consider any smooth map $g_\tau : [-h,0] \to [-\tau,0]$ squeezing the interval (e.g., $\theta \mapsto\theta \tau/h $).
Consider $g_\tau^*: C([-\tau,0];\mathbb{R}^n) \to C([-h,0];\mathbb{R}^n)$ given by $g^*_\tau u : = u \circ g_\tau$.
A brief inspection of the characteristic equation, \eqref{eq:characteristic_eq}, shows that if the right-hand side $f$ of \eqref{eq:DDE_classical1} is replaced by $f \circ g^*_\tau$, then the number of roots (including multiplicity) of \eqref{eq:characteristic_eq} within a compact subset $K \subset \mathbb{C}$ is constant in a neighbourhood of a fixed $\tau$ for which there are no roots on $\partial K$ by  Hurwitz's theorem.
The limit $\tau \to 0$ yields the characteristic matrix $\Delta(z) = z \, \mathrm{id} - \zeta(h)$, whose roots correspond to eigenvalues of the non-delayed system.
In this setting, the proof of Corollary~\ref{corollary:smalldelays} constructs the spectral gap between (perturbations of) the set of eigenvalues present in the non-delayed system and the new ones appearing due to the delay.
This is not explicitly apparent from the proof, but for small enough delays it certainly must be the case due to the uniqueness result Theorem 3 of \cite{driver1968ryabov} and the above remarks.
We thus gain explicit understanding of the inertial manifold for small delays: 
it must be tangent at the fixed point to the spectral subspace spanned by the perturbed eigenvalues also present in the non-delayed system. 
This gives further confirmation to the natural intuition that sufficiently small delays are negligible, in the sense that all solutions decay exponentially fast (with exponent depending on the delay -- c.f.\ the role of $r$ in Corollary~\ref{corollary:smalldelays}) to a finite-dimensional invariant manifold of the same dimension and characteristics as the non-delayed system. 
A different, perhaps more direct perspective on negligible delays is given in \cite{kurzweil2006small}.

\begin{remark} \label{remark:IM_with_F}
    If $f(0) = 0$, one could also apply Theorem~\ref{thm:IM} to the fixed point $0$ with $A = A_0$ and $R = F$, i.e., without carrying out the linearization procedure of Section~\ref{sect:linearization} and keeping the linear part within the nonlinearity $F$.
    Indeed, Theorem~\ref{thm:CHT} -- see also \cite{chen1997invariant} -- itself does not assume that the nonlinearity is of higher order. 
    We lose, however, the tangency conclusion of Theorem~\ref{thm:main1}\ref{thm1st1}.
    In this case, our linear operator is $A_0$, whose spectral properties were discussed in Remark~\ref{remark:A0proj}.
    For any $\gamma < 0$, we have $\alpha = 0$, and that $\Sigma' = \emptyset$.
    The projection $P_{\Sigma'} = \mathrm{id}_X - P_{\{0\}}$ is given by $P_{\Sigma'}u = u - u(0)$.
    Lemma~\ref{lemma:dichotomy} continues to hold in this setting, we can in fact deduce this explicitly:
    For any $\beta_2 < 0$ (taking the place of $\beta + \varepsilon_2$), 
    \begin{displaymath}
        \Vert T_0(t) P_{\Sigma'} \Vert_{\mathcal{L}(X)} \leq K_2 e^{\beta_2 t} \Vert P_{\Sigma'} \Vert_{\mathcal{L}(X)}
    \end{displaymath}
    holds trivially since $T_0(t)P_{\Sigma'}u = 0$ for all $t \geq h$ (recalling \eqref{eq:T_0(t)}).
    The only caveat is that $K_2$ thus depends on $\beta_2$ and the delay $h$ through $K_2 e^{\beta_2 h} = 1$.
    The other part of the dichotomy, \eqref{eq:dichotomy1}, remains unaffected and we may safely take $K_1 = 1$ for any $\varepsilon_1 > 0$.
    In this setting, \eqref{eq:nu} translates to $\nu < (-\beta_2-\varepsilon_1)/\ln2$, which we may take as large as we want at the expense of enlarging $K_2$.
    We also note that $M = 1$ and any $\omega > 0$ works for the case of $T_0$ -- let us fix $\omega =1$ and henceforth assume $\beta_2 < -1$.
    The condition \eqref{eq:replacement} can be easily satisfied by the choice of $\beta_2$ alone.
    The other condition, \eqref{eq:nu_cond_remarked}, translates (with $\varepsilon_1 \ll 1$) to
    \begin{equation}
        16 \left( 1 + e^{-\beta_2 h/2}\right)^2  \lip(F) < |\beta_2|,
        \label{eq:h_cond_finale}
    \end{equation}
    which is perhaps more in line with the original  assumptions required for inertial manifolds in, e.g., \cite{driver1968ryabov} and \cite{chicone2003inertial}.
    What we gain on top of classical results is that the rate of attraction $\kappa$ towards the inertial manifold can be controlled via increasing $\beta_2$ and shrinking $h$, that the  $C^\ell$ smoothness of $W^\Sigma$ can be improved via $\beta_2$ and $h$ once more, and that trajectories are confined to a $C^0$ foliation along their approach to the inertial manifold. 
\end{remark}

The precise way in which the rate of decay $\kappa$ can be controlled via Remark~\ref{remark:varrho} is best understood via examples.
We will give a more elaborate description during the course of Example 2 in Section~\ref{sect:example2}. 

We also remark that, by squeezing the delay via $g_\tau$ as above, \eqref{eq:h_cond_finale} can always be achieved, since $\lip(F \circ g^*_\tau) = \lip(F)$.

\section{Examples}
\label{sect:examples}

The local invariance property of $W^\Sigma$ (Theorem~\ref{thm:main1}\ref{thm1st1}) can be alternatively characterized as 
\begin{equation}
    \varphi_t \circ K (u_1)= K \circ \psi_t (u_1), \qquad u_1 \in V \text{ such that } (t,K(u_1)) \in \mathcal{U}(U),
    \label{eq:invariance_examples}
\end{equation}
where $K : V \to X$ (with $V \subset X_\Sigma$ open, containing $0$) is the embedding with image $W^\Sigma$ and $\psi_t$ is the $C^\ell$ flow representing the reduced dynamics.
In particular, $\psi$ is generated by a $C^{\ell-1}$ vector field $H$.

In practice, one seeks the desired manifold and reduced dynamics in the form of polynomial expressions for $K$ and $H$.
For the purposes herein, we shall content ourselves with third order reduced dynamics -- this requires the manifold to be computed to second order.
The explicit computation of the polynomial coefficients
has been carried out in \cite{diekmann2012delay}, Chapter X, in the context of the Hopf bifurcation on a center manifold,
and more recently in \cite{szaksz2025spectral} in the context of spectral submanifolds.
Here, we give a brief recollection of the results in the present nomenclature.
(The projections $P_{\Sigma}^{\odot *}$ $P_{\Sigma'}^{\odot *}$ are formally introduced in Appendix~\ref{sect:mfdsmoothness}.)

\begin{lemma} \label{lemma:expansions}
    Suppose Theorem~\ref{thm:main1} holds with $W^\Sigma$ $C^\ell$, $\ell \geq 4$.
    Then,
    \begin{align*}
        &DK(0) = \imath_\Sigma, \\
        &D^2K(0) = \imath^{-1} \int_{0}^\infty T^{\odot *}(s) P_{\Sigma'}^{\odot *} D^2R(0)[T(-s)\imath_\Sigma \cdot, T(-s)\imath_\Sigma \cdot] \, ds,
    \end{align*}
    where $\imath_\Sigma : X_\Sigma \xhookrightarrow{} X$ is the inclusion.
    Moreover, $H = A_\Sigma +  \widetilde{P}_\Sigma^{\odot *} R \circ K$ and
    \begin{align*}
        &DH(0) = A_\Sigma, \\
        & D^2H(0) = \widetilde{P}_\Sigma^{\odot *} D^2R(0)[DK(0),DK(0)], \\
        & D^3H(0) = \widetilde{P}_\Sigma^{\odot *} D^3R(0)[DK(0),DK(0),DK(0)] + 2  \widetilde{P}_\Sigma^{\odot *} D^2R(0)[D^2K(0),DK(0)].
    \end{align*}
\end{lemma}

\begin{proof}
    See Appendix~\ref{sect:expansionslemmaproof}.
\end{proof}

In what follows, we use $(\cdot)^{\leq}$ to denote the $3$-jet of a map, e.g., $H^\leq (u)= DH(0)u + \frac12 D^2H(0)[u,u] + \frac{1}{6}D^3H(0)[u,u,u]$.

We remark that the form of $K$, $K = (\mathrm{id},\phi)$, is of no coincidence -- our method of proof assumed that the manifold can be written as a graph of some function $\phi: X_\Sigma \to X_{\Sigma'}$. 
In the language of \cite{haro2016parameterization}, this is the graph style of parameterization.
The parameterization, however, is not unique (even if the manifold $W^\Sigma$ is) -- there is a degree of freedom given by the invariance of \eqref{eq:invariance} under $K \mapsto K \circ \Theta$ and $\psi_t \mapsto \Theta^{-1} \circ \psi_t \circ \Theta$ (hence $H \mapsto D \Theta^{-1}[H \circ \Theta]$) for any diffeomorphism $\Theta: M \to X_\Sigma$ for a manifold $M$ modeled on $X_\Sigma$.
We shall make use of this degree of freedom in the examples, so as to put the reduced dynamics in a more manageable and familiar from.

The formula for $D^2K(0)$ can be significantly simplified if $\Sigma$ consists of algebraically simple eigenvalues.
If $w_1,\ldots,w_m$ is a basis for $X_\Sigma$ with respect to which $A_\Sigma$ is diagonal\footnote{This can only be achieved over the complexification of $X_\Sigma$ and $A_\Sigma$; we omit this from the notation, otherwise it would get too cluttered. Once the full operator $D^2K(0)$ is reconstructed, it will of course preserve the real subspace.} and we denote by $\lambda_i$ the eigenvalue corresponding to $w_i$, then for $1 \leq i,j \leq m$,
\begin{displaymath}
    D^2K(0)[w_i,w_j] = \imath^{-1} \int_{0}^\infty e^{- (\lambda_i + \lambda_j)s} T^{\odot *}(s) P_{\Sigma'}^{\odot *}  D^2R(0)[w_i,w_j] \, ds  .
\end{displaymath}
By the assumption on $\ell$, $ \mathrm{Re} \, (\lambda_i+\lambda_j) > \beta$.
Therefore,
\begin{equation}
    \int_{0}^\infty e^{- (\lambda_i + \lambda_j)s} T^{\odot *}(s) P_{\Sigma'}^{\odot *} \, ds = \imath^{-1} \big((\lambda_i+\lambda_j) \, \mathrm{id} - A^{\odot *} \big)^{-1} P_{\Sigma'}^{\odot *}.
    \label{eq:resolvent_identity}
\end{equation}
For strongly continuous semigroups,  \eqref{eq:resolvent_identity} is well known and is usually termed the resolvent identity (see, e.g., Lemma 7.23 of \cite{buhler2018functional}).
It continues to hold in the present setting due to the properties $A^{\odot *}$ and $T^{\odot *}$ satisfy (in particular, that $\Vert T^{\odot *}(t) P_{\Sigma'}^{\odot *} \Vert \leq Me^{\omega t}$ for  some $\omega > \beta$, $M>0$ and all $t \geq 0$, see Lemma~\ref{lemma:appendixdichotomy}), as can be directly checked by going through the proofs of Lemmas 7.13 and 7.23 in \cite{buhler2018functional}. 
See also Exercise IX.2.8, \cite{diekmann2012delay}.
We hence arrive at the formula 
\begin{displaymath}
    D^2K(0)[w_i,w_j] =P_{\Sigma'} \imath^{-1} \big((\lambda_i+\lambda_j) \, \mathrm{id} - A^{\odot *} \big)^{-1} D^2R(0)[w_i,w_j],
\end{displaymath}
from which the full operator $D^2K(0)$ can be reconstructed.

We recall from \cite{diekmann2012delay}, IV.(3.3), the following formula for a projection $P_{\{\lambda\}}$ associated to a simple eigenvalue $\lambda$: 
\begin{equation}
    [P_{\{\lambda\}} u](\theta) = e^{\lambda \theta} \xi(\lambda) \left[ u(0) + \int_0^h d\zeta(s) \int_0^s e^{-\lambda \tau} u(\tau-s) \, d\tau \right] ,
    \label{eq:P_lambda}
\end{equation}
where $\xi(z) : = \Delta(z)^{-1} (z-\lambda)$.
The obvious modifications -- replacing $u(0)$ above with the first component of an element of $X^{\odot *}$ -- give the extended projection $P^{\odot *}_{\{\lambda\}}$.
We will also need the formula for $(z \, \mathrm{id} - A^{\odot *})^{-1}$, but only the part pertaining to elements of the form $(c,0) \in X^{\odot *}$.
This is given in Corollary IV.5.4, \cite{diekmann2012delay}, as 
\begin{displaymath}
    (z \, \mathrm{id} - A^{\odot *})^{-1} (c,0) = ( \Delta(z)^{-1}c,\theta \mapsto e^{z\theta} \Delta(z)^{-1}c).
\end{displaymath}

Combining these together, one obtains the explicit formula
\begin{multline}
    D^2K(0)[w_i,w_j] = \left[ e^{(\lambda_i + \lambda_j) [\cdot]} - \sum_{\lambda \in \Sigma} e^{\lambda[ \cdot]} \xi(\lambda) \left( 1 + \int_0^h d\zeta (s) \frac{e^{-\lambda s}-e^{-(\lambda_i + \lambda_j) s}}{\lambda_i + \lambda_j - \lambda} \right) \right] \\ \times \Delta(\lambda_i + \lambda_j)^{-1} \mathrm{proj}_1 \big(D^2R(0)[w_i,w_j] \big),
    \label{eq:D2K}
\end{multline}
where $\mathrm{proj}_1:X^{\odot *} \to \mathbb{R}^n$ is the projection onto the first factor.
Higher order terms can be computed analogously -- for instance, if $D^2K(0) = 0$, the coefficient $D^3K(0)$ is obtained by replacing $D^2R(0)$ by $D^3R(0)$ in the above and using three eigenvectors/eigenvalues.

We remark if $\ell$ does not meet the conditions of Lemma~\ref{lemma:expansions}, one may still compute the right-hand sides of the coefficients given therein, so long as $R$ is sufficiently smooth and 'nonresonance' conditions are met (so that $\lambda_i+\lambda_j \notin \Sigma'$ in the above).
However, one may no longer associate them to Taylor coefficients of $K$ and $H$.

\subsection{The Cushing equation}
\label{sect:example1}

We consider the system (with its linear part taken from \cite{insperger2011semi}, originally from \cite{cushing1977time})
\begin{equation}
    \dot{x}(t) = b \int_0^hx(t-\theta) \, d\theta  + a \big(  x(t) - \sin (x(t)) \big)
    \label{eq:cushing}
\end{equation}
on $\mathbb{R}$, for some $ b \in \mathbb{R}$ and $a > 0$. 
Equation \eqref{eq:cushing} admits an equilibrium state at $0$.
Throughout, we shall fix the delay $h = 1$.
In the setting of Section~\ref{sect:linearization},  we have $\zeta (t)= b t$, $t \in [0,h]$, and
\begin{displaymath}
    R(u) = \big(au(0) -a\sin u(0) ,0 \big).
\end{displaymath}
We have immediately that $\lip(R) = 2a$.

The characteristic matrix \eqref{eq:characteristic_matrix} reads
\begin{equation}
    \Delta(z) = z - b \frac{1 - e^{-z h}}{z}.
    \label{eq:charact_cushing}
\end{equation}
Note that $\lim_{z \to 0} \Delta(z) = -bh$, that is,  roots can only cross the imaginary axis away from the origin, so long as $b \neq 0$.
It can be shown that \eqref{eq:cushing} is linearly stable for $b \in (-\pi^2/(2h^2),0)$ (see Section~2.1.2 of \cite{insperger2011semi} or plug in $z = \mathrm{i} \varpi$, $\varpi \in \mathbb{R} \setminus\{0\}$, into \eqref{eq:charact_cushing} and seek the smallest-in-modulus solution for $b$).

\begin{figure}
     \centering
     \begin{subfigure}[b]{0.49\textwidth}
         \centering
         \includegraphics[width=\textwidth]{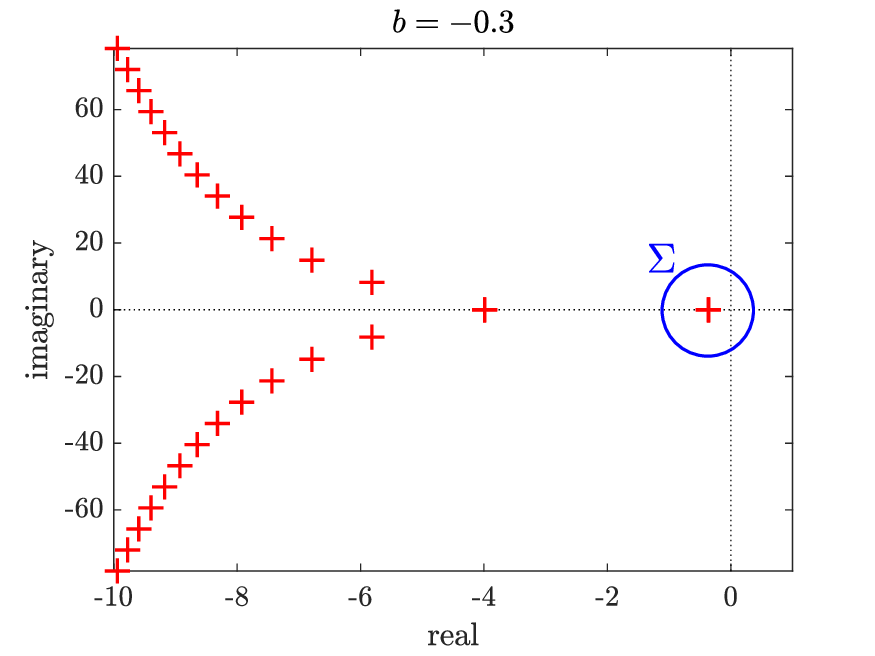}
     \end{subfigure}
     \hfill
     \begin{subfigure}[b]{0.49\textwidth}
         \centering
         \includegraphics[width=\textwidth]{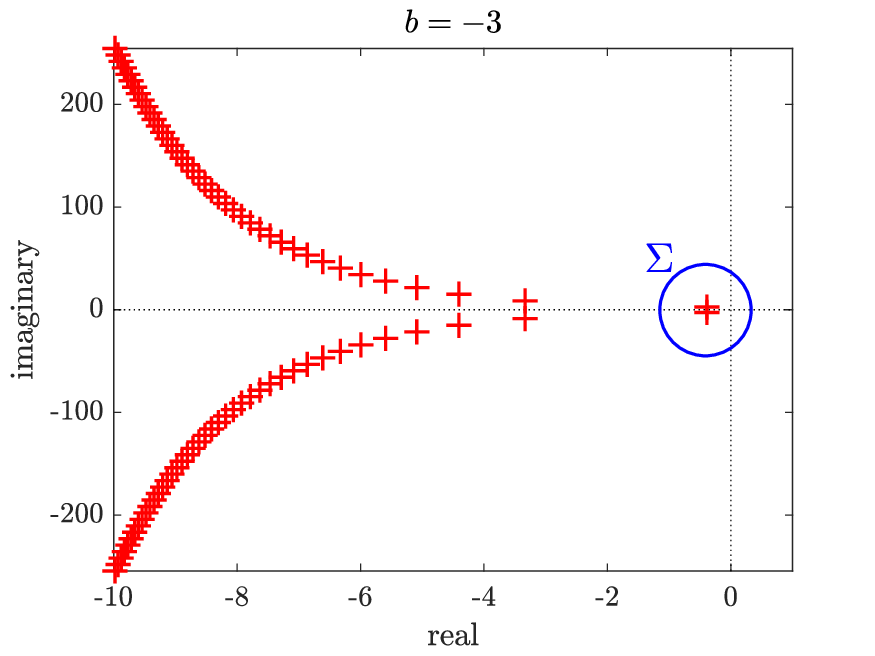}
     \end{subfigure}
        \caption{Roots of \eqref{eq:charact_cushing} in the complex plane.}
        \label{fig:spect_cushing}
\end{figure}

We perform SSM computations at two values of $b$, $b = -0.3$ and $b = -3$ within the stable domain.
As a first step, we need to compute the corresponding spectra, which is shown in Figure~\ref{fig:spect_cushing}.
For later reference, we note here the values of the slowest eigenvalues:
\begin{align*}
     &\lambda_1 \approx -0.361, & &\lambda_2 \approx -3.99, &  &\text{for }b = -0.3;  \\
    &\lambda_{1,2} \approx -0.387 \pm \mathrm{i} 2.66, & &\lambda_{3,4} \approx -3.33 \pm \mathrm{i} 8.67 , & &\text{for }b = -3.
\end{align*}
We take $\Sigma = \{ \lambda_1\}$ for the $b = -0.3$ case and $\Sigma = \{\lambda_1,\lambda_2\}$ for the $b = -3$ case.
The values above determine $\alpha = -0.361$ and $\beta = -3.99$ for $b = -0.3$; and $\alpha = -0.387$ and $\beta = -3.33$ for $b = -3$.
Our manifold $W^\Sigma$ will thus have sufficient smoothness, $\ell \geq 4$, to recover third order dynamics (Lemma~\ref{lemma:expansions}).

Next we show under what conditions on $a$ there exist inertial manifolds, which will be uniquely determined by their dimension in this example.
An inspection of \eqref{eq:charact_cushing} shows that the equations $\Delta(z) =0$ and $\Delta'(z) = 0$, when considered simultaneously, admit no solutions if $b = -0.3$ or $b = -3$ (and $h = 1$). 
This implies, that all roots of $\Delta$ are simple. 
In turn,
Corollary V.6.4 of \cite{diekmann2012delay} implies that the solution to the linear part of \eqref{eq:cushing} can be written as the convergent series 
$\sum_j e^{\lambda_j t} c_j$ for some $c_j \in \mathbb{C}$.
Thus, for any $\varepsilon_1,\varepsilon_2 > 0$ and $\omega > \alpha$, we may take $K_1 = K_2 = M = 1$ in the notation of Theorem~\ref{thm:IM}.
Therefore, using that $\lip(R)  = 2a$, condition \eqref{eq:nu_cond_remarked} boils down to the existence of $\omega >  \alpha$ for which
\begin{displaymath}
    \max \left\{ 64a \ln 2 e^{(\omega - \alpha)/|\omega|}, \ln 2 |\omega| \right\} < \alpha - \beta
\end{displaymath}
holds.
Since, in both cases, $\ln 2|\alpha| < \alpha - \beta$, we may take $\omega$  sufficiently close to $\alpha$, so that the only remaining condition is
\begin{displaymath}
    a < \frac{\alpha - \beta}{64 \ln 2}.
\end{displaymath} 

Next, we compute the reduced dynamics.
We first note that $R = g \circ \mathrm{ev_0}$ for $g: \mathbb{R} \to X^{\odot{*}}$, $y \mapsto (a\sin(y) - ay,0)$, and $\mathrm{ev}_0: u \mapsto u(0)$ the (bounded, linear) evaluation map $X \to \mathbb{R}$.
Hence,
\begin{displaymath}
    D^jR(0) = D^jg(0) [\mathrm{ev}_0 (\cdot),\ldots,\mathrm{ev}_0(\cdot)].
\end{displaymath}
We have that $D^2R(0) = 0$, which implies by $D^2H(0)$ and $D^2K(0) = 0$ by Lemma~\ref{lemma:expansions}.
The third coefficient is simply $D^3R(0) = (a[\mathrm{ev}_0(\cdot)]^3,0)$.

Consider the case $b = -0.3$.
Then $A_\Sigma = \lambda_1$; what remains to compute is $D^3H(0)$.
By \eqref{eq:P_lambda}, we have
\begin{displaymath}
    D^3H(0) = \theta \mapsto  e^{\lambda_1 \theta} \xi(\lambda_1)  a[\mathrm{ev}_0(\cdot)]^3.
\end{displaymath}
The third order dynamics are therefore given by
\begin{equation}
    \frac{d}{dt} w = H^\leq(w) =\lambda_1w + \frac16 (\theta \mapsto  e^{\lambda_1 \theta} \xi(\lambda_1)  a[\mathrm{ev}_0(w)]^3), \qquad w \in X_\Sigma.
    \label{eq:reddyn_w}
\end{equation}
In the present setting, $X_\Sigma$ is a one-dimensional vector space spanned by $\theta \mapsto e^{\lambda_1 \theta}$ (see Section~\ref{sect:characteristiceq}).
We may use the linear isomorphism $\Theta:\mathbb{R} \to X_\Sigma$,  $\Theta(y) = (\theta \mapsto ye^{\lambda_1 \theta}) $ to pull \eqref{eq:reddyn_w} back to $\mathbb{R}$.
The desired dynamics are $\widetilde{H} := D\Theta^{-1} [H^\leq \circ \Theta] = \Theta^{-1} \circ H^\leq \circ \Theta$, or explicitly,
\begin{equation}
    \frac{d}{dt} y =  \widetilde{H}(y) = \lambda_1 y + \frac16  \xi(\lambda_1)  a y^3.
    \label{eq:reddyn_y}
\end{equation}
The coefficient $\xi(\lambda_1)$ can be explicitly computed via L'Hôpital's rule as
\begin{displaymath}
    \xi(\lambda_1) = \lim_{z \to \lambda_1} \frac{z - \lambda_1}{z - b \frac{1 - e^{-z h}}{z}} = \frac{1}{1+\frac{b}{\lambda_1^2}(1-e^{-\lambda_1h}) - \frac{bh}{\lambda_1}e^{-\lambda_1h}}.
\end{displaymath}

We now turn to $b = -3$, the case in which $\Sigma$ comprises a pair of a complex conjugate eigenvalues.
The relevant projection is $\widetilde{P}^{\odot *}_\Sigma = \widetilde{P}^{\odot *}_{\{\lambda_1\}}+\widetilde{P}^{\odot *}_{\{\lambda_2\}}$, whose image defines the (two dimensional, real) subspace 
\begin{equation}
    X_\Sigma = \left\{ \theta \mapsto ze^{\lambda_1 \theta} + \bar{z}  e^{\lambda_2 \theta} \in X \; \big| \;  z \in \mathbb{C} \right\}. 
    \label{eq:Xsigma_z}
\end{equation}
An alternative characterization, via $z = x+\mathrm{i}y$, is
\begin{equation}
    X_\Sigma = \left\{ \theta \mapsto 2 e^{\mathrm{Re} \, \lambda_1 \theta} \big(x \cos (\mathrm{Im} \, \lambda_1 \theta) -y\sin (\mathrm{Im} \, \lambda_1 \theta)  \big)  \in X \; \big| \;  (x,y)\in \mathbb{R}^2 \right\}. 
    \label{eq:Xsigma2}
\end{equation}
The third order dynamics on $X_\Sigma$ are given by
\begin{displaymath}
    \frac{d}{dt} w = H^\leq(w) =A_\Sigma w + \frac16 \big(\theta \mapsto  (e^{\lambda_1 \theta} \xi(\lambda_1) +e^{\lambda_2 \theta} \xi(\lambda_2)) a[\mathrm{ev}_0(w)]^3\big), \qquad w \in X_\Sigma.
\end{displaymath}
The second characterization of $X_\Sigma$, \eqref{eq:Xsigma2}, suggests the parameterization $\Theta: \mathbb{R}^2 \to X_\Sigma$, $\Theta(x,y) =\theta \mapsto 2 e^{\mathrm{Re} \, \lambda_1 \theta} \big(x \cos (\mathrm{Im} \, \lambda_1 \theta) -y\sin (\mathrm{Im} \, \lambda_1 \theta)  \big) $.
By linearity once more, $\widetilde{H} = \Theta^{-1} \circ H^\leq \circ \Theta$, resulting in
\begin{equation}
    \begin{pmatrix}
        \dot{x} \\ \dot{y}
    \end{pmatrix}
    =
    \begin{pmatrix}
    \mathrm{Re} \, \lambda_1 & - \mathrm{Im} \, \lambda_1 \\
    \mathrm{Im} \, \lambda_1 & \mathrm{Re} \, \lambda_1
    \end{pmatrix}
    \begin{pmatrix}
        x \\ y
    \end{pmatrix}
    +
    \frac43 x^3 a 
    \begin{pmatrix}
        \mathrm{Re} \, \xi(\lambda_1) \\ \mathrm{Im} \, \xi(\lambda_1)
    \end{pmatrix}.
    \label{eq:R2dynamics}
\end{equation}
It is perhaps more suitable to parameterize dynamics characterized by complex conjugate pairs of eigenvalues on the half-cylinder $\mathbb{R}^{\geq 0} \times \mathbb{S}$, where $\mathbb{S} = \mathbb{R}/2\pi \mathbb{Z}$, corresponding to polar coordinates.
This can be achieved via $\Theta : \mathbb{R}^{\geq 0} \times \mathbb{S} \to X_\Sigma$ given by $\Theta(r,\phi) = \theta \mapsto re^{\mathrm{i} \phi}e^{\lambda_1 \theta} +re^{-\mathrm{i} \phi}  e^{\lambda_2 \theta}$. 
(Note that $\Theta$ is a diffeomorphism on $\{r > 0 \}$ but not at the origin, as the preimage of $\{0\}$ is $\{0\} \times \mathbb{S}$).
The reduced dynamics would then be obtained as $\widetilde{H} = D\Theta^{-1} [ H^\leq \circ \Theta]$ -- we will only carry out the explicit computation for the following example, however.

\subsection{Small delay inertial manifolds for a one-dimensional example}
\label{sect:example2}

Consider now
\begin{equation}
    \dot{x}(t) =-x(t-h)  + x(t) - \sin x(t).
    \label{eq:example2}
\end{equation}
We examine the vicinity of the steady state at $x=0$.
As before, we consider the $x(t)$ term as part of the nonlinearity, so as to make it third order.
The linear part hence agrees with the problem treated in \cite{szaksz2025spectral}.

In the setting of Section~\ref{sect:linearization},  
\begin{displaymath}
    \zeta(\theta) = \begin{cases}
                0, & \text{for } \theta \in [0,h)\\
        -1, & \text{for } \theta = h
    \end{cases}
\end{displaymath}
and $R(u) = \big(u(0) -\sin u(0) ,0 \big)$, thus $\lip(R) = 2$.
The characteristic equation takes the simple form
\begin{equation}
    \Delta(z) = z +e^{-zh} = 0.
    \label{eq:charact_example2}
\end{equation}
For $h < 1/e$, \eqref{eq:charact_example2} admits two real solutions $\lambda_1(h),\lambda_2(h) \in \mathbb{R}$, $\lambda_2(h) < \lambda_1(h)$.
Here $\lambda_1(h)$ corresponds to the eigenvalue that perturbs from  $\lambda_1(0) = -1$; whereas $\lim_{h\to 0}\lambda_2(h) = - \infty$ (see also Figure~\ref{fig:dumb}). 
All other roots must be complex and have real part less than $\lambda_2(h)$ (clearly there is no root with $\mathrm{Re} \, z \in (\lambda_2(h),\lambda_1(h))$ -- now the continuity argument preceding Remark~\ref{remark:IM_with_F} along with Theorem~\ref{thm:spectrum} implies the desired conclusion). 

\begin{figure}
     \centering
     \includegraphics[width=0.5\textwidth]{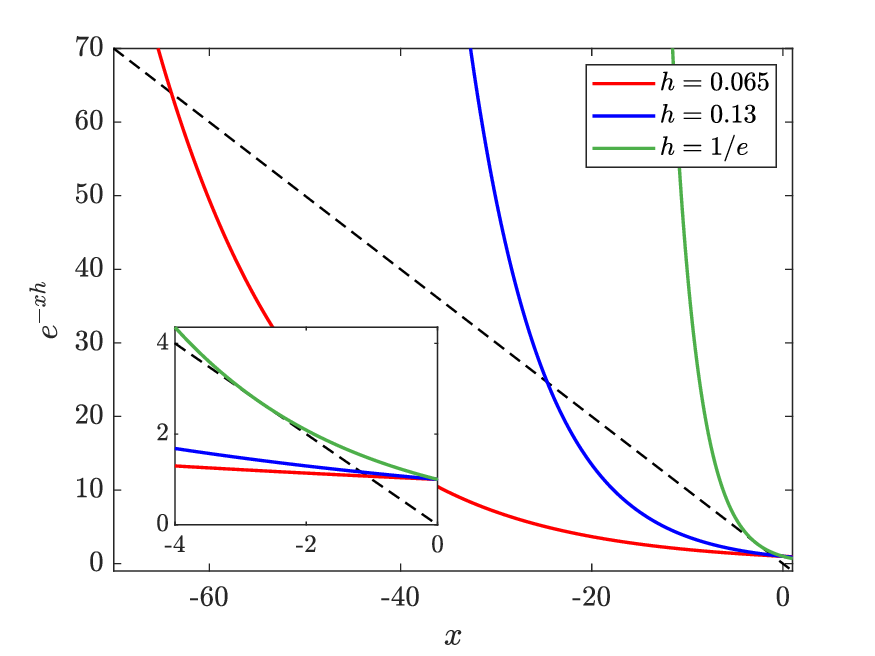}
     \caption{The real roots of \eqref{eq:charact_example2}, obtained as the intersection of $x \mapsto e^{-xh}$ and $x \mapsto -x$ for three different values of $h$. The black dashed line shows $x \mapsto -x$.}
     \label{fig:dumb}
\end{figure}

Next, we derive conditions on the delay $h$ under which there exists a unique inertial manifold associated to $\Sigma = \{ \lambda_1(h)\}$.
Via the same argument as in Example~\ref{sect:example1}, we may  take for the constants appearing in Theorem~\ref{thm:IM} $M=K_1=K_2 = 1$ for any $\varepsilon_1,\varepsilon_2 > 0$ and $\omega > \lambda_1(h)$.
The constant $Q$ from \eqref{eq:Q} is simply $Q = \mathrm{TV}(\zeta) = 1$, hence $\gamma = -3$ in the statement of Corollary~\ref{corollary:smalldelays}.
We compute $r$ from \eqref{eq:r} as 
\begin{displaymath}
    r= 3+ \ln 2 \max \left\{ 64 e^{(\omega-\lambda_1(h))/|\omega|},|\omega|, \frac{\omega}{\ln(\omega/4+1)} \right\}.
\end{displaymath}
Noting that $\lambda_1(h) \in (-5,-1)$ for all $h \in (0,1/e)$ (see Figure~3 of \cite{szaksz2025spectral}), we may safely take $\omega$ arbitrarily close to $\lambda_1(h)$ (for each $h$) and still have the first term dominate -- we may hence use $r = 3 + 64 \ln2$.
Plugging this in \eqref{eq:h_cond},
Corollary~\ref{corollary:smalldelays} now asserts, for $h < 0.066$, the existence of a unique inertial manifold $W^\Sigma$ tangent to the space $X_\Sigma = \{ \theta \mapsto ce^{\lambda_1(h) \theta} \in X \; | \; c \in \mathbb{R} \}$ at the origin.

Next, we estimate (from above) the rate of attraction $\kappa$.
We do this via the initial version of the inertial manifold result, Theorem~\ref{thm:IM}, more specifically, via the route proposed in Remark~\ref{remark:varrho}.
For a given $h < 1/e$, we may compute the  two rightmost eigenvalues $\lambda_1,\lambda_2 \in \mathbb{R}$ and set 
$\varepsilon = 10^{-5}$ (arbitrary), $\nu = (\lambda_1-\lambda_2-3\varepsilon)/\ln(2)$, $\alpha_1 = e^{(\lambda_1-\varepsilon)/\nu}$ and $\alpha_2 =e^{(\lambda_2+\varepsilon)/\nu}$.
We estimate the Lipschitz constant of the nonlinearity $\lip(\varphi_{1/\nu}-T(1/\nu))$ according to \eqref{eq:LipN_IM} (with $\omega = \lambda_1 + \varepsilon$).
The function of interest (from Remark~\ref{remark:varrho}) is
\begin{equation}
    \tau_R(s) : = \left(\frac{1}{\alpha_1-e^{s/\nu}}  + \frac{1}{e^{s/\nu}- \alpha_2} \right)\frac{2}{\omega} \left( e^{\omega /\nu} - 1 \right)  e^{\omega /\nu} \lip(R), \qquad s \in(\lambda_2 + \varepsilon,\lambda_1 - \varepsilon), 
    \label{eq:tau_R}
\end{equation}
 plotted in Figure~\ref{fig:R2h} for $h = 0.065$ and $h = 0.13$. 
The intersection of \eqref{eq:tau_R} with the line $s \mapsto 1$ give $\gamma_2$ and $\gamma_1$ from the statement of Theorem~\ref{thm:CHT}. 
In particular, the first intersection point, $\gamma_2$, approximates the decay rate $\kappa$.
This is an upper bound on $\kappa$, since the estimate on the Lipschitz constant (possibly) pushes the curve \eqref{eq:tau_R} upwards.

\begin{figure}
     \centering
     \begin{subfigure}[b]{0.49\textwidth}
         \centering
         \includegraphics[width=\textwidth]{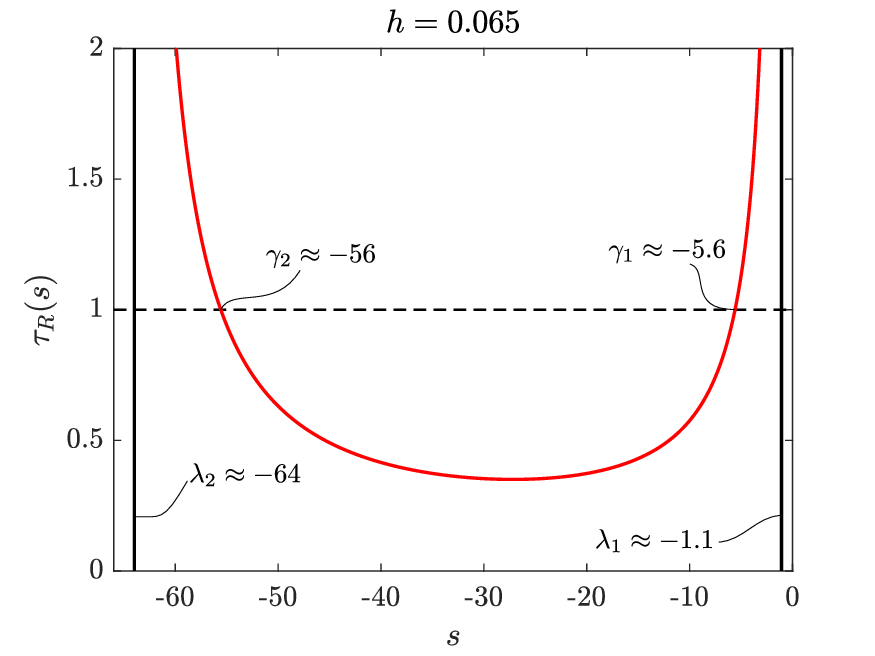}
     \end{subfigure}
     \hfill
     \begin{subfigure}[b]{0.49\textwidth}
         \centering
         \includegraphics[width=\textwidth]{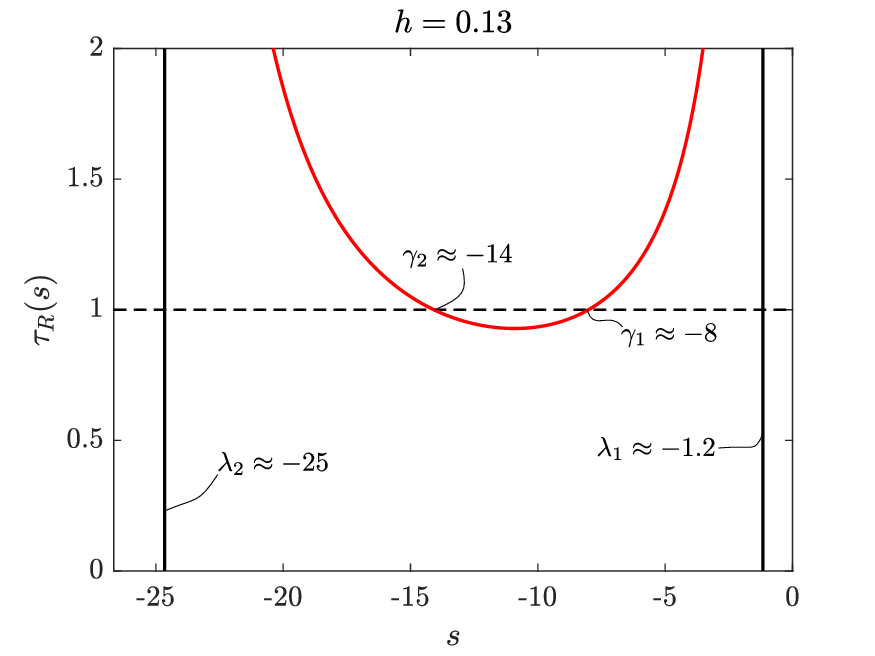}
     \end{subfigure}
        \caption{The graph of \eqref{eq:tau_R} for two values of $h$.}
        \label{fig:R2h}
\end{figure}

We also observe that, while Corollary~\ref{corollary:smalldelays} necessitated $h < 0.066$, the initial version in Theorem~\ref{thm:IM} shows existence and uniqueness of the inertial manifold $W^\Sigma$ for up to $h =0.13$ (see Figure~\ref{fig:R2h}).
The difference roughly gauges  the loss incurred through the estimates of Remark~\ref{remark:IM_altassumption} and the proof of Corollary~\ref{corollary:smalldelays}.
Let us compare this with what Remark~\ref{remark:IM_with_F} would suggest. 
For this, we verify \eqref{eq:h_cond_finale}.
First, we note that $|\beta_2|$ must be chosen larger than the left hand side of  \eqref{eq:h_cond_finale} evaluated at $h = 0$.
In our case, this translates to $\beta_2 < - 192$ (using $\lip(F) = 3$).
Beyond this value, we may always find $h$ satisfying \eqref{eq:h_cond_finale} -- the dependence of the maximal such $h$ on $\beta_2$ is depicted in Figure~\ref{fig:F} (left).
Recall that Remark~\ref{remark:IM_with_F} applied Theorem~\ref{thm:CHT} with the linear part in the 'nonlinearity' -- it hence does not carry any tangency property.
Similarly to the case above, we may compute $\varrho$ from \eqref{eq:varrho} to obtain the Lyapunov exponents characterizing the manifold (and thus also the decay rate).
In the present context, $\omega = 1$, $\nu = (-\beta_2-2\varepsilon)/\ln2$, $\alpha_1 = e^{-\varepsilon/\nu}$, $\alpha_2 = e^{\beta_2/\nu}$, and $C_2 =e^{-h \beta_2}$.
We plot 
\begin{equation}
        \tau_F(s) :=  \left(\frac{1}{\alpha_1-e^{s/\nu}}  + \frac{C_2}{e^{s/\nu}- \alpha_2} \right)\frac{2}{\omega} \left( e^{\omega /\nu} - 1 \right)  e^{\omega /\nu} \lip(F), \qquad s \in(\beta_2,- \varepsilon) 
        \label{eq:tau_F}
\end{equation}
alongside the prior curve \eqref{eq:tau_R} in Figure~\ref{fig:F} (right), for the value $h = 0.002$ (and $\beta_2 = -500$).
Comparing the two, we may conclude  via the negative semiorbit characterization of  Theorem~\ref{thm:CHT}\ref{item1:CHT} that the two manifolds must coincide.

\begin{figure}
     \centering
     \begin{subfigure}[b]{0.49\textwidth}
         \centering
         \includegraphics[width=\textwidth]{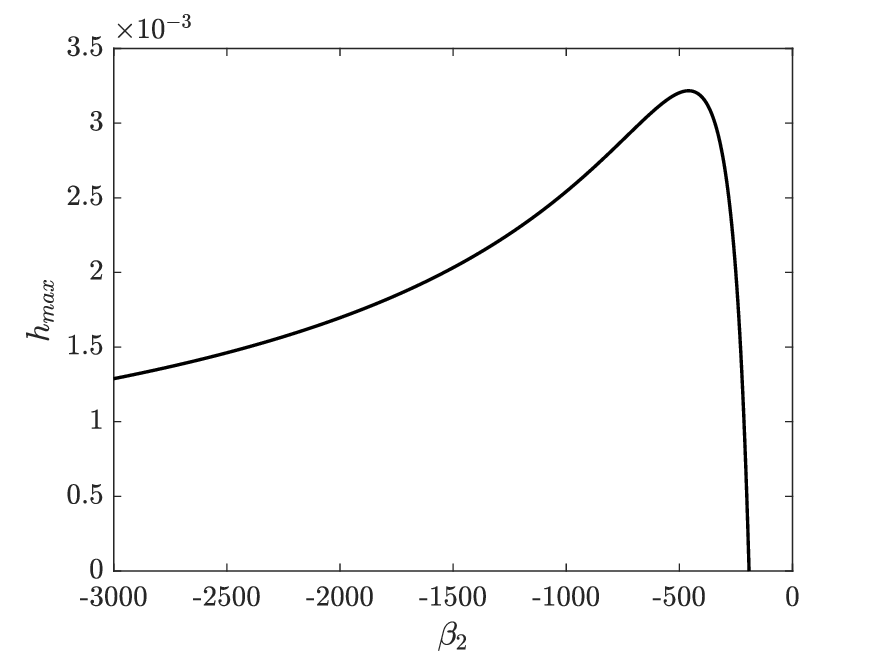}
     \end{subfigure}
     \hfill
     \begin{subfigure}[b]{0.49\textwidth}
         \centering
         \includegraphics[width=\textwidth]{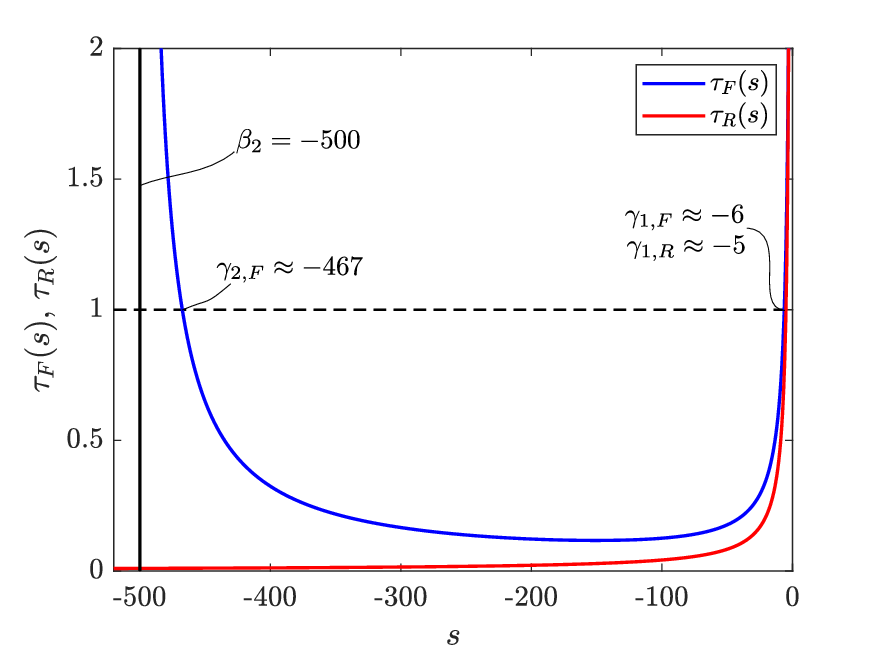}
     \end{subfigure}
        \caption{(Left) dependence of of the maximal delay $h_{max}$ on $\beta_2$ -- according to the (suboptimal) prediction of \eqref{eq:h_cond_finale}. (Right) The graph of \eqref{eq:tau_F} compared with \eqref{eq:tau_R} for $h = 0.002$ and $\beta_2 = -500$.}
        \label{fig:F}
\end{figure}

\subsubsection{The example of \texorpdfstring{\cite{szaksz2025spectral}}{[SOS25]}}
\label{sect:szaksz}

Next, we revisit the example considered in \cite{szaksz2025spectral}. 
In particular, we modify the nonlinearity in \eqref{eq:example2} to $-x(t)^3$, that is,
\begin{equation}
    \dot{x}(t) = - x(t-h) - x(t)^3.
    \label{eq:szaksz}
\end{equation}
This differs from the equation in \cite{szaksz2025spectral} by a sign (on the nonlinearity) -- 
we performed this change to make the Hopf bifurcation occurring at the critical delay $h = \pi/2$ supercritical.
Nonetheless, all coefficients of the parameterization will be comparable to \cite{szaksz2025spectral} up to a sign difference.

We take $\Sigma =\{\lambda_1\}$ when the dominant eigenvalue is real (when $h < 1/e$) and $\Sigma = \{\lambda_1,\lambda_2\}$ when the dominant eigenvalues form a complex conjugate pair (when $h > 1/e$), $\overline{\lambda}_2 = \lambda_1$.
Despite the nonlinearity not being globally Lipschitz,
the local theory (i.e., Theorems~\ref{thm:main1}, \ref{thm:main2}, \ref{thm:main3}) continues to hold, we thus obtain local existence of $W^\Sigma$ about $0$.
Next, we carry out the explicit computation of the parameterization up to third order via Lemma~\ref{lemma:expansions} and a normal form transformation for the complex case to recover the findings of \cite{szaksz2025spectral}. 
Note that the use of Lemma~\ref{lemma:expansions} necessitates a spectral gap of the form $\lambda_2 < 4\lambda_1$ in the case of a dominant real eigenvalue,
which is satisfied for $ h < \sqrt[3]{2} \ln(2)/3$.
It is also certainly satisfied in the complex conjugate case near the bifurcation.

When the dominant eigenvalue is real, $\Sigma = \{\lambda_1\}$, the reduced dynamics are 
\begin{equation}
    \dot{y}= \widetilde{H}(y) = \lambda_1 y - \xi(\lambda_1) y^3,
    \label{eq:szaksz_y}
\end{equation}
where $\widetilde{H} := \Theta^{-1} \circ H^\leq \circ \Theta$ with $\Theta : \mathbb{R} \to X_\Sigma$ denoting the isomorphism $y \mapsto (\theta \mapsto y e^{\lambda_1 \theta})$ as before.
This is computed entirely analogously to \eqref{eq:reddyn_y}, now with $R (u)= (-u(0)^3,0)$ -- and hence $D^3R(0) = (-6 [\mathrm{ev}_0 (\cdot)]^3, 0)$, noting that $D^2K(0)$ vanishes just like in Section~\ref{sect:example1}.
As for the parameterization of the manifold, $\widetilde{K} : = K^\leq \circ \Theta$, it remains to compute the third coefficient. 
For this, the third order analogue of \eqref{eq:D2K} is employed, which yields
\begin{equation}
    D^3K(0)[\Theta(y),\Theta(y),\Theta(y)] = -6 \left[ e^{3 \lambda_1 [\cdot]} -e^{\lambda_1[\cdot]}  \xi(\lambda_1) \frac{\Delta(3\lambda_1)}{2 \lambda_1} \right]\Delta(3\lambda_1)^{-1} y^3.
    \label{eq:szaksz_K3}
\end{equation}
Noting that $\xi$ is given explicitly as 
\begin{displaymath}
    \xi(\lambda_1) = \lim_{z \to \lambda_1} \frac{z-\lambda_1}{\Delta(z)} = \frac{1}{1-he^{-\lambda_1 h}} = \frac{1}{1+\lambda_1 h},
\end{displaymath}
we observe that \eqref{eq:szaksz_y} and \eqref{eq:szaksz_K3} recover (89)-(92) of \cite{szaksz2025spectral} (with a sign difference).

In the complex conjugate case, when $\Sigma = \{ \lambda_1,\lambda_2\}$, the reduced dynamics are obtained analogously to \eqref{eq:R2dynamics}. 
In the coordinates \eqref{eq:Xsigma_z}, they are given by
\begin{subequations}\label{eq:zequations}
    \begin{align}
        &\dot{z} = \lambda_1 z - \xi(\lambda_1) (z + \overline{z})^3 =: g(z), \\
        &\dot{\overline{z}} = \lambda_2 \overline{z}  - \xi(\lambda_2) (z + \overline{z})^3.
    \end{align} 
\end{subequations}
By nature of our procedure, \eqref{eq:zequations} corresponds to the graph style of parameterization.
We proceed via a normal form transformation to achieve the form reported in \cite{szaksz2025spectral}.

When $\lambda_1,\lambda_2$ are purely imaginary,
the (semisimple) normal form of \eqref{eq:zequations} is given by (see e.g.\ Chapter 1 of \cite{murdock2003normal})
\begin{subequations}\label{eq:qequations}
\begin{align}
    &\dot{q} = \lambda_1 q - \beta_{21} |q|^2 q, \\
    &\dot{\overline{q}} = \lambda_2 \overline{q}  - \overline{\beta}_{21} |q|^2 \overline{q}.
\end{align} 
\end{subequations}
Classically, in normal form computations, the transformation which leads to the form \eqref{eq:qequations} is  sought as a change of coordinates $z = q + p(q,\overline{q})$ for which
\begin{displaymath}
    g(q + p(q,\overline{q})) = \frac{d}{dt} \big(q + p(q,\overline{q}) \big)
\end{displaymath}
holds up to third order; where $p$ is a degree three homogeneous polynomial in $q$ and $\overline{q}$ over $\mathbb{C}$ (in this example).
At third order, this equation is
\begin{equation}
    Dp (q,\overline{q})[\lambda_1 q,\lambda_2 \overline{q}] - \lambda_1 p(q,\overline{q}) = -\xi(\lambda_1) (q + \overline{q})^3 +\beta_{21} |q|^2 q.
    \label{eq:hom}
\end{equation}
The coefficient $\beta_{21}$ is uniquely determined by requiring that \eqref{eq:hom}
is solvable for $p$ when $\lambda_1,\lambda_2$ are purely imaginary.
We obtain  $\beta_{21} = 3 \xi(\lambda_1)$.
Moreover, the solution to \eqref{eq:hom} (unique when $\lambda_1,\lambda_2$ are not purely imaginary) is 
\begin{displaymath}
    p(q,\overline{q}) = -\frac{\xi(\lambda_1)}{2 \lambda_1} q^3 - \frac{3 \xi (\lambda_1)}{2 \lambda_2} |q|^2 \overline{q} - \frac{\xi(\lambda_1)}{3 \lambda_2-\lambda_1} \overline{q}^3.
\end{displaymath}

We now introduce polar coordinates via $q = re^{\mathrm{i} \phi}$, $(r,\phi) \in \mathbb{R}^{\geq 0} \times \mathbb{S}$.
For $r >0$, the dynamics are hence given by
\begin{subequations}\label{eq:polars}
    \begin{align}
        & \dot{r} = \mathrm{Re} \, \lambda_1  r - 3 \mathrm{Re} \, \xi(\lambda_1) r^3,\\
        & \dot{\phi} = \mathrm{Im} \, \lambda_1 - 3 \mathrm{Im} \, \xi(\lambda_1) r^2. \label{eq:polars2}
    \end{align} 
\end{subequations}
To identify the associated parameterization of the manifold, we first formalize the computations done thus far. 
Recall that $H^\leq$ denotes the original third order vector field on $X_\Sigma$; and consider the following maps:
\begin{align*}
    &\vartheta: \mathbb{R}^{\geq 0} \times \mathbb{S} \to \mathbb{C}^2,  & &\vartheta(r,\phi) := (re^{\mathrm{i} \phi},re^{-\mathrm{i} \phi}), \\
    &\Phi: \mathbb{C}^2 \to \mathbb{C}^2, & &\Phi(q_1,q_2) := \big(q_1+p(q_1,q_2),q_2+\tilde{p}(q_2,q_1)\big), \\
    & \Psi: \mathbb{C}^2 \to (X_\Sigma)_\mathbb{C}, & &\Psi(z_1,z_2) :=\theta \mapsto z_1e^{\lambda_1 \theta} +z_2e^{\lambda_2 \theta},
\end{align*}
where $\tilde{p}$ is simply $p$ with its coefficients conjugated; and set $\Xi : = \Psi \circ \Phi$.
Then, the vector field defining \eqref{eq:polars} (for $r > 0$) is 
\begin{displaymath}
   \widetilde{H} (r,\phi) = \big(D\vartheta(r,\phi) \big)^{-1} \left( (D\Xi^{-1}) \circ \Xi [ H^\leq_\mathbb{C} \circ \Xi] \right)^{\leq} \circ \vartheta (r,\phi).
\end{displaymath}
Consequently, the associated parameterization of the manifold (to third order in $q$) is 
\begin{equation}
    \widetilde{K} = \left( K^\leq \circ \Xi \right)^\leq \circ \vartheta.
    \label{eq:widetildeK}
\end{equation}
For $w_1 : \theta \mapsto e^{\lambda_1 \theta}$ and $w_2 : \theta \mapsto e^{\lambda_2 \theta}$ 
we write 
\begin{displaymath}
    K^\leq(zw_1+\overline{z}w_2) = zw_1+\overline{z}w_2 + \frac16 \left(K_{30}z^3 + 3K_{21} |z|^2z + 3K_{12} |z|^2 \overline{z} + K_{03}\overline{z}^3  \right),
\end{displaymath}
where $K_{30} = D^3K(0)[w_1,w_1,w_1]$ and similarly for the rest.
An explicit computation via \eqref{eq:D2K} yields
\begin{align*}
    &K_{30} = - 6 \left[ \Delta(3\lambda_1)^{-1} e^{3 \lambda_1[\cdot]} -\frac{\xi(\lambda_1)}{2\lambda_1}e^{\lambda_1[\cdot]} - \frac{\xi(\lambda_2)}{3 \lambda_1 - \lambda_2} e^{\lambda_2 [\cdot]} \right], \\
    &K_{21} =  - 6 \left[\Delta(2\lambda_1+\lambda_2)^{-1} e^{(2\lambda_1+\lambda_2)[\cdot]} -\frac{\xi(\lambda_1)}{\lambda_1+\lambda_2}e^{\lambda_1[\cdot]} - \frac{\xi(\lambda_2)}{2 \lambda_1} e^{\lambda_2 [\cdot]}  \right];
\end{align*}
$K_{03}$ and $K_{12}$ are their complex conjugates.
Substituting into \eqref{eq:widetildeK},
\begin{displaymath}
    \widetilde{K} (r, \phi) = r \left( e^{\lambda_1 [\cdot]}  e^{\mathrm{i} \phi} + e^{\lambda_2 [\cdot]} e^{-\mathrm{i} \phi} \right) + r^3 \left( \widetilde{K}_{30} e^{3\mathrm{i} \phi} + \widetilde{K}_{21} e^{\mathrm{i} \phi} + \widetilde{K}_{12} e^{-\mathrm{i} \phi} + \widetilde{K}_{03} e^{-3\mathrm{i} \phi} \right),
\end{displaymath}
where
\begin{align*}
    &\widetilde{K}_{30} = -\Delta(3\lambda_1)^{-1} e^{3 \lambda_1[\cdot]}, \\
    &\widetilde{K}_{21} = -3 \left[\Delta(2\lambda_1+\lambda_2)^{-1} e^{(2\lambda_1+\lambda_2)[\cdot]} -\frac{\xi(\lambda_1)}{\lambda_1+\lambda_2}e^{\lambda_1[\cdot]}  \right];
\end{align*}
the rest of the coefficients are given by their complex conjugates once more.
This agrees with the parameterization of \cite{szaksz2025spectral} (with a sign difference). 

The reduced dynamics \eqref{eq:polars} predict the existence of a limit cycle whenever $\mathrm{Re} \, \xi(\lambda_1) \neq 0$ and $\mathrm{Re} \, \lambda_1 / \mathrm{Re} \, \xi(\lambda_1) > 0$. 
Both of these are certainly satisfied for $\mathrm{Re} \, \lambda_1 > 0$ (i.e., $h>\pi/2$), in which case the predicted limit cycle is stable (this corresponds to the 'supercritical' part of the Hopf bifurcation), with 
amplitude $\hat{r} = \sqrt{\mathrm{Re} \, \lambda_1 /  (3\mathrm{Re} \, \xi(\lambda_1))}$ and angular frequency $\omega$ determined by the right hand side of \eqref{eq:polars2} evaluated at $\hat{r}$.
Its trajectory in physical space can be estimated -- to third order -- via $x_{per}(t) \approx\widetilde{K}(\hat{r},e^{\mathrm{i}\omega t})$ (c.f.\ (112) of \cite{szaksz2025spectral}).

In examples where the nonlinearity is no longer globally Lipschitz but the existence of a compact global attractor is known, we may still conclude the existence (but no longer uniqueness) of an inertial manifold by use of cutoff functions.
This still requires a moderately large spectral gap (or small delay), at least in comparison to the Lipschitz constant $\lip(R_\rho)$ of the cut-off nonlinearity (as in Section~\ref{sect:cutoff}), where $\rho$ is selected such that $B_\rho^X(0)$ contains the global attractor $\mathcal{A}$. 
Due to the cutoff function, the foliation is no longer global either, but the exponential attraction rate in Definition~\ref{def:IM} can still be achieved so long as $\varphi_t(u)$ enters (and stays within) $B_\rho^X(0)$ for all times $t \geq t^*$ for some finite $t^*$ depending only on $|u|_X$ (this absorbing property can also be used to prove the existence of a global attractor via Proposition~\ref{prop:semiflow}\ref{sfitem:4} and the procedure described in Chapter 10 of \cite{robinson2001}).

In the simple example \eqref{eq:szaksz} considered here, we may exhibit the permissible size of the attractor $\mathcal{A}$ explicitly, by means of \eqref{eq:tau_R} (or perhaps more precisely, via \eqref{eq:H4Lipcond}).
With $\alpha_1,\alpha_2,\omega,\nu$ defined as prior to \eqref{eq:tau_R}, the Lipschitz constant of the cut-off nonlinearity, $R_\rho$ (see Section~\ref{sect:cutoff} for a precise definition), must satisfy
\begin{equation}
    \lip(R_\rho) < \frac{\alpha_1-\alpha_2}{\frac{8}{\omega} \left( e^{\omega /\nu} - 1 \right)  e^{\omega /\nu}}
    \label{eq:example2Lipcond}
\end{equation}
for the manifold $W^\Sigma_\rho$ to exist.
By Lemma~\ref{lemma:Lipschitz} and its footnote, we have
\begin{displaymath}
    \lip(R_\rho) \leq 2(1+2\Vert P_\Sigma \Vert)\lip_{4\rho}(R) \leq2(1+2\Vert P_\Sigma \Vert) 48\rho^2.
\end{displaymath}
Computing $\Vert P_\Sigma \Vert$ via \eqref{eq:P_lambda} now shows that 
the critical $\rho$ for which \eqref{eq:example2Lipcond} is met is bounded from below by $\rho_{crit} \approx 0.008$ in the delay range $h \in (\pi/2,\pi/2+0.1)$. 
Certainly then, if $h$ is only marginally above its critical value, the stable limit cycle can be confined to be within $B^X_{\rho_{crit}}(0)$. 
Provided no other attractors exist, this shows the existence of an inertial manifold (this is using Remark~\ref{remark:varrho} and the numerical observation that the smaller solution will be attained below $0$).
The conclusions obtained in the context of this example are rather weak and seemingly insignificant; our main purpose was to demonstrate the general procedure of obtaining inertial manifolds in the absence of globally Lipschitz nonlinearities.

\section{Proof of Theorem~\ref{thm:main1}}
\label{sect:proofmain1}

In this section, we provide a proof of Theorem~\ref{thm:main1}, which will largely be based on Theorem 1.1 of \cite{chen1997invariant}.

\subsection{Cutoff functions}
\label{sect:cutoff}

In order to apply this result, one needs \textit{global} control over the Lipschitz constant of the nonlinearity.
This is usually achieved by modifying the nonlinearity $R$ away from the fixed point at $0$ by means of cutoff functions.
Since $X$ is not a Hilbert space, nor is it finite dimensional, the existence of \textit{smooth} cutoff functions is a nontrivial question.
In fact, it is known that $X = C([-h,0];\mathbb{R}^n)$ does \textit{not} admit a Fréchet differentiable norm or cutoff function (this follows from Theorem 24, \cite{FRY2002}, noting that $X$ is separable but $X^*$ is not).

To retain smoothness on some (strictly smaller) subset of $O$, we proceed via two cutoff functions 
and define
\begin{equation}
    R_\rho (u) : =  \chi_\rho ( |P_\Sigma u |_X) \chi_\rho (|P_{\Sigma'} u |_X) R(u), \qquad u \in O,
    \label{eq:Rrho}
\end{equation}
where $\chi: \mathbb{R}^{\geq 0} \to [0,1]$ is a $C^\infty$ function satisfying
\begin{enumerate}[label=\upshape{(\roman*)}]
    \item $\chi(y) = 1$ for $y \in [0,1]$,
    \item $\chi(y) \in [0,1]$ for $y \in (1,2)$,
    \item $\chi(y) = 0$ for $y \geq 2$,
    \item \label{chipoint4} $|D\chi (y)| \leq 2$ for all $y \in \mathbb{R}^{\geq 0}$,
\end{enumerate}
and $\chi_\rho(u) : = \chi(u/\rho)$.
Note that $\lip(\chi_\rho) \leq 2/\rho$ by point \ref{chipoint4}.
Suppose $\rho$ is small enough so that the support of the two cutoff functions is contained in $O$.
Then, we may extend the domain of $R_\rho$ from $O$ to the whole of $X$ by setting $R_\rho|_{X \setminus O} \equiv 0$ (this modification does not alter the smoothness properties of \eqref{eq:Rrho}).
Since $\im(P_\Sigma) = X_\Sigma$ is finite dimensional, $R_\rho $ preserves the smoothness of $R$ on the open set 
\begin{equation}
    S = \left\{    u \in X     \; \big| \;      |P_{\Sigma'} u |_X < \max\{\rho , 2 |P_\Sigma u|_X \}        \right\} . 
    \label{eq:smoothnessset}
\end{equation}
Note also that $\mathrm{supp}(R_\rho) \subset \overline{B_{4\rho}(0)}^X$.

We have the following result.

\begin{lemma} \label{lemma:Lipschitz}
    Let $R_\rho :X \to X^{\odot *}$ be as in \eqref{eq:Rrho}, extended to the whole of $X$ as above.
    Denote by $\lip_\delta(R)$ the Lipschitz constant of $R$ on the (closed) ball $\overline{B_\delta(0)}^X$ of radius $\delta$ centered at the origin, and by $\lip(R_\rho)$ the global Lipschitz constant of $R_\rho:X \to X^{\odot *} $.
    Then, there exist  constants $C,c >0 $ such that $\lip(R_\rho) \leq  C \, \lip_{c \rho}(R) = o(1)$ as $\rho \to 0$.\footnote{A straightforward computation shows that we may take $c = 4$ and $C = 2(1+2\Vert P_\Sigma \Vert)$.}
\end{lemma}

\begin{proof}
    This is a consequence of Lemma IX.4.1, \cite{diekmann2012delay}, and the properties of $R$ discussed below \eqref{eq:Fexpansion}. 
    For completeness, we replicate the full proof here.
    
    Set $\xi_\rho(u) := \chi_\rho ( |P_\Sigma u |_X) \chi_\rho (|P_{\Sigma'} u |_X)$, and note that $\xi_\rho(u) = 0 $ for $|u|_X \geq c \rho $ for some $c > 0$; moreover, $\xi_\rho$ is globally Lipschitz with $\lip(\xi_\rho) = C/\rho$ for some $C > 0$ (depending on the norm of $P_\Sigma$).
    For $u,v \in X$, we have
    \begin{align*}
        | \xi_\rho (u) &R(u) - \xi_\rho (v) R(v)  |_{X^{\odot *}} \\
        &\leq | R(u) -  R(v)  |_{X^{\odot *}} \xi_{\rho}(v) + |\xi_\rho (u)- \xi_\rho (v)| | R(u)   |_{X^{\odot *}}  \\
        & \leq 
        \begin{cases}
             \lip_{c\rho}(R) |u-v|_X + \frac{C}{\rho}|u-v|_X \lip_{c\rho}(R) c\rho  & \text{if $|u|_X,|v|_X \leq c\rho $} \\
             0 & \text{if $|u|_X,|v|_X \geq c\rho $} \\
             \frac{C}{\rho} |u-v|_X \lip_{c\rho}(R) c\rho  & \text{if $|u|_X \leq c\rho$, $|v|_X \geq c\rho $} 
        \end{cases}\\
        &\leq \lip_{c\rho}(R) (c C + 1) |u-v|_X.
    \end{align*}
    Finally, note that $\lip_{c \rho}(R) \leq \sup\{ \Vert DR (u)\Vert \; | \; u \in \overline{B_{c\rho}(0)}^X  \}$, which goes to $0$ as $\rho \to 0$ by $DR(0) = 0 $ and continuity of $u \mapsto DR(u)$.
\end{proof}

\subsection{Global existence of forward solutions for the cut-off system}

Denote by $\varphi_t^\rho$ the semiflow produced by replacing $R$ with $R_\rho$ in \eqref{eq:semiflow_pert}, i.e.,
\begin{equation}
    \varphi_t^\rho(u) = T(t) u + \imath^{-1} \int_0^t T^{\odot *} (t-s) R_\rho \circ \varphi_s^\rho(u) \, ds,  \qquad (t,u) \in \mathcal{D}^{\varphi^\rho}.
    \label{eq:varphirho}
\end{equation}
The statement of Theorem 1.1, \cite{chen1997invariant}, assumes global existence of forward solutions, i.e., $\mathcal{D}^{\varphi^\rho} = \mathbb{R}^{\geq 0} \times X$.
This follows from arguments entirely analogous to the ODE case; we include a proof here for completeness.

\begin{lemma}[Global existence for $\varphi^\rho$] \label{lemma:globalexist}
    Equation \eqref{eq:varphirho} defines a maximal semiflow $\varphi^\rho$ with domain $\mathcal{D}^{\varphi^\rho} = \mathbb{R}^{\geq 0} \times X$.
\end{lemma}

\begin{proof}
    It follows directly from Proposition~\ref{prop:semiflow} (or Section VII.3, \cite{diekmann2012delay}) that \eqref{eq:varphirho} determines a unique maximal semiflow $\varphi^\rho$ on some domain $\mathcal{D}^{\varphi^\rho} \subset \mathbb{R}^{\geq 0} \times X$.
    We show global existence.
    Assume, for contradiction, the existence of $v \in X$ for which there exists $t^* \notin \mathcal{D}^{\varphi^\rho}_v$, $t^*  \in (0,\infty)$.
    
    Since $T(t)$ is strongly continuous on $X$, there exists $M > 0$ and $\omega > \omega(A)$, $\omega \neq 0$, (c.f.\ \eqref{eq:growthbound}) for which  $\Vert T(t)\Vert_{\mathcal{L}(X)} \leq Me^{\omega t}$.
    Note that $T^{\odot *}(t)$ satisfies the same bounds, since it is obtained from $T(t)$ by taking the dual twice and a restriction inbetween (neither of which increase its norm). 
    Choose $T > 0$ such that
    \begin{equation}
        \frac{M}{\omega}\left( e^{\omega T } - 1 \right) \lip(R_\rho) < 1.
        \label{eq:Tdef}
    \end{equation}
    Choose $m \in \mathbb{N}$ such that $Tm < t^* < T(m+1)$, and set $u = \varphi^\rho_{Tm}(v)$.\footnote{The crux of the argument is that $T$, as defined by \eqref{eq:Tdef}, is independent of the initial condition $u$ due to the uniformity of the Lipschitz constant.}
    The map $\mathcal{T}_u :   C([0,T];X) \to C([0,T];X)$ defined by 
    \begin{displaymath}
        [\mathcal{T}_u (x)](t) = T(t) u + \imath^{-1} \int_0^t T^{\odot *} (t-s) R_\rho ( x(s))\, ds, \qquad t \in [0,T],
    \end{displaymath}
    has image in $C([0,T],X)$ (Lemma III.2.1, \cite{diekmann2012delay}); and moreover is a uniform contraction by \eqref{eq:Tdef}:
    \begin{displaymath}
        \Vert \mathcal{T}_u (x) - \mathcal{T}_u(y) \Vert_{C([0,T];X)} \leq \frac{M}{\omega}\left( e^{\omega T } - 1 \right) \lip(R_\rho) \Vert x-y \Vert_{C([0,T];X)},
    \end{displaymath}
    where we have used also that $\Vert \imath^{-1} \Vert_{\mathcal{L}(\imath(X),X)} =1$, with $\imath(X) \subset X^{\odot *}$ equipped with the subspace topology.
    Hence $\mathcal{T}_u$ has a unique fixed point in $C([0,T];X)$, which implies $T \in \mathcal{D}^{\varphi^{\rho}}_u$ (by uniqueness of $\varphi^\rho$).
    This contradicts the maximality of $\varphi^\rho$.
\end{proof}

\subsection{Global Lipschitz constant for the cut-off semiflow \texorpdfstring{$\varphi^\rho$}{}}

Next, we translate the bounds on the Lipschitz constant of $R_\rho$ to the nonlinearity of the semiflow $\varphi^\rho$.
Let us define $N_\rho: \mathbb{R}^{\geq 0} \times X \to X$ by
\begin{displaymath}
    N_\rho(t,u ) := \imath^{-1} \int_0^t T^{\odot *}(t-s) R_\rho \circ \varphi_s^\rho(u) \, ds.
\end{displaymath} 

\begin{lemma} \label{lemma:nonlinearity_sf}
    Let $\omega > \omega(A)$, $\omega \neq 0$, and $M>0$ be as in the proof of Lemma~\ref{lemma:globalexist}, and fix $t \geq 0$.
    Then $\varphi_t^\rho : X \to X$ is globally Lipschitz for $\rho > 0$ sufficiently small (depending on $t$), with
    \begin{displaymath}
       \sup_{s \in [0,t]} \lip\big( \varphi_s^\rho\big) \leq 2M e^{t \omega}.
    \end{displaymath}
    Moreover,
    \begin{equation}
         \lip\big(N_\rho(t,\cdot)\big) \leq \frac{2M^2}{\omega} \left( e^{\omega t} - 1 \right)  e^{\omega t} \lip(R_\rho).
         \label{eq:LipN}
    \end{equation}
\end{lemma}

\begin{proof}
    The case $t =0$ is trivial, so assume $t > 0$.
    Then, for $u,v \in X$,
    \begin{displaymath}
        |\varphi_t^\rho (u) - \varphi_t^\rho(v)|_X \leq Me^{\omega t}  |u-v|_X + \frac{M}{\omega} \left( e^{\omega t} - 1 \right) \lip(R_\rho) \sup_{s \in [0,t]}  |\varphi_s^\rho (u) - \varphi_s^\rho(v)|_X.
    \end{displaymath}
    Taking the supremum over $t \in [0,T]$ and choosing $\rho$ small enough such that (c.f.\ Lemma~\ref{lemma:Lipschitz})
    \begin{equation}
        \lip(R_\rho) < \frac{\omega}{ 2 M \left( e^{\omega t} - 1 \right)},
        \label{eq:replaced}
    \end{equation}
    we obtain 
    \begin{displaymath}
        \sup_{t \in [0,T]} |\varphi_t^\rho (u) - \varphi_t^\rho(v)|_X \leq 2Me^{\omega T}  |u-v|_X .
    \end{displaymath}
    Moreover, 
    \begin{displaymath}
        |N_\rho(t,u ) - N_\rho(t,v)|_X \leq \frac{M}{\omega} \left( e^{\omega t} - 1 \right) \lip(R_\rho) 2Me^{\omega t}  |u-v|_X. \qedhere
    \end{displaymath}
\end{proof}

\subsection{Theorem 1.1 of \texorpdfstring{\cite{chen1997invariant}}{[CHT97]}}
\label{sect:CHT}

In this section, we quote Theorem 1.1 of \cite{chen1997invariant}, and verify its assumptions.
Theorem~\ref{thm:main1} is simply deduced as a corollary of this result.

\begin{theorem}[Theorem 1.1, \cite{chen1997invariant}]
\label{thm:CHT}
Assume \ref{A1} and \ref{A3}.
Let $\varphi_t^\rho$ denote the cut-off semiflow \eqref{eq:varphirho}.
Suppose $\gamma_1$ and $\gamma_2$ are real numbers such that $\beta < \gamma_2 < \gamma_1 < \alpha$.
For $\rho>0$ sufficiently small, there is a globally Lipschitz 
map 
$\phi_\rho : X_\Sigma \to X_{\Sigma'}$ with $\phi_\rho(0) = 0$, such that the submanifold 
\begin{equation}
W^{\Sigma}_\rho : = \gr ( \phi_\rho) = \left\{ \, v_1 + \phi_\rho(v_1) \; \vert \; v_1 \in X_\Sigma  \, \right\}
\label{eq:Wsigmarho}
\end{equation}
of $X$ satisfies the following properties:
\begin{enumerate}[label=\upshape{(\roman*)}]
  \item \label{item1:CHT} The restriction of the semiflow $\varphi_t^\rho$ to $W^{\Sigma}_\rho$ leaves the submanifold invariant and can be uniquely extended to a Lipschitz flow $\{\varphi_t^\rho \vert_{W^{\Sigma}_\rho} \}_{t \in \mathbb{R}}$ defined for all times.
  If $u \in W^{\Sigma}_\rho$, then the negative semiorbit $\{ \varphi_t^\rho(u) \}_{t \leq 0}$ satisfies
  \begin{equation}
      \limsup_{t \to -\infty} \frac{1}{|t|} \ln |\varphi_t^\rho(u)|_X \leq -  \gamma_1
      \label{eq:negative_semiorbit_1}
  \end{equation}
  Conversely, if $\{w(t)\}_{t \leq 0} \subset X$ is a negative semiorbit of $\varphi^\rho$\footnote{A function $w : \mathbb{R}^{\leq 0} \to X$ is called a negative semiorbit of $\varphi^\rho$ if $\varphi_t^\rho(w(s)) = w(t+s)$ for any $t \geq 0$ and $s \leq -t$.} that satisfies
  \begin{equation}
      \limsup_{t \to -\infty} \frac{1}{|t|} \ln |w(t)|_X \leq -  \gamma_2,
      \label{eq:negative_semiorbit_2}
  \end{equation}
  then $\{w(t)\}_{t \leq 0} \subset W^\Sigma_\rho$.
  
  \item There exists a continuous map $\psi_\rho : X \times X_{\Sigma'} \to X_\Sigma$ such that for each $u \in W^{\Sigma}_\rho$, $\psi_\rho(u,P_{\Sigma'} u ) = P_{\Sigma} u$ and the manifold $M^{\Sigma'}_\rho(u) = \{ \, \psi_\rho(u,v_2)+v_2 \; | \; v_2 \in X_{\Sigma'} \, \}$ passing through $u$ satisfies
  \begin{subequations} \label{eq:foliationrho}
  \begin{gather}
      \varphi_t^\rho \big( M^{\Sigma'}_\rho(u) \big) \subset M^{\Sigma'}_\rho\big(\varphi_t^\rho(u)\big), \qquad t \geq 0, \label{eq:foliationinvariance} \\
      M^{\Sigma'}_\rho(u) = \left\{ \, v \in X \; \Big\vert \; \limsup_{t \to \infty} \frac1t \ln \vert \varphi_t^\rho(u) - \varphi_t^\rho (v) \vert_X \leq  \gamma_2 \,  \right\}.\label{eq:foliationdecay}
  \end{gather}
  \end{subequations}
  Moreover, $\psi_\rho : X \times X_{\Sigma'} \to X_\Sigma$ is uniformly Lipschitz  with respect to its second argument. 
  For every $u \in X$, $M^{\Sigma'}_\rho(u) \cap W^{\Sigma}_\rho$ is a single point. 
  In particular, 
  \begin{displaymath}
    M^{\Sigma'}_\rho(u) \cap M^{\Sigma'}_\rho(v) = \emptyset \quad \text{for } u,v \in W^{\Sigma}_\rho, \; u \neq v; \qquad \bigcup_{u \in W^{\Sigma}_\rho} M^{\Sigma'}_\rho(u) = X;
  \end{displaymath}
  that is, $\{ M^{\Sigma'}_\rho(u)  \}_{u \in W^{\Sigma}_\rho}$ form a foliation of $X$ over $W^{\Sigma}_\rho$.

  \item \label{extralipschitzstatement} The Lipschitz constants $\lip(\phi_\rho)$ and $\lip \big(\psi_\rho(u,\cdot) \big)$ tend to $0$ as $\rho \to 0$, the latter independently of $u \in X$. 
\end{enumerate}
\end{theorem}

\begin{proof}
We verify hypotheses (H.1)-(H.4) of \cite{chen1997invariant}
for the case of 
$\varphi_t^\rho$, which we recall below for ease of use.
\begin{enumerate}[label =(H.\arabic*)]
  \item \label{H1CHT} $\varphi^\rho$ is continuous as a map  $\mathbb{R}^{\geq0} \times X \to X$ and there exists a constant $t > 0$ such that 
  \begin{displaymath}
    \sup_{s \in [0,t]} \mathrm{Lip} ( \varphi_s^\rho ) <\infty.
  \end{displaymath}
  
  \item \label{H2CHT} The map $\varphi_t^\rho$, with $t$ as in \ref{H1CHT}, can be decomposed as $\varphi_t^\rho = L + N$ where $L \in \mathcal{L}(X)$ and $N: X \to X$ is globally Lipschitz. 
  
  \item \label{H3CHT} There are subspaces $X_i \subset X$, $i =1,2$ and continuous projections $P_i : X \to X_i$ such that $P_1 + P_2 = \mathrm{id}_X$, $X = X_1 \oplus X_2$, $L$ leaves $X_i$ invariant and $L$ commutes with $P_i$, $i = 1,2$. 
  Denoting by $L_i:X_i \to X_i$ the restrictions of $L$, $L_1$ has bounded inverse and there exist constants $\alpha_1  > \alpha_2 \geq 0$, $C_1 \geq 1$, and $C_2 \geq 1$ such that
  \begin{subequations} \label{eq:dichotomyreq}
  \begin{align}
      \Vert L_1^{-j} P_1\Vert_{\mathcal{L}(X)} &\leq C_1 \alpha_1^{-j},  &j \in \mathbb{N}_0, \\
      \Vert L_2^{j} P_2\Vert_{\mathcal{L}(X)} &\leq C_2 \alpha_2^{j},   &j \in \mathbb{N}_0.
  \end{align}
  \end{subequations}

  \item \label{H4CHT} 
  Let 
  \begin{equation}
      \varrho(s):= \frac{C_1}{\alpha_1 - s} + \frac{C_2}{s-\alpha_2}, \qquad s \in (\alpha_2,\alpha_1).
      \label{eq:varrho}
  \end{equation}
  The numbers $\gamma_1,\gamma_2$ from the statement satisfy $\alpha_2 <e^{\gamma_2 t}<e^{\gamma_1 t} < \alpha_1$ and $\varrho(s) \lip(N) < 1$ for all $s \in (e^{\gamma_2 t},e^{\gamma_1 t})$.
\end{enumerate}

Hypotheses \ref{H1CHT} and \ref{H2CHT} follow from Lemma~\ref{lemma:nonlinearity_sf}, with $t =1$, $N = N_\rho(1,\cdot)$ and $L = T(1)$. 
The dichotomy in \ref{H3CHT} is supplied by Lemma~\ref{lemma:dichotomy}, with $X_1 = X_\Sigma$ and $X_2 = X_{\Sigma'}$.
In particular, \eqref{eq:dichotomy} implies \eqref{eq:dichotomyreq} for any $\alpha_1 < e^\alpha$ and $\alpha_2 > e^\beta$ -- we choose these so that $\alpha_1 > e^{\gamma_1}$ and $\alpha_2 < e^{\gamma_2}$.
The condition \ref{H4CHT} can be achieved through control of $\lip(R_\rho)$ via \eqref{eq:LipN} and Lemma~\ref{lemma:Lipschitz}, taking $\rho$ sufficiently small.

Statement \ref{thm:CHT}\ref{extralipschitzstatement} is not part of the main statement in \cite{chen1997invariant}, but is a direct consequence of equations (3.16) and (3.18) therein (see also \eqref{eq:lipphirho}).
\end{proof}

\begin{remark}\label{remark:altcondCHT}
As remarked also in \cite{chen1997invariant}, if the operators $L$ and $N$ satisfy
\begin{equation}
    \frac{(\sqrt{C_1}+\sqrt{C_2})^2}{\alpha_1 - \alpha_2} \mathrm{Lip} ( N ) < 1,
    \label{eq:H4Lipcond}
\end{equation}
then one can find $\gamma_1,\gamma_2$ satisfying \ref{H4CHT}.
(Equation \eqref{eq:H4Lipcond} is obtained simply by evaluating $\varrho$ at its minimum within $(\alpha_2,\alpha_1)$.)
\end{remark}

We may now deduce Theorem~\ref{thm:main1} as a corollary of Theorem~\ref{thm:CHT}.

\begin{proof}[Proof of Theorem~\ref{thm:main1}]
Suppose $\rho >0$ is small enough so that Theorem~\ref{thm:main1} applies and such that $\lip(\phi_\rho) < 1$ and $\sup_{u \in X} \lip \big(\psi_\rho(u,\cdot) \big) < 1$.
Then, an application of the parametric Banach fixed point theorem (see, e.g., Theorem 21, \cite{Irwin72}) to the map
\begin{gather*}
     g: X \times X_\Sigma \to  X_\Sigma \\ 
    (u,v_1) \mapsto \psi_\rho(u,\phi_\rho(v_1))
\end{gather*}
shows that, for each $u \in X$, the map $g(u,\cdot) :X_\Sigma \to  X_\Sigma$ has a unique fixed point.
Moreover, if we denote this fixed point by $f (u)$, then the map $ f: X \to X_\Sigma$ is continuous.
Now, defining $\pi_\rho: X \to  W^{\Sigma_1}_\rho$ as
\begin{displaymath}
   \pi_\rho(u): = f(u)+\phi_\rho \circ f (u)
\end{displaymath}
yields a continuous map, which satisfies   $M^{\Sigma'}_\rho(u) = \pi_\rho^{-1} \big( \pi_\rho(u) \big)$ for each $u \in X$.
Hence, \eqref{eq:foliationinvariance} is equivalent to
\begin{equation}
    \pi_\rho \circ \varphi_t^{\rho}(u) = \varphi_t^\rho \vert_{W^{\Sigma}_\rho} \circ \pi_\rho(u), \qquad \text{for all } (t,u) \in \mathbb{R}^{\geq 0 } \times X.
    \label{eq:pirho}
\end{equation}
It follows from \eqref{eq:foliationdecay}, with $u \in X$ fixed, that for any $\varepsilon > 0$, there exists $C > 0$ such that 
\begin{equation}
    | \varphi_t^\rho(u) - \varphi_t^\rho \circ \pi_\rho (u)|_X \leq C e^{(\gamma_2 + \varepsilon) t} , \qquad \text{for all } t \geq 0. 
    \label{eq:decayrho}
\end{equation}
Here, we have the freedom to choose $\gamma_2 < \gamma$, since $\gamma_2 \in (\beta,\alpha)$ was arbitrary in Theorem~\ref{thm:CHT} and $\beta < \gamma$.
Note also that $\pi_\rho |_{W^\Sigma_\rho} = \mathrm{id}_{W^{\Sigma}_\rho}$.

Take an open neighbourhood $V$ of $0$ such that 
\begin{displaymath}
    V \subset \left\{ u \in X \; \big| \; |P_\Sigma u | < \rho, \;  |P_{\Sigma'} u | < \rho \right\}.
\end{displaymath}
Then $\varphi \equiv \varphi^\rho$ on $\mathcal{U}(V)$.
Set $U := \big( \pi_\rho^{-1} (W^\Sigma_\rho \cap V) \big) \cap V$.
This is an open neighbourhood of $0$.

We may now declare $W^\Sigma : = W^\Sigma_\rho \cap U \;  (= W^\Sigma_\rho \cap V)$ and $\pi : = \pi_\rho \vert_U$. 
Statement \ref{thm:main1}\ref{thm1st1} is immediate from the fact that $\varphi \equiv \varphi^\rho$ on $\mathcal{U}(U)$.
Statement \ref{thm:main1}\ref{thm1st3} follows from the choice of $U$ and \eqref{eq:foliationrho}.
More precisely, 
by construction of $U$, we have $\pi(U) \subset W^\Sigma \subset  U$.
Given $(t,u) \in \mathcal{U}(U)$, \eqref{eq:pirho} combined with the above observation implies that $(t,\pi(u)) \in \mathcal{U}(U)$ and $\pi \circ \varphi_t(u) = \varphi_t|_{W^\Sigma} \circ \pi(u)  $.
This shows \eqref{eq:pi}.
If $\mathcal{U}_u(U) \neq \mathbb{R}^{\geq 0}$,  \eqref{eq:attractivity} is trivially satisfied.
Otherwise, \eqref{eq:pi} holds for all times on a domain on which $\pi_\rho = \pi$ and $\varphi^\rho = \varphi$;
hence \eqref{eq:decayrho} implies \eqref{eq:attractivity}, being that $\gamma_2 < \gamma$.

The first part of statement \ref{thm:main1}\ref{thm1st4} was covered in Remark~\ref{remark:strongmfds}.
The second part follows from the argument coined by \cite{burchard1992smooth} for the case of center manifolds.
It is written out in the specific context of the present result in \cite{buza2024spectral}, Lemma 4.8 (note that its proof is independent of the equation at hand, and relies only on the existence of the foliation map $\pi$).

For the proof of statement \ref{thm:main1}\ref{thm1st2} and the $C^1$ part of \ref{thm:main1}\ref{thm1st1}, see Appendix~\ref{sect:mfdsmoothness}.
Once a certain degree of smoothness, say $C^\ell$, $\ell \leq k$, of the manifold $W^\Sigma$ has been obtained, Corollary (8A.9) of \cite{marsden1976hopf} immediately implies that the restricted semiflow extends to a jointly $C^\ell$ flow.
The tangency of the manifold $W^\Sigma$ to $X_\Sigma$ at $0$ is obtained in Corollary~\ref{corollary:tangency}.
\end{proof}

\section{Proof of Theorem~\ref{thm:main2}}
\label{sect:proofmain2}

We start by reducing the problem to the case when $0$ is a (spectrally) stable fixed point, by restricting the semiflow $\varphi$ to the stable manifold $W^s$.
Explicitly, we consider the setting of Appendix~\ref{sect:mfdsmoothness} (or that of Chapter VIII, \cite{diekmann2012delay}).
Choose $\sup\{ \mathrm{Re} \, \lambda \; | \; \lambda \in \sigma(A;X), \; \mathrm{Re} \, \lambda < 0 \}< \gamma < 0$.
The stable manifold can be characterized via a $C^k$ map $\mathcal{W}^s : U_s \to BC^{-\gamma}(\mathbb{R}^{\geq 0};X)$\footnote{The Banach space $BC^{-\gamma}(\mathbb{R}^{\geq 0};X)$ consists of continuous functions $\mathbb{R}^{\geq 0} \to X$ for which the norm  given by $|f|_{-\gamma}=\sup_{t \geq 0} e^{-\gamma t}|f(t)|_X$ is finite.
This is analogous to Definition~\ref{def:BC}.
$\gamma$ being negative means that forward orbits on $W^s$ decay exponentially, at least as fast as $e^{\gamma t}$.
},
which maps points of an open neighbourhood $U_s$ of $0$ in $X_s$ to forward orbits on the submanifold $W^s$.
Here $X_s$ denotes the stable subspace characterized by the half-plane  of the spectrum strictly to the left of the imaginary axis; let us also denote by $\widetilde{P}_s:X \to X_s$ the associated projection with restricted range.
We define
\begin{equation}
    \psi_t : = \widetilde{P}_s \circ \mathrm{ev}_t \circ \mathcal{W}^s, \qquad t \geq 0,
    \label{eq:psi}
\end{equation}
on $U_s$,
where $\mathrm{ev}_t: BC^{-\gamma}(\mathbb{R}^{\geq 0};X) \to X$ is the continuous linear map given by $f \mapsto f(t)$.

\begin{lemma} \label{lemma:psi}
    The maps $\{\psi_t\}_{t \geq 0}$ from \eqref{eq:psi} constitute a semiflow  with domain $\mathcal{D}^\psi \subset\mathbb{R}^{\geq 0} \times U_s$, which satisfies the same smoothness properties as the original semiflow $\varphi$.
\end{lemma}

\begin{proof}
    Denoting by $K: U_s \to X$ the $C^k$ map $\mathrm{ev}_0 \circ \mathcal{W}^s$ (whose image is $W^s$),
    we have 
    \begin{equation}
        \psi_t = \widetilde{P}_s \circ \varphi_t \circ K, \qquad t \geq 0.
        \label{eq:psi2}
    \end{equation}
    Therefore, smoothness properties are inherited.
    
    Invariance of $W^s$ implies that $\varphi_t \circ K(u) \in \im(K)$ for all $t\geq 0$ such that $\psi_\tau(u) \in U_s$ for $\tau \in [0,t]$.
    By construction, we have $K \circ \widetilde{P}_s|_{\im(K)} = \mathrm{id}_{\im(K)}$ and $\widetilde{P}_s \circ K = \mathrm{id}_{U_s}$ -- for a detailed derivation see Appendix~\ref{sect:mfdsmoothness}, specifically the paragraph above \eqref{eq:phirho}.
    The semiflow property thus follows from \eqref{eq:psi2} on the domain of $\psi$, which can be identified as 
    \begin{equation}
        \mathcal{D}^\psi = \{ (t,u) \in \mathbb{R}^{\geq 0} \times U_s \; | \; \psi_\tau (u) \in U_s \text{ for all } \tau \in [0,t] \}. \qedhere
        \label{eq:mathcalDpsi}
    \end{equation}
\end{proof}

\begin{lemma} \label{lemma:LyapstabCkb}
    The fixed point $0$ of $\psi_t$ is Lyapunov and exponentially stable (with respect to the induced norm $| \cdot |_{X_s}$).
    Moreover, there exists an open neighbourhood of the origin $U \subset U_s$ such that $\mathbb{R}^{\geq 0} \times U \subset \mathcal{D}^\psi$ and $\psi_t \in C^k_b(U)$ for all $t \geq 0$.
\end{lemma}

\begin{proof}
        Choose $\varepsilon> 0$ such that $B_\varepsilon^{X_s}(0) \subset U_s$.
    Since $\mathcal{W}^s : U_s \to BC^{-\gamma}(\mathbb{R}^{\geq 0};X)$ is continuous and $\mathcal{W}^s(0) = 0$, we may choose a small enough neighbourhood $U$ of $0$ in $U_s$ such that 
    \begin{displaymath}
        \Vert P_s \Vert_{\mathcal{L}(X)}\sup_{t \geq 0} e^{-\gamma t} \left|\mathrm{ev}_t \circ \mathcal{W}^s(u) \right|_{X} < \varepsilon, \qquad u \in U.
    \end{displaymath}
    Since $\gamma < 0$, this implies Lyapunov and exponential stability of $0$ under $\psi_t$; and moreover that 
    \begin{displaymath}
        \Vert P_s \Vert_{\mathcal{L}(X)}\left| \mathrm{ev}_t \circ \mathcal{W}^s(u) \right|_{X} < \varepsilon \qquad  \text{for } (t,u) \in \mathbb{R}^{\geq 0} \times U.
    \end{displaymath}
    In particular, $\mathbb{R}^{\geq 0} \times U \subset \mathcal{D}^\psi$, c.f.\ \eqref{eq:mathcalDpsi}, and $\psi_t \in C^0_b(U)$ for all $t \geq 0$.

    Since $\mathcal{W}^s$ is of class $C^k$, we may shrink $U$ such that\footnote{For the introduction of $M_j$ spaces of multilinear maps, see Appendix~\ref{sect:mfdsmoothness}.}
    \begin{displaymath}
        \left\Vert D^{j} \mathcal{W}^s(u) - D^j \mathcal{W}^s(0) \right\Vert_{M_j(X_s;BC^{-\gamma}(\mathbb{R}^{\geq 0};X))} < \varepsilon, \qquad \text{for all } u \in U, \; 1 \leq j \leq k.
    \end{displaymath}
    We have
    \begin{displaymath}
        D^j \psi_t(u) = \mathrm{ev}_t \left( \widetilde{P}_s D^j \mathcal{W}^s(u) \right). 
    \end{displaymath}
    Hence $\psi_t \in C^k_b(U)$ follows for all $t \geq 0$, noting that $\Vert \mathrm{ev}_t\Vert \leq 1$ due to $\gamma < 0$.
\end{proof}

\begin{lemma} \label{lemma:jointcont}
    For each $0 \leq j \leq k$,  $(t,u) \mapsto D^j(\psi_t \circ \imath_{\widetilde{\Sigma}}) (u)$ is jointly continuous over the line $\mathbb{R}^{\geq 0} \times \{0\}$, where $\imath_{\widetilde{\Sigma}} : U_s \cap X_{\widetilde{\Sigma}} \xhookrightarrow{} U_s$ is the inclusion.
\end{lemma}

\begin{proof}
    Apply Remark~\ref{remark:jointcont} with $W = K(X_{\widetilde{\Sigma}} \cap U_s)$; noting that $ K \circ \imath_{\widetilde{\Sigma}} = \imath_W \circ \widetilde{K \circ \imath_{\widetilde{\Sigma}}}$, where $\widetilde{K \circ \imath_{\widetilde{\Sigma}}}$ is considered with range restricted to $W$ (it retains the smoothness of $K$ by Theorem 18.2, \cite{munkres2000topology}).
\end{proof}

We are now ready to prove the main result.

\begin{proof}[Proof of Theorem~\ref{thm:main2}]
    Let
    \begin{displaymath}
        \Omega := \left\{ \lambda \in \sigma(A;X) \; | \; \mathrm{Re} \, \lambda \geq \inf \mathrm{Re} \, \Sigma \right\}.
    \end{displaymath}
    We may apply Theorem~\ref{thm:main1} with $\Omega$ as the spectral subset to obtain a locally invariant $C^\ell$ submanifold $W^\Omega$ of $X$ (contained in a small enough neighbourhood of the origin).
    The inclusion $W^\Omega \xhookrightarrow{} X$ is transversal over $W^s$ (locally, about $0$) 
    in the sense of \cite{lang2012fundamentals}, page 29.
    We hence have, up to perhaps shrinking the manifolds, that $W^{\Omega,s} : = W^\Omega \pitchfork W^s$ is a locally invariant $C^\ell$ submanifold of $X$.
    It is tangent to $X_\Omega \cap X_s$ at $0$. 
    If $\widetilde{\Sigma}$ is empty, we are done -- hence we suppose otherwise in what follows.

    Next, we apply Theorem 2.6, \cite{buza2025smooth}, to the semiflow $\psi$ with respect to the splitting $X_s = X_{\widetilde{\Sigma}} \oplus X_{\widetilde{\Sigma}'}$.
    We verify assumptions (A.1) and (A.2) of \cite{buza2025smooth}.
    The contents of (A.1) were verified by Lemmas \ref{lemma:psi} and \ref{lemma:LyapstabCkb}.
    The splitting for (A.2) therein is provided above,   $X_{\widetilde{\Sigma}}$ taking the role of $X_0$.
    That $t \mapsto D\psi_t(0)$ is strongly continuous is a consequence of the tangency of $W^s = \im(K)$ to $X_s$ at $0$:
    \begin{align*}
    D \psi_t(0) &= \widetilde{P}_s D \varphi_t(0)DK(0) \\
                &= T(t)|_{X_s} \widetilde{P}_sDK(0) \\
                &= T(t)|_{X_s}.
    \end{align*}
    This leaves  $X_{\widetilde{\Sigma}'}$ invariant; its generator, $A|_{X_s}$, satisfies $X_{\widetilde{\Sigma}} \subset\dom(A|_{X_s})$.
    The joint continuity assumption was verified by Lemma~\ref{lemma:jointcont}. 
    The remainder of (A.2) therein concerns spectral assumptions.
    Noting that spectral mapping holds for the semigroup $T(t)|_{X_s}$ (Theorem~\ref{thm:spectrum}) and that the spectrum of $A|_{X_s}$ is confined to the region $|z| \leq C_1 e^{-h \mathrm{Re} \, z} $ for large enough $|z|$ (Theorem I.4.1, \cite{diekmann2012delay}), we may verify the spectral assumptions directly on the generator, as in Remark 2.8, \cite{buza2025smooth} (simply replace its sectoriality assumption with the above). 
    (A.2) therein is hence fulfilled by \ref{A4} herein.

    This yields a $C^k$ submersion $\pi: U \to X_{\widetilde{\Sigma}}$ on some small enough neighbourhood $U$ of the origin in $X_s$ that satisfies $\pi(0) = 0$, 
    \begin{equation}
        D\pi(0)|_{X_{\widetilde{\Sigma}}} \in \mathrm{Aut}(X_{\widetilde{\Sigma}})
        \label{eq:Dpi0}
    \end{equation}
    and
    \begin{equation}
        \pi \circ \psi_t = \vartheta_t \circ \pi \qquad \text{for all } t \geq 0
        \label{eq:pi_invariance}
    \end{equation}
    on some open $V \subset U$ containing the origin; where $\{ \vartheta_t \}_{t \geq 0}$ is a semiflow on $X_{\widetilde{\Sigma}}$, with a fixed point at $0$.
    Up to perhaps confining $W^{\Omega,s}$ to a smaller neighbourhood of the origin, we may assume $\widetilde{P}_s (W^{\Omega,s}) \subset V$.
    Precomposing \eqref{eq:pi_invariance} with $\widetilde{P}_s|_{W^{\Omega,s}}$, we obtain (c.f.\ the proof of Lemma~\ref{lemma:psi})
    \begin{displaymath}
        \pi \circ \widetilde{P}_s \circ \varphi_t |_{W^{\Omega,s}} = \vartheta_t \circ \pi \circ \widetilde{P}_s|_{W^{\Omega,s}}, \qquad t \geq 0.
    \end{displaymath}
    By Lyapunov stability of $0$ with respect to  $\varphi_t |_{W^{\Omega,s}}$ and local invariance of $W^{\Omega,s}$, we may choose a neighbourhood $\widetilde{V} \subset W^{\Omega,s}$ of the origin such that $\varphi_t(\widetilde{V}) \subset W^{\Omega,s}$ for all $t \geq 0$.
    We hence arrive at
    \begin{equation}
        \pi \circ \widetilde{P}_s |_{W^{\Omega,s}}  \circ \varphi_t = \vartheta_t \circ \pi \circ \widetilde{P}_s|_{W^{\Omega,s}}, \qquad t \geq 0,
        \label{eq:invariance_finale}
    \end{equation}
    which holds on $\widetilde{V}$.
    
    Due to the restriction map, $\pi \circ \widetilde{P}_s |_{\widetilde{V}}$ is only of class $C^\ell$.
    We check that it is a submersion at $0$.
    We have
    \begin{equation}
        D \big( \pi \circ \widetilde{P}_s |_{\widetilde{V}} \big) (0) = D \pi (0) \widetilde{P}_s  D \imath_{W^{\Omega,s}}(0),
        \label{eq:DpiP}
    \end{equation}
    where $\imath_{W^{\Omega,s}}:W^{\Omega,s} \xhookrightarrow{} X$ is the inclusion. 
    Noting the tangency of $W^{\Omega,s}$ to $X_\Omega \cap X_s $ at $0$,
    the map \eqref{eq:DpiP} is
    \begin{equation}
        X_\Omega \cap X_s \xhookrightarrow{} X_s \xrightarrow{D \pi (0)} X_{\widetilde{\Sigma}}.
        \label{eq:quoteforkernel}
    \end{equation}
    Considering that $X_{\widetilde{\Sigma}} \subset X_\Omega \cap X_s$, \eqref{eq:Dpi0} indeed confirms that $\pi \circ \widetilde{P}_s |_{\widetilde{V}}$ is a $C^\ell$ submersion on a neighbourhood of $0$.
    The kernel of \eqref{eq:quoteforkernel} can be identified as $X_{\Sigma}$.
    
    Hence, $\big(   \pi \circ \widetilde{P}_s |_{\widetilde{V}}  \big)^{-1}(0)$ is a locally invariant (by \eqref{eq:invariance_finale}) $C^\ell$ submanifold of $\widetilde{V}$, and hence of $X$, tangent to $X_{\Sigma}$ at $0$. 
    The joint smoothness of the restricted semiflow follows from the same argument as in the proof of Theorem~\ref{thm:main1}.
\end{proof}

\section{Proof of Theorem~\ref{thm:IM} and Corollary~\ref{corollary:smalldelays}}
\label{sect:IM}

The proof of Theorem~\ref{thm:IM} proceeds analogously to the proof of Theorem~\ref{thm:main1} in Section~\ref{sect:proofmain1}.
We simply verify the assumptions of Theorem 1.1, \cite{chen1997invariant}, now for the full, unaltered semiflow. 

\begin{proof}[Proof of Theorem~\ref{thm:IM}]
    Under the global Lipschitz assumption, the solutions \eqref{eq:semiflow_pert} exist for all times (just as in Lemma~\ref{lemma:globalexist}), hence the domain of $\varphi$ is $\mathcal{D}^\varphi = \mathbb{R}^{\geq 0} \times X$.
    Repeating the proof of Lemma~\ref{lemma:nonlinearity_sf}
    with
    $N: \mathbb{R}^{\geq 0} \times X \to X$,
    \begin{displaymath}
        N(t,u ) := \imath^{-1} \int_0^t T^{\odot *}(t-s) R \circ \varphi_s(u) \, ds,
    \end{displaymath}
    we obtain
    \begin{equation}
        \lip\big(N(1/\nu,\cdot)\big) \leq \frac{2M^2}{\omega} \left( e^{\omega /\nu} - 1 \right)  e^{\omega /\nu} \lip(R),
        \label{eq:LipN_IM}
    \end{equation}
    having used \eqref{eq:replacement} in place of \eqref{eq:replaced}.

    The assumptions \ref{H1CHT}-\ref{H3CHT} are now easily verified, just as in the proof of Theorem~\ref{thm:CHT}.
    In place of \ref{H4CHT}, we shall verify \eqref{eq:H4Lipcond} of Remark~\ref{remark:altcondCHT}.
    This follows from \eqref{eq:nu_cond_2} and the condition   $\nu < (\alpha - \beta - \varepsilon_1 - \varepsilon_2)/ \ln 2$.
    Indeed, the latter implies
    \begin{displaymath}
      \frac{1}{2} e^{( \alpha - \varepsilon_1)/\nu} <  e^{( \alpha - \varepsilon_1) / \nu} - e^{(\beta + \varepsilon_2) / \nu},
    \end{displaymath}
    whereas \eqref{eq:nu_cond_2} is simply
    \begin{displaymath}
        \left(\sqrt{K_1} + \sqrt{K_2} \right)^2\lip\big(N(1/\nu,\cdot)\big) < \frac{1}{2} e^{( \alpha - \varepsilon_1) / \nu}.
    \end{displaymath}
    This shows \eqref{eq:H4Lipcond}; and hence all conclusions of Theorem~\ref{thm:IM}, by an application of Theorem~\ref{thm:CHT} to the unaltered semiflow (i.e., without the $\rho$ sub/superscripts).
    Since Theorem~\ref{thm:CHT} provides a rate of attraction strictly less than $\alpha \leq 0$, we must have that the resulting manifold is exponentially attracting. 
\end{proof}

\begin{proof}[Proof of Corollary~\ref{corollary:smalldelays}]
    We shall produce a sufficiently large spectral gap contained in the region $\mathrm{Re} \, z < 0$.
    In this region, from \eqref{eq:detDelta}, we have
    \begin{displaymath}
        |\det(\Delta (z))| \geq |z|^n \left( 1 - \sum_{j=1}^n \left( \frac{e^{-h \mathrm{Re} \, z}}{|z|}\right)^j \prod_{i =1}^j \mathrm{TV}(\eta_{ji})  \right).
    \end{displaymath}
    We apply Young's inequality to each term in the summation.
    Explicitly, for $j = 1,\ldots,n-1$, 
    \begin{displaymath}
        \left( \frac{e^{-h \mathrm{Re} \, z}}{|z|}\right)^j \prod_{i =1}^j \mathrm{TV}(\eta_{ji}) \leq \frac{j}{n} (\varepsilon_j)^{-\frac{n}{j}} \left( \frac{e^{-h \mathrm{Re} \, z}}{|z|}\right)^n  + \frac{n-j}{n} \left( \varepsilon_j \prod_{i =1}^j \mathrm{TV}(\eta_{ji}) \right)^{\frac{n}{n-j}},
    \end{displaymath}
    for some $\varepsilon_j >0$.
    We choose $\varepsilon_j$ such that
    \begin{displaymath}
        \frac{n-j}{n} \left( \varepsilon_j \prod_{i =1}^j \mathrm{TV}(\eta_{ji}) \right)^{\frac{n}{n-j}} = \frac{1}{2n}.
    \end{displaymath}
    Therefore, we arrive at
    \begin{displaymath}
       |\det(\Delta (z))| \geq |z|^n \left( \frac{1}{2} -  \left( \frac{e^{-h \mathrm{Re} \, z}}{|z|}\right)^n Q \right),
    \end{displaymath}
    with $Q$ as in \eqref{eq:Q}.
    Hence, if 
    \begin{equation}
        h < \frac{1}{-\mathrm{Re} \, z} \ln \left( (2Q)^{-\frac{1}{n}} |z|\right),
        \label{eq:h_cond2}
    \end{equation}
    then $|\det(\Delta (z))| > 0$. 

    The condition \eqref{eq:h_cond} on $h$ implies that \eqref{eq:h_cond2} holds on the vertical strip 
    \begin{displaymath}
        \{ z \in \mathbb{C} \; | \; \mathrm{Re} \, z \in [-r,\gamma] \}.
    \end{displaymath}
    This shows the existence of a spectral gap of size $r+\gamma$, containing $\gamma$.
    We hence must have $\alpha - \beta >r + \gamma$.
    The existence of $\nu>0$ satisfying the assumptions of Theorem~\ref{thm:IM} now follows from the choice of $r$ in \eqref{eq:r} and Remark~\ref{remark:IM_altassumption}.
    The rate of attraction claim follows from the fact that we could have equivalently defined $\gamma = -r$ (since both choices are within the same spectral gap).
\end{proof}

\section{Conclusions}

In this paper, we have derived existence and regularity results for spectral submanifolds (SSMs) in time delay systems.
The main purpose in doing so was to provide rigorous
mathematical support to existing works on the subject \cite{szaksz2024reduction,szaksz2025spectral} and those to come.
The statements of all main results are gathered in Section~\ref{sect:statement}.
Fleshing out the theoretical aspects 
is perhaps more practically pertinent
here than in the case of other infinite-dimensional dynamical systems (e.g., PDEs), since the embedding defining the invariant manifold can be explicitly constructed in a way that preserves the infinite-dimensional nature of the problem \cite{szaksz2025spectral}.
This observation 
provided the main theme of examples explored herein. 
In particular, the procedure described in the examples (Section~\ref{sect:examples}) is approximate only in the sense of truncating the Taylor expansion of the embedding, 
no other form of discretization was performed.

Explicit formulas for the Taylor coefficients of the embedding and the reduced dynamics in the graph style of parameterization are given in Lemma~\ref{lemma:expansions}.
These formulas are completely rigorous under spectral gap assumptions ensuring the necessary smoothness of the manifold, but the coefficients themselves could, in practice, be computed under less strict non-resonance conditions. 
We demonstrated, through the course of Examples~\ref{sect:example1} and \ref{sect:example2}, how these formulas 
are  utilized in
explicit calculations.
Moreover, based on the theory of Section~\ref{sect:IM_statement}, we derived conditions on the nonlinearity (Example~\ref{sect:example1}) or delay (Example~\ref{sect:example2}) under which the SSM extends to a globally defined inertial manifold, which is then uniquely characterized by Lyapunov exponents of its trajectories.
In the final example (Section~\ref{sect:szaksz}), we demonstrated how the existence (but no longer uniqueness) of inertial manifolds may be concluded if the nonlinearity is not globally Lipschitz but the existence and size of the global attractor is known.

The examples highlight 
that the dimension of an inertial manifold need not be equal to the dimension of the physical configuration of the system, which is perhaps a novelty that has not been explicitly pointed out before (to the authors' knowledge).

The conditions for the existence of SSMs locally about a fixed point are quite permissive.
Indeed, $C^1$ pseudo-unstable manifolds exist whenever $f$ (from \eqref{eq:DDE_classical}) is of class $C^1$, without any additional assumptions.  
Even the non-resonance conditions for Theorem~\ref{thm:main2}, spelled out in \ref{A4}, hold generically.
In contrast, global theory, i.e., the existence of inertial manifolds, hinges on stringent conditions which may be difficult to verify in practice.
Even in the academic example of Section~\ref{sect:szaksz}, the permissible size of the attractor is rather small. 
Of course, if one has full control over the delay or the nonlinearity, the conditions can be easily verified, as already observed in \cite{driver1968ryabov}.

The local theory described herein should extend to the case when the fixed point is replaced by a periodic orbit in a straightforward fashion, by applying similar invariant manifold results to the Poincaré map.
More generally, neighborhoods of normally hyperbolic invariant manifolds could also be considered, but this necessarily hinges on a uniform spectral gap condition (in the Sacker-Sell sense), which would be difficult to verify in practice.
Perhaps a more interesting question is whether the picture described in the work of Haller et al.~\cite{haller2023nonlinear} holds near a fixed point of a DDE.
Namely, if there exists a family of invariant manifolds of fractional smoothness tangent to the slow subspace that tessellate the complement of the strong stable manifold in phase space.
The practical significance of this question is intrinsically tied to data-driven techniques, whereas in the theoretical realm, it poses an interesting problem due to the failure of Sternberg-type linearization results  in the DDE setting, which the original proof in \cite{haller2023nonlinear} relies on.

\appendix
\section{Preliminaries}
\label{sect:prelim}

In this appendix, we give a brief overview of some basic concepts appearing in the work that might be unfamiliar to some readers.
Section~\ref{sect:BV} is largely based on Appendix I of \cite{diekmann2012delay}; 
the definitions in Section~\ref{sect:SGprelim} were taken from \cite{buhler2018functional} and \cite{marsden1976hopf};
whereas Section~\ref{sect:weakstar_int} follows III.1 and A.II.3.13 of \cite{diekmann2012delay}.

\subsection{Functions of bounded variation}
\label{sect:BV}

A function $\eta:[a,b] \to \mathbb{R}$ is said to be of \textit{bounded variation}, $\eta \in \bv\big([a,b];\mathbb{R}\big)$, if its total variation,
\begin{equation}
    \mathrm{TV}(\eta) = \sup_{P \in \mathcal{P}(a,b)} \sum_{j=1}^{m_P} |\eta(\sigma_j)-\eta(\sigma_{j-1})|,
    \label{eq:TV}
\end{equation}
is bounded; here the supremum is taken over all partitions $P = \{ \sigma_0,\ldots,\sigma_{m_P} \} \in \mathcal{P}(a,b)$ of $[a,b]$ with $a = \sigma_0 < \sigma_1<\cdots<\sigma_{m_P}=b$.
The width of a partition is
\begin{displaymath}
    \mathrm{w}(P) = \max_{1\leq j \leq m_P} (\sigma_j-\sigma_{j-1}).
\end{displaymath}
A function $\eta \in \bv\big([a,b];\mathbb{R}\big)$ is called \textit{normalized} if $\eta(a) = 0$ and $\eta$ is continuous from the right on $(a,b)$.
We write $\eta \in \nbv\big([a,b];\mathbb{R}\big)$ to express that $\eta$ is of normalized bounded variation.
These definitions extend to the complex- and multidimensional case in the obvious fashion (through real/complex parts and component-wise).

Let $\zeta :[a,b] \to \mathbb{C}^{n \times n}$ and $u : [a,b]\to \mathbb{C}^{n}$ be two given maps.
Then for any partition $P \in \mathcal{P}(a,b)$ and any choice of $\tau_j \in [\sigma_{j-1},\sigma_j]$ we denote by
\begin{displaymath}
    S(P,\zeta,u) := \sum_{j=1}^{m_P} [\zeta(\sigma_j)-\zeta(\sigma_{j-1})]u(\tau_j).
\end{displaymath}
If there exists $c \in \mathbb{C}^n$ such that for all $\varepsilon > 0$ there exists $\delta>0$ for which $|c-S(P,\zeta,u)| < \varepsilon$ holds for all $P \in \mathcal{P}(a,b)$ with $\mathrm{w}(P) < \delta$ and all choices of $\tau_j $, we say that $u$ is Riemann-Stieltjes integrable with respect to $\zeta$ over $[a,b]$ and write
\begin{displaymath}
    c = \int_a^b d\zeta(t) u(t).
\end{displaymath}

The (original form of the) Riesz representation theorem \cite{Riesz1909Sur} asserts that any continuous linear functional on $C \big( [a,b]; \mathbb{C} \big)$ is represented by a unique element of $\nbv \big( [a,b] ;\mathbb{C}\big)$.
This is used in establishing \eqref{eq:RieszDf0} and \eqref{eq:pairing}. 
For the reader more familiar with the modern version of Riesz's theorem (e.g., Theorem 3.15, \cite{salamon2016measure}), we remark that this is equivalent to the statement that the dual space of $\mathbb{C}$-valued continuous functions on $[a,b]$ is (isometrically) isomorphic to the space of finite complex Borel measures on $[a,b]$; with the equivalence given by setting $\zeta: t \mapsto \mu([a,t]) \in \nbv$ for $\mu$ a Borel measure.

A function $\eta \in \nbv([a,b]; \mathbb{C})$ belongs to $ \sbv([a,b]; \mathbb{C})$ if there exists $t \in(0,b]$ such that $\lim_{s \searrow t}\eta(s) \neq \eta(t)$ and $\eta(s) = \eta(t)$ for all $s \in [t,b]$. 
Given a kernel $\zeta\in \nbv([a,b]; \mathbb{C}^{n \times n})$ determining $\Delta$ through \eqref{eq:characteristic_matrix},
we say $\zeta \in \sbv$ if the characteristic function, which can be written as (c.f.\ Section V.8, \cite{diekmann2012delay})
\begin{displaymath}
    \det \Delta(z) = z^n + \sum_{j = 1}^n \int_0^{\tau_j} e^{-z \theta} d \eta_j (\theta) z^{n-j}
\end{displaymath}
for some $\tau_j \leq j h$ and $\eta_j \in \nbv([0,\tau_j];\mathbb{C})$, satisfies in addition the requirement $\eta_j \in \sbv([0,\tau_j];\mathbb{C})$.
This is a fairly general concept; e.g., $\zeta \in \sbv$ if it is purely atomic (that is, \eqref{eq:DDE_classical1} consists of discrete delays only), or if $\det_* \zeta \in \sbv$, where $\det_*$ denotes the determinant with respect to the convolution product (Definition V.5.2, \cite{diekmann2012delay}).

\subsection{Semiflows and semigroups} 
\label{sect:SGprelim}

Let $M$ be a Banach manifold. 
A family of maps $\{\varphi_t\}_{t \geq 0}$, $\varphi_t:M \to M$, is called a \textit{semiflow} on $M$ if
\begin{subequations} \label{eq:semiflow}
    \begin{align}
    \varphi_0 &= \mathrm{id}_M, \label{eq:semiflow1}\\ 
    \varphi_t \circ \varphi_s &= \varphi_{t+s}, \qquad \text{for all } t,s \geq 0. \label{eq:semiflow2}
\end{align}\end{subequations}
If \eqref{eq:semiflow2} holds for all $t,s \in \mathbb{R}$, $\varphi$ is called a \textit{flow}.
A \textit{local semiflow} is a map $\varphi:\mathcal{D}^\varphi \to M$ satisfying \eqref{eq:semiflow} wherever defined, where  $\mathcal{D}^\varphi \subset [0,\infty) \times M$ is an open subset such that $(0,p) \in \mathcal{D}^\varphi$ for all $p \in M$.
Let $\mathcal{D}_t^\varphi : = \{p \in M \, | \, (t,p) \in \mathcal{D}^\varphi \}$ and $\mathcal{D}_p^\varphi : = \{t \in [0,\infty) \, | \, (t,p) \in \mathcal{D}^\varphi \}$.  
The semiflow is called \textit{maximal} if $t \in \mathcal{D}^\varphi_p$ and $s\in \mathcal{D}^\varphi_{\varphi_t(p)}$ imply that $s+t \in \mathcal{D}^\varphi_p$.

A \textit{semigroup} (of operators on a Banach space $(X,|\cdot|)$) is a map $T: [0,\infty) \to \mathcal{L}(X)$ satisfying the semiflow property \eqref{eq:semiflow}.
A semigroup is \textit{strongly continuous} if it moreover satisfies
\begin{displaymath}
    \lim_{t \searrow 0}\vert T(t) x - x \vert = 0, \qquad \text{for all } x \in X.
\end{displaymath}
The \textit{infinitesimal generator} of $T$ is the linear operator $A : \mathrm{dom}(A) \to X$ given by
\begin{equation}
    Ax = \lim_{t \searrow 0} \frac1t (T(t) x - x)
    \label{eq:generator_sg}
\end{equation}
on its domain 
\begin{displaymath}
    \dom(A) = \left\{x \in X \; \Big| \; \text{the limit } \lim_{t \searrow 0} \frac1t (T(t) x - x) \text{ exists} \right\}.
\end{displaymath}
If $T$ is strongly continuous, $A$ is necessarily a closed operator and its domain $\dom(A)$ is dense in $X$. 

Given an unbounded operator $A:\dom(A) \to X$ with a dense domain $\dom(A) \subset X$, 
its dual operator,
\begin{displaymath}
    A^* : \dom(A^*) \to X^*, \qquad \dom(A^*) \subset X^*,
\end{displaymath}
is defined via its domain, the linear subspace
\begin{equation}
    \dom(A^*) = \left\{ x^* \in X^* \; \bigg| \; \parbox{7cm}{there exists a constant $c \geq 0$ such that \\ $|\langle x^*,Ax\rangle| \leq c |x|_X$ for all $x \in \dom(A)$}  \right\},
    \label{eq:dual_domain}
\end{equation}
on which its action is determined by the relation
\begin{displaymath}
    \langle A^* x^*,x \rangle = \langle x^*,Ax\rangle, \qquad \text{for all } x \in \dom(A). 
\end{displaymath}

\subsection{Weak* integration}
\label{sect:weakstar_int}

Suppose $X$ is a Banach space, denote its dual by $X^*$.
Let $I \subset \mathbb{R}$ be an interval and let $f: I \to X^{ *}$ be such that
\begin{equation}
    \langle f( \cdot ),x\rangle_{X^{ *},X } \in L^1(I,\mathbb{C}), \qquad \text{for all } x \in X.
    \label{eq:weakstarintreq}
\end{equation}
Then,
\begin{displaymath}
    \int_I f(t) \,dt
\end{displaymath}
is defined as the unique element of $X^{ *}$ satisfying
\begin{displaymath}
    \left\langle \int_I f(t) \,dt,x \right\rangle_{X^{ *},X} = \int_I \langle f(t),x\rangle_{X^{ *},X} \, dt  \qquad \text{for all } x \in X.
\end{displaymath}
See A.II.3.13 of \cite{diekmann2012delay}, for instance, for an argument substantiating this claim using the Closed Graph Theorem.

In the main text, weak* integration is typically used in the context when (using the notation of Section~\ref{sect:DDEs})
\begin{displaymath}
    f(t,s) = T_0^{\odot *}(t-s) F \circ \varphi_s(u), \qquad \text{for some } u \in O.
\end{displaymath}
Since $f(t,\cdot):[0,t] \to X^{\odot *}$ is norm continuous for any $t \in \mathcal{D}^\varphi_u$, \eqref{eq:weakstarintreq} is clearly met.
Moreover, $t \mapsto \int_0^tf(t,s) \, ds$ is norm continuous and takes values in $\imath(X)$, see Lemma III.2.1, \cite{diekmann2012delay}.

\subsection{Notions of smoothness}
\label{sect:notions_of_smoothness}

Let us gather the notions of smoothness used throughout this work.

Let $(X,|\cdot|_X)$ and $(Y,|\cdot|_Y)$ be Banach spaces.
Let $f : O \to Y$ denote a map, with $O \subset X$ open.
We say $f$ is \textit{locally Lipschitz} if for every $u \in O$, there exists an open neighbourhood $U$ of $u$ and a constant $L \geq 0$ for which $|f(v)-f(w)|_Y \leq L |v-w|_X$ for all $v,w \in U$.
We say $f$ is \textit{globally Lipschitz} if $U = O$ in the above.
In this case, we denote by $\lip(f)$ the global Lipschitz constant.

$f$ is called \textit{Fréchet differentiable} at $u \in O$ if there exists a bounded linear operator $A \in \mathcal{L}(X;Y)$ such that
\begin{displaymath}
    \lim_{|\delta|_X \to 0} \frac{| f(u+\delta)-f(u)-A\delta|_Y}{|\delta|_X} = 0.
\end{displaymath}
In this case, we denote $Df(u) = A$.
If $f$ is Fréchet differentiable for all points in $O$, it is said to be of class $C^1$ if the map $u \mapsto Df(u)$ is continuous from $O $ to $\mathcal{L}(X;Y)$.
Inductively, $f$ is of class $C^k$, $k \in \mathbb{N}$ if $Df$ is of class $C^{k-1}$.
$f$ is said to be of class $C^\infty$ if it is $C^k$ for all $k \in \mathbb{N}$.

$f$ belongs to the space $C^k_b(O)$ if it is $C^k$, and moreover, $u \mapsto D^jf(u)$ is bounded on $O$ for each $0 \leq j \leq k$.
The space $C^k_b(O)$ is Banach with the norm $\sup \{ \Vert D^jf(u) \Vert \; | \; u \in O, \; 0\leq j \leq k \}$.

\section{Smoothness of the semiflow}
\label{sect:semiflow_smoothness}

Here, we derive, for completeness, $C^k$-in-space smoothness properties of the semiflow generated by a delay equation \eqref{eq:DDE_classical} with $f \in C^k(X;\mathbb{R}^n)$, for integers $k \geq 2$.

Recall the following result about the smoothness of composition maps.

\begin{theorem}[Theorem 6, \cite{Irwin72}] \label{thm:irwin}
    Let $A$ be a compact topological space and let $Y,Z$ denote Banach spaces, and $U \subset Y$ an open subset.
    Suppose $g : U \to Z$ is a $C^k$ map, $k \in \mathbb{N}_0$.
    Then $\hat{g} : C(A;U) \to C(A;Z)$ given by $\hat{g}(f) = g \circ f$ is $C^k$.
\end{theorem}

We remark that the first derivative $D \hat{g} : C(A;U) \to \mathcal{L}\big( C(A;Y);C(A;Z) \big)$ is given by
\begin{displaymath}
    \big(D\hat{g} (f) [\eta] \big) (x) = Dg(f(x))[\eta(x)], \qquad f \in C(A;U), \; \eta \in C(A;Y).
\end{displaymath}
Thus, 
\begin{equation}
    \Vert D\hat{g} (f) \Vert \leq \sup_{x \in A} \Vert Dg(f(x)) \Vert.
    \label{eq:Dghat}
\end{equation}

Recall also that for an open subset $U$ of a Banach space $Y$, we have that $C(A;U)$ is an open subset of $C(A;Y)$. 

\begin{proposition} \label{prop:sf_smoothness}
    Suppose $f$ in \eqref{eq:DDE_classical} is $C^k$, with $k \in \mathbb{N}$.
    Let $\varphi: \mathcal{D}^\varphi \to O $ denote the semiflow from Proposition~\ref{prop:semiflow}.
    Given $u \in O$, there exists $t >0$ small enough and a neighbourhood $U \subset \mathcal{D}_t^\varphi$ of $u$ such that $\varphi_t|_{U}$ is $C^k$.
    Moreover, $(s,u)\mapsto D^j \varphi_s(u)[v_1,\ldots,v_j]$ is jointly continuous on $ [0,t] \times U$ for each integer $j \leq k$, and fixed $v_i \in X$, $1 \leq i \leq j$.
\end{proposition}

\begin{proof}
    Let $M, \omega$ be as in the proof of Lemma~\ref{lemma:globalexist}.
    By continuity of $\varphi$ (Proposition~\ref{prop:semiflow}\ref{sfitem:1}) and continuity of $v \mapsto DR(v)$, we may choose $t > 0$ small enough so that
    \begin{equation}
        \frac{M}{\omega} \left( e^{\omega t} - 1 \right) \sup_{s \in [0,t]} \Vert DR (\varphi_s(u)) \Vert < 1.
        \label{eq:smoothnesspropcondition}
    \end{equation}

    Consider the map
    \begin{gather*}
        \mathcal{F} : \mathcal{D}^\varphi_t \times C([0,t];O) \to  C([0,t];X), \\
        (u,y) \mapsto \left( s \mapsto y(s) -  T(s)u - \imath^{-1} \int_0^s T^{\odot *}(s-\tau) R (y(\tau)) \, d\tau \right).
    \end{gather*}
    Note that $\mathcal{F}$ is well defined and indeed has image in $C([0,t];X)$, by Lemma III.2.1, \cite{diekmann2012delay}. 
    Then, by Proposition~\ref{prop:semiflow}, since $u \in \mathcal{D}^\varphi_t$,
    \begin{displaymath}
        \mathcal{F}\big(u, s \mapsto \varphi_s(u) \big) = 0.
    \end{displaymath}
    Moreover, the map $\mathcal{F}$ can be written as
    \begin{displaymath}
        \mathcal{F} = L_1 - L_3 \circ \widehat{R} \circ L_2,
    \end{displaymath}
    where 
    \begin{align*}
        &L_1: \mathcal{D}^\varphi_t \times C([0,t];O) \to  C([0,t];X),& &L_1 (u,y) (s) = y(s) - T(s)u, \\
        &L_2 : \mathcal{D}^\varphi_t \times C([0,t];O) \to C([0,t];O), & &L_2(u,y) = y, \\
       &L_3 :C([0,t];X^{\odot *}) \to  C([0,t];X), & &L_3(z)(s) = \imath^{-1} \int_0^s T^{\odot *}(s-\tau) R(z(\tau)) \, d \tau.
    \end{align*}
    All of these are continuous linear maps or restrictions thereof.
    Hence, by Theorem~\ref{thm:irwin}, $\mathcal{F}$ is of class $C^k$.
    Moreover, 
    \begin{displaymath}
        D_2 \mathcal{F}\big(u,s \mapsto \varphi_s(u)\big) = \mathrm{id}_{C([0,t];X)} -L_3 \circ D \widehat{R} \big(s \mapsto \varphi_s(u)\big) .
    \end{displaymath}
    Since, by \eqref{eq:Dghat},
    \begin{displaymath}
        \left\Vert L_3 \circ D \widehat{R} \big(s \mapsto \varphi_s(u)\big) \right\Vert \leq \frac{M}{\omega} \left( e^{\omega t} - 1 \right) \sup_{s \in [0,t]} \Vert DR (\varphi_s(u)) \Vert,
    \end{displaymath}
    \eqref{eq:smoothnesspropcondition} implies that $D_2 \mathcal{F}\big(u,s \mapsto \varphi_s(u)\big)$ is a linear isomorphism.

    We may hence apply the Implicit Function Theorem (Theorem I.5.9, \cite{lang2012fundamentals}) to obtain a neighbourhood $U \subset  \mathcal{D}^\varphi_t$ of $u$ and a unique $C^k$ map $\mathcal{G}:U \to C([0,t];O)$ such that
    \begin{displaymath}
        \mathcal{F}(w, \mathcal{G}(w)) = 0 \qquad \text{for all }  w \in U.
    \end{displaymath}
    By uniqueness, and since $U \subset \mathcal{D}^\varphi_t$, 
    we necessarily have $\mathcal{G}(w) = s \mapsto \varphi_s(w)$.
    Composing $\mathcal{G}$ with the (continuous, linear) evaluation map $\mathrm{ev}_t : y \mapsto y(t)$, we have
    \begin{displaymath}
        \varphi_t = \mathrm{ev}_t \circ \mathcal{G} \qquad \text{on } U.
    \end{displaymath}
    This yields the first claim immediately.
    For the second, observe that, for $u \in U$ and $v_i \in X$, $1 \leq i \leq j $, $j \leq k$,
    \begin{displaymath}
        D^j\varphi_s(u)[v_1,\ldots,v_j] = \mathrm{ev}_s \big(  D^j \mathcal{G}(u)[v_1,\ldots,v_j]\big), \qquad s \in [0,t].
    \end{displaymath}
    Note, in particular, that $D^j \mathcal{G}(u)[v_1,\ldots,v_j]\in C([0,t];X)$.
    Hence, for $(s,u),(\tau,w) \in [0,t] \times U$,
    \begin{align*}
        \big| D^j\varphi_s(u)[v_1,\ldots,v_j] - &D^j\varphi_\tau(w)[v_1,\ldots,v_j] \big|_X  \\
        &\leq \left| D^j \mathcal{G}(u)[v_1,\ldots,v_j] - D^j \mathcal{G}(w)[v_1,\ldots,v_j] \right|_{C([0,t];X)} \\
        &+\left|  \left( D^j \mathcal{G}(u)[v_1,\ldots,v_j] \right)(s) - \left( D^j \mathcal{G}(u)[v_1,\ldots,v_j] \right)(\tau)  \right|_{X},
    \end{align*}
    from which the second assertion follows. 
\end{proof}

\begin{proposition}
    Let $t > 0$ be fixed, and assume the conditions of Propositon~\ref{prop:sf_smoothness}.
    Then $\varphi_t$ is $C^k$ on $\mathcal{D}_t^\varphi$. 
\end{proposition}

\begin{proof}
    Fix $u \in \mathcal{D}_t^\varphi$ and set $T_1 = 0$.
    For each $j \in \mathbb{N}$, we may select $t_j > 0$ inductively such that
    \begin{equation}
       \frac{1}{2} < \frac{M}{\omega} \left( e^{\omega t_j} - 1 \right) \sup_{s \in [T_j,T_j+t_j]} \Vert DR (\varphi_s(u)) \Vert < 1,
       \label{eq:contrad}
    \end{equation}
    where $T_j = \sum_{i < j} t_i$ for $j \geq 2$
    (using the axiom of dependent choice).    

    Since $(T_j)_{j\in \mathbb{N}}$ is a (strictly) increasing sequence, it either tends to $\infty$ or converges to a limit $T$.
    We claim $T > t$.
    For if, on the contrary, $T \leq t$, we have by \eqref{eq:contrad} and the fact that necessarily $t_j \to 0$,
    \begin{displaymath}
        \lim_{s\to T}\Vert DR (\varphi_s(u)) \Vert  = \infty,
    \end{displaymath}
    which would contradict that $R$ is $C^1$ at $\varphi_T(u) \in O$ (since $T \in \mathcal{D}^\varphi_u$).

    Hence, $T >t$, and we may choose $m \in \mathbb{N}$ large enough so that $T_{m+1} \geq t > T_m$.
    Up to perhaps shrinking $t_m$, we may assume $T_{m+1} = t$.
    For each $ 1 \leq j \leq m$, Proposition~\ref{prop:sf_smoothness} provides a neighbourhood $U_j \subset \mathcal{D}^\varphi_{t_j}$ of $\varphi_{T_j}(u)$ on which $\varphi_{t_j}$ is $C^k$.
    Then,
    \begin{displaymath}
        V :=\bigcap_{j =1}^{m} \varphi_{T_j}^{-1}(U_j)
    \end{displaymath}
    is a nonempty open neighbourhood of $u$ in $\mathcal{D}^\varphi_t$ (by maximality of $\varphi$) on which
    \begin{displaymath}
        \varphi_t = \varphi_{t_m} \circ \cdots \circ \varphi_{t_1}
    \end{displaymath}
    is $C^k$.
\end{proof}

\begin{remark}[Smoothness of $\varphi_1^\rho$ about $W^\Sigma_\rho$] \label{remark:cutoff_semiflow_smoothness}
    While $R_\rho$ is not smooth in general, it retains the $C^k$ smoothness of $R$ on the open set $S$ defined by \eqref{eq:smoothnessset}.
    Moreover, $C([0,t];S)$ is an open subset of $C([0,t];X)$; $W^\Sigma_\rho$ is an invariant set for $\varphi_t^\rho$ contained within $S$ for $\rho$ sufficiently small\footnote{This can be ensured via Theorem~\ref{thm:CHT}\ref{extralipschitzstatement}, which gives control over $\lip(\phi_\rho)$. If this is chosen to be less than $2$, $W^\Sigma_\rho \subset S$ according to \eqref{eq:smoothnessset}.}, and $s \mapsto \varphi_s^\rho(u) \in C([0,t];S)$ for all $u \in W^\Sigma_\rho$. 
    Therefore, the proof of Proposition~\ref{prop:sf_smoothness} goes through with $\mathcal{F}_\rho : X \times C([0,t];S) \to C([0,t];X)$ in place of $\mathcal{F}$ (with the obvious modification $R \mapsto R_\rho$).
    In particular, $\mathcal{F}_\rho$ is $C^k$ on this restricted space, and the above discussion shows that each $u \in W^\Sigma_\rho$ produces a solution within the domain of $\mathcal{F}_\rho$ of the form $\mathcal{F}_\rho (u, s \mapsto \varphi^\rho_s(u)) =0$.
    As in the proof of Proposition~\ref{prop:sf_smoothness}, the Implicit Function Theorem then produces $C^k$ smoothness of $\varphi_1^\rho$ on neighbourhoods of $u \in W^\Sigma_\rho$, where $\rho >0$ is small enough so that \eqref{eq:smoothnesspropcondition} holds with $t =1 $ and $R = R_\rho$, the supremum now being taken over the whole of $W^\Sigma_\rho$.
\end{remark}

\begin{remark}[Joint continuity property for Theorem~\ref{thm:main2}] \label{remark:jointcont}
    In the setting of Section~\ref{sect:linearization}, we have that $\varphi_t(0) = 0$ for all $t \geq 0$; and $DR(0) = 0$.
    Hence, \eqref{eq:smoothnesspropcondition} is satisfied for all $t \geq 0$ if $u$ is taken to be $0$.
    The last assertion of Proposition~\ref{prop:sf_smoothness} then implies that, for any $t > 0$, there exists $U \subset \mathcal{D}_t^\varphi$ containing $0$ such that $(s,u) \mapsto D^j\varphi_s(u)[v_1,\ldots,v_j]$ is jointly continuous on $[0,t] \times U$.
    Hence,  $(s,u) \mapsto D^j\varphi_s(u)[v_1,\ldots,v_j]$ is jointly continuous over $\mathbb{R}^{\geq 0} \times \{0\}$.
    Suppose $W$ is a finite dimensional submanifold embedded in $X$; denote the embedding by $\imath_W: W \xhookrightarrow{} X$.
    Recall that 'strong' continuity in the above sense implies 'norm' continuity in the finite dimensional case, hence $(s,u_1) \mapsto D^j(\varphi_s \circ \imath_W )(u_1) = D^j \varphi_s (u_1) [\imath_W (\cdot),\ldots,\imath_W(\cdot)]$ is jointly continuous over $\mathbb{R}^{\geq 0} \times \{0\}$. 
\end{remark}

\section{Smoothness of \texorpdfstring{$W^\Sigma_\rho$}{}}
\label{sect:mfdsmoothness}

Here, we outline the proof of the smoothness claim of Theorem~\ref{thm:main1}\ref{thm1st2}, following the 'scale of Banach spaces' approach of \cite{vanderbauwhede1987center}, designed originally for center manifolds.
The discrete-time version, known as Irwin's method \cite{irwin1980new,de1995irwin}, would produce the same smoothness conclusions.
The latter, specifically, would be directly applicable to the setting of Section~\ref{sect:proofmain1} if $R_\rho$ was $C^k$ on its whole domain (using uniqueness of $W^\Sigma_\rho$ once a cutoff function has been fixed).
If one traces through the proof of Theorem 2.1 in \cite{de1995irwin}, it is apparent (see Lemma 3.4 specifically) that $C^k$ smoothness of the map $\varphi^\rho_1$ is only used on a neighbourhood of negative semiorbits contained in $W^\Sigma_\rho$ -- which we do indeed have (c.f.\ Remark~\ref{remark:cutoff_semiflow_smoothness}).

We  include a proof anyway so as to avoid concealing any details. 
It will moreover serve to provide proofs of the assertions made in Remark~\ref{remark:strongmfds}.
We assume throughout this section that $f$ is $C^k$, $k \geq 1$.

First, we remark that the projections $P_\Sigma$ of Lemma~\ref{lemma:proj} can be extended to the larger space $X^{\odot *}$ via $P^{\odot *}_\Sigma :=(P_\Sigma^*|_{X^\odot})^*$ -- these will have the same range and satisfy $\imath P_\Sigma =P^{\odot *}_\Sigma \imath$.
(They moreover coincide with spectral projections of the operator $A^{\odot *}$.) 
We shall only make use of $P^{\odot *}_\Sigma$ in the special case when $\Sigma$ corresponds to a half space of the spectrum, as in \ref{A3}, for which we can directly quote the following result.

\begin{lemma}[Theorem 2.12, \cite{diekmann2012delay}] \label{lemma:appendixdichotomy}
    Let $A$ be as in \eqref{eq:A}, let $T(t)$ denote the semigroup it generates.
    If 
    \begin{displaymath}
         \Sigma = \{ \lambda \in \sigma(A;X) \; | \; \mathrm{Re} \, \lambda \geq \gamma \},
    \end{displaymath}
    then $\Sigma = \{ \lambda \in \sigma(A^{\odot *};X) \; | \; \mathrm{Re} \, \lambda \geq \gamma \}$,
    \begin{displaymath}
        X^{\odot *} = \im (P^{\odot *}_\Sigma) \oplus \ker (P^{\odot *}_\Sigma)
    \end{displaymath}
    and, if $\alpha :=\inf_{\lambda \in \Sigma} \mathrm{Re} \, \lambda,$ and $\beta := \sup_{\lambda \in \Sigma'} \mathrm{Re} \, \lambda$,
    for any $\varepsilon > 0$ there exists $K > 0$ such that
     \begin{subequations}\label{eq:dichotomy_appendix}
     \begin{align}
         \Vert T^{\odot *}(t) P_\Sigma^{\odot *} \Vert_{\mathcal{L}(X^{\odot *})} &\leq K e^{(\alpha -\varepsilon)t} \Vert P_{\Sigma}^{\odot *} \Vert_{\mathcal{L}(X^{\odot *})}, \qquad &t \leq 0,  \label{eq:dichotomyapp1}\\ 
         \Vert T^{\odot *}(t)P_{\Sigma'}^{\odot *} \Vert_{\mathcal{L}(X^{\odot *})} &\leq K e^{(\beta+\varepsilon) t} \Vert P_{\Sigma'}^{\odot *}\Vert_{\mathcal{L}(X^{\odot *})}, \qquad &t \geq 0. \label{eq:dichotomyapp2} 
     \end{align}    
     \end{subequations}
\end{lemma}

To characterize backwards orbits, we introduce the following Banach spaces (see  \cite{vanderbauwhede1987center}).

\begin{definition} \label{def:BC}
    For each $\eta \in \mathbb{R}$, 
\begin{displaymath}
    BC^\eta (\mathbb{R}^{\leq 0};X) = \left\{ f \in C(\mathbb{R}^{\leq 0};X) \; \Big| \; \sup_{t \in \mathbb{R}^{\leq 0}} e^{\eta t} |f(t)|_X < \infty\right\}
\end{displaymath}
is a Banach space when equipped with the norms
\begin{displaymath}
    | f |_\eta = \sup_{t \in \mathbb{R}^{\leq 0}} e^{\eta t} |f(t)|_X.
\end{displaymath}
\end{definition}

Note that for $ \tilde{\eta} \leq \eta$ there exists a continuous embedding $\jmath_{ \tilde{\eta}}^\eta : BC^{\tilde{\eta}} (\mathbb{R}^{\leq 0};X) \xhookrightarrow{} BC^\eta (\mathbb{R}^{\leq 0};X)$ with norm one.
We shall occasionally use the shorthand notation $X_\eta := BC^\eta (\mathbb{R}^{\leq 0};X)$ in sub/superscripts.

\begin{lemma} \label{lemma:mathcalK}
    The operator defined by
    \begin{align}
        [\mathcal{K} F](t) =- \imath^{-1} \int_t^0 &T^{\odot *} (t-s) P_{\Sigma}^{\odot *} F(s) \, ds \nonumber\\
         &+ \imath^{-1} \int_{-\infty}^t T^{\odot *} (t-s) P_{\Sigma'}^{\odot *} F(s) \, ds, \qquad t \leq 0, \label{eq:mathcalK}
    \end{align}
    is bounded as a map $\mathcal{K}: BC^{-\eta} (\mathbb{R}^{\leq 0};X^{\odot *}) \to BC^{-\eta} (\mathbb{R}^{\leq 0};X)$ for each $\eta \in (\beta,\alpha)$.
    Moreover, $\Vert \mathcal{K} \Vert_{\mathcal{L}(X_{-\eta}^{\odot *};X_{-\eta})}$ can be bounded uniformly with respect to $\eta$ on compact subsets of $(\beta,\alpha)$. 
\end{lemma}

\begin{proof}
    Fix $\eta \in (\beta,\alpha)$ and $F \in BC^\eta (\mathbb{R}^{\leq 0};X^{\odot *})$.
    Choose $0 <\varepsilon < \min \{\eta-\beta,\alpha-\eta\}$.
    For $t \leq 0$, we have, by \eqref{eq:dichotomy_appendix},
    \begin{align*}
        e^{-\eta t}\left|[\mathcal{K} F](t) \right|_X &\leq C |F|_{-\eta} \left( \int_t^0 e^{( \alpha - \varepsilon-\eta)(t-s)} \, ds +  \int_{-\infty}^t e^{( \beta + \varepsilon-\eta)(t-s)} \, ds  \right) \\ 
        &\leq C |F|_{-\eta} \left( \frac{1}{\alpha - \varepsilon-\eta} - \frac{1}{\beta + \varepsilon-\eta} \right),
    \end{align*}
    for some $C > 0$.
    The uniformity claim, for a compact subset $[\tilde{\eta},\eta] \subset (\beta,\alpha)$, follows by choosing $\varepsilon < \min\{\tilde{\eta}-\beta,\alpha-\eta\}/2$.
\end{proof}

We need the following result on composition operators acting on $BC^\eta$ spaces.

\begin{lemma}[Lemma A.IV.1.1, \cite{diekmann2012delay}] \label{lemma:compositionBC}
    Let $Y$ and $Z$ be two Banach spaces.
    Let $g:Y \to Z$ be a $C^k_b$ mapping with  $g(0) = 0$, $k \geq 1$.
    Then, the map $\hat{g} : f \mapsto g \circ f$ satisfies the following.
    \begin{enumerate}[label=\upshape{(\roman*)}]
        \item \label{item_comp1} For $\eta > 0$, the map $\hat{g}: BC^{-\eta} (\mathbb{R}^{\leq 0};Y) \to BC^{-\eta} (\mathbb{R}^{\leq 0};Z)$ is $C^k$.
        \item \label{item_comp2} For $\eta, \tilde{\eta} < 0$ such that $\tilde{\eta} <k\eta$, the map $ \hat{g}: BC^{-{\eta}} (\mathbb{R}^{\leq 0};Y) \to BC^{-\tilde{\eta}} (\mathbb{R}^{\leq 0};Z)$ is $C^k$.
    \end{enumerate}
    In both cases, for $f,h \in BC^{-{\eta}} (\mathbb{R}^{\leq 0};Y) $,
    \begin{equation}
       \big( D \hat{g}(f) [h] \big) (s) = Dg(f(s)) [h(s)].
       \label{eq:derivRhat}
    \end{equation}
\end{lemma}

\begin{proof}
    The proof is given in Appendix IV of \cite{diekmann2012delay} for the spaces $BC^\eta(\mathbb{R};Y)$ (equipped with norms $\sup_{t \in \mathbb{R}} e^{-\eta|t|} |f(t)|_Y$, for a Banach space $Y$).
    The proof for our case would be virtually the same, but the result can also be deduced directly, by considering the embedding $\imath_Y : BC^{\eta} (\mathbb{R}^{\leq 0};Y) \xhookrightarrow{} BC^\eta(\mathbb{R};Y)$ given by mirroring about the origin ($[\imath_Yf](t) : = f(-|t|)$, $t \in \mathbb{R}$) and noting that $\widehat{g}$ preserves its image; that is, $\imath_{Z} \circ \hat{g} =  \hat{g}\circ\imath_Y$.
\end{proof}

Invariant manifolds will be extracted as fixed points of (some variant of) the map
\begin{equation}
    \mathcal{F}(u_1,y) (t) = T(t) u_1 + [\mathcal{K} \circ \widehat{R} (y)](t).
    \label{eq:contraction}
\end{equation}
Here, $u_1 \in X_\Sigma = \im(P_\Sigma^{\odot *})$; but the exact $BC$ space on which $y$ is considered depends on the number $\gamma$ (from Lemma~\ref{lemma:appendixdichotomy}) defining $\Sigma$, since $\mathcal{K}$ needs to act boundedly on the image of $\widehat{R}$  (c.f.\ Lemma~\ref{lemma:mathcalK}).
We first deal with the simplest case, when $\gamma > 0$, for which we will be able to show $C^k$ smoothness without any spectral gap assumptions (hence proving the statements of Remark~\ref{remark:strongmfds}).

\subsection{The strong-unstable case}

\begin{lemma} \label{lemma:Rsmoothness}
    Suppose $\eta>0$ and that $R:O \to X^{\odot*}$ is as in Section~\ref{sect:linearization}, of class $C^k$, $k \geq 1$.
    For $\delta > 0$ sufficiently small, the map $\widehat{R}:B_\delta^{^{X_{-\eta}}}(0) \to BC^{-\eta} (\mathbb{R}^{\leq 0};X^{\odot *})$ is $C^k$.
    Moreover, $\sup_{f \in B_\delta^{^{X_{-\eta}}}(0)}\Vert D \widehat{R}(f) \Vert_{\mathcal{L}(X_{-\eta};X_{-\eta}^{\odot *})} = o(1)$ as $\delta \to 0$.
\end{lemma}

\begin{proof}
    Choose $\delta>0$ small enough so that $B_\delta^X(0) \subset O$ and $R|_{B_\delta^X(0)} \in C^k_b$.
    Since $\eta > 0$, we have $\cup_{t\leq0} \mathrm{ev}_t (B_\delta^{^{X_{-\eta}}}(0)) \subset B_\delta^X(0)$, and
    hence the assumptions of Lemma~\ref{lemma:compositionBC} are  met on the open set $B_\delta^{^{X_{-\eta}}}(0)$. 
    Noting that smoothness is a local property, 
    we may conclude that $\widehat{R}$ is $C^k$ on $B_\delta^{X_{-\eta} }(0)$.
    The bound is an immediate consequence of \eqref{eq:derivRhat}, the inclusions above, and 
    \begin{displaymath}
        \sup_{|u| \leq \delta} \Vert D R(u) \Vert = o(1) \qquad \text{as } \delta \to 0,
    \end{displaymath}
    (c.f.\ Section~\ref{sect:linearization}).
\end{proof}

\begin{proof}[Partial proof of Remark~\ref{remark:strongmfds}]
Suppose $\gamma > 0$ and choose $\eta \in (\min\{0,\beta\},\alpha)$.
Consider \eqref{eq:contraction} as $\mathcal{F}: X_\Sigma \times B_\delta^{X_{-\eta}}(0) \to BC^{-\eta} (\mathbb{R}^{\leq 0};X)$.
This is well defined and $C^k$ for $\delta$ small enough by Lemmas \ref{lemma:mathcalK} and \ref{lemma:Rsmoothness}.
In fact, we may ensure, by perhaps further adjusting $\delta$, that
\begin{equation}
    \Vert \mathcal{K} \Vert_{\mathcal{L}(X_{-\eta}^{\odot *};X_{-\eta})} \sup_{|u| \leq \delta} \Vert D R(u) \Vert < \frac{1}{2},
    \label{eq:Kcontraction1}
\end{equation}
so that $\mathcal{F}$ is a contraction on the second factor.
We may hence apply the Implicit Function Theorem to the map $(u_1,y) \mapsto y-\mathcal{F}(u_1,y)$ about the fixed point $\mathcal{F}(0,0) = 0$ to obtain a neighbourhood $U_\Sigma$ of $0$ in $X_\Sigma$ and a unique $C^k$ map
\begin{displaymath}
    \mathcal{W}: U_\Sigma \to B_\delta^{X_{-\eta}}(0)
\end{displaymath}
such that $\mathcal{F}(u_1,\mathcal{W}(u_1)) = \mathcal{W}(u_1)$.
The image of the $C^k$ map $\mathrm{ev}_0 \circ \mathcal{W} : U_\Sigma \to B_\delta^X(0)$ is then the desired 'strong-unstable' manifold.
It is uniquely characterized by backwards Lyapunov exponents of its orbits, manifest in the choice of $BC^{-\eta} (\mathbb{R}^{\leq 0};X)$, $\eta \in (\min\{0,\beta\},\alpha)$.

Next, we show that it agrees with $W^\Sigma_\rho$ from Theorem~\ref{thm:CHT} on a small enough neighbourhood of the origin.
For a given selection of $\gamma_1,\gamma_2 > 0$ in Theorem~\ref{thm:CHT}, we may consider $\eta$ on the interval $ [\gamma_2-\varepsilon,\gamma_1+\varepsilon]$, where $\varepsilon$ is chosen such that this interval lies within $(\min\{0,\beta\},\alpha)$.
By the final claim of Lemma~\ref{lemma:mathcalK}, we may ensure \eqref{eq:Kcontraction1} holds uniformly for  $\eta \in [\gamma_2-\varepsilon,\gamma_1+\varepsilon]$.

Since the spaces $BC^{-\eta}(\mathbb{R}^{\leq 0}; X)$, $\eta \in [\gamma_2-\varepsilon,\gamma_1+\varepsilon]$, are all embedded in $BC^{-(\gamma_2-\varepsilon)}(\mathbb{R}^{\leq 0}; X)$, we must have $\jmath^{-\eta}_{-(\gamma_2-\varepsilon)} \circ \mathcal{W}^{-(\gamma_2-\varepsilon)} = \mathcal{W}^{-\eta}$ by the uniqueness claim of the Implicit Function Theorem (where the superscript of $\mathcal{W}$ refers to the space in which it was obtained).

Comparing these observations with Theorem~\ref{thm:CHT}\ref{item1:CHT}; in particular, \eqref{eq:negative_semiorbit_1} and \eqref{eq:negative_semiorbit_2}, we must have that $\im(\mathrm{ev}_0 \circ \mathcal{W}) = W^\Sigma_\rho$
on a sufficiently small neighbourhood of the origin where both $\mathcal{W}$ and $W^\Sigma_\rho$ are defined and $\varphi = \varphi^\rho$.
\end{proof}

Strong-stable manifolds of Remark~\ref{remark:strongmfds} are proven analogously, with appropriate modifications to $\mathcal{K}$, and considering positive semiorbits.

\subsection{The pseudo-unstable case}

The case $\gamma \leq 0$ (when necessarily $\eta <0$ is required for the $BC$ space) is much more complicated; firstly because Lemma~\ref{lemma:Rsmoothness} no longer applies (this necessitates the application of a cutoff function), and secondly due to the difference in items \ref{item_comp1} and \ref{item_comp2} of Lemma~\ref{lemma:compositionBC}.
We first state a replacement of Lemma~\ref{lemma:Rsmoothness} for the modified nonlinearity $R_\rho$ from Section~\ref{sect:cutoff}.

\begin{lemma}
    For all $\eta \in \mathbb{R}$, the substitution operators 
    \begin{displaymath}
        \widehat{R}_\rho: BC^{\eta} (\mathbb{R}^{\leq 0};X)\to BC^{\eta} (\mathbb{R}^{\leq 0};X^{\odot *})
    \end{displaymath}
    defined by $\widehat{R}_\rho : f \mapsto R_\rho \circ f$
    are globally Lipschitz continuous with $\lip(\widehat{R}_\rho) \leq \lip(R_\rho) = o(1)$ as $\rho \to 0$.
\end{lemma}

\begin{proof}
    Observe, for $f,g \in  BC^{\eta} (\mathbb{R}^{\leq 0};X)$, that
    \begin{align*}
        \big| \widehat{R}_\rho(f) - \widehat{R}_\rho(g) \big|_{\eta} &= \sup_{t \leq 0} e^{\eta t} \left| R_\rho (f(t)) - R_\rho (g(t))\right|_{X^{\odot *}} \\
        &\leq \lip(R_\rho) \left| f-g \right|_{\eta}.
    \end{align*}
    This shows that $\widehat{R}_\rho$ is well defined and $\lip(\widehat{R}_\rho) \leq \lip(R_\rho)$. Lemma~\ref{lemma:Lipschitz} completes the proof.
\end{proof}

We may now produce a Lipschitz manifold via the same argument as above.
We shall consider, for later purposes, a range of $\eta \in [\tilde{\eta},\overline{\eta}] \subset (\beta,\alpha)$  with $k\overline{\eta} > \tilde{\eta}$, $k \geq 1$.
In particular, consider $\mathcal{F}_\rho: X_\Sigma \times BC^{-\eta} (\mathbb{R}^{\leq 0};X) \to BC^{-\eta} (\mathbb{R}^{\leq 0};X)$, with $\eta \in [\tilde{\eta},\overline{\eta}]$, given by
\begin{equation}
    \mathcal{F}_\rho(u_1,y) (t)=  T(t) u_1 + [\mathcal{K} \circ \widehat{R}_\rho (y)](t).
    \label{eq:Frho}
\end{equation}
Shrink $\rho$ so that
\begin{equation}
    \Vert \mathcal{K} \Vert_{\mathcal{L}(X_{-\eta}^{\odot *};X_{-\eta})} \lip(R_\rho) < \frac{1}{2}, \qquad \text{for all } \eta \in [\widetilde{\eta},\overline{\eta}].
    \label{eq:choice_of_rho}
\end{equation}
(This can be done by the last assertion of Lemma~\ref{lemma:mathcalK}.)
Since $\mathcal{F}_\rho$ is globally Lipschitz (in both entries), the (parametric) Banach fixed point theorem (Theorem 21, \cite{Irwin72}) implies the existence of a globally Lipschitz map $\mathcal{W}_\rho^{-\eta}:X_\Sigma \to BC^{-\eta} (\mathbb{R}^{\leq 0};X)$ such that 
\begin{equation}
    \mathcal{F}_\rho(u_1,\mathcal{W}^{-\eta}_\rho(u_1)) = \mathcal{W}_\rho^{-\eta}(u_1).
    \label{eq:FrhoW}
\end{equation}
Note that uniqueness of the solution implies $\jmath_{-\overline{\eta}}^{-\eta} \circ \mathcal{W}^{-\overline\eta}_\rho = \mathcal{W}^{-\eta}_\rho$ (since $\mathcal{K} \circ \widehat{R}_\rho$ commutes with $\jmath_{-\overline{\eta}}^{-\eta}$).
Uniqueness also implies the invariance property
\begin{equation}
    \mathrm{ev}_t \circ \mathcal{W}_\rho^{-\eta} = \mathrm{ev}_0 \circ  \mathcal{W}_\rho^{-\eta}  \circ \widetilde{P}_\Sigma \circ \mathrm{ev}_t \circ\mathcal{W}_\rho^{-\eta}, \qquad t \leq 0.
    \label{eq:invariance}
\end{equation}
We may estimate the Lipschitz constant of $\mathcal{W}_\rho^{-\eta}$ via
\begin{multline*}
    \big| \mathcal{W}^{-\eta}_\rho(u_1) - \mathcal{W}^{-\eta}_\rho(v_1) \big|_{-\eta } \leq \\ \Vert \mathcal{K} \Vert_{\mathcal{L}(X_{-\eta}^{\odot *};X_{-\eta})} \lip(R_\rho) \big| \mathcal{W}^{-\eta}_\rho(u_1) - \mathcal{W}^{-\eta}_\rho(v_1) \big|_{-\eta } + K |u_1-v_1|_X
\end{multline*}
to be $\lip\big( \mathcal{W}^{-\eta}_\rho\big)\leq 2K$,
with $K$ as in \eqref{eq:dichotomy_appendix} (for $\varepsilon < \alpha - \overline{\eta}$).
Observe, from \eqref{eq:mathcalK}, that $P_\Sigma \circ \mathrm{ev}_0  \circ \mathcal{K} = 0$.
This shows $\widetilde{P}_\Sigma \circ\mathrm{ev}_0 \circ \mathcal{W}_\rho^{-\eta} = \mathrm{id}_{X_\Sigma}$.
Hence, the manifold $\im(\mathrm{ev}_0 \circ \mathcal{W}^{-\eta}_\rho)$ can be alternatively written as the graph of the map
\begin{equation}
    \phi_\rho :=\widetilde{P}_{\Sigma'} \circ\mathrm{ev}_0 \circ \mathcal{W}^{-\eta}_\rho:X_\Sigma \to X_{\Sigma'}.
    \label{eq:phirho}
\end{equation}
From \eqref{eq:FrhoW}, we have  $\phi_\rho = \widetilde{P}_{\Sigma'} \circ\mathrm{ev}_0 \circ \mathcal{K} \circ \widehat{R}_\rho \circ \mathcal{W}_\rho^{-\eta}$ and
\begin{equation}
    \lip(\phi_\rho) \leq 2K \Vert P_{\Sigma'}\Vert_{\mathcal{L}(X)} \Vert \mathcal{K} \Vert_{\mathcal{L}(X_{-\eta}^{\odot *};X_{-\eta})} \lip(R_\rho), \qquad \text{for all } \eta \in [\widetilde{\eta},\overline{\eta}].
    \label{eq:lipphirho}
\end{equation}
Up to potentially adjusting $\rho$, we may hence assume $\lip(\phi_\rho) < 2$.
We have thus ensured that 
is contained in the open set $S$ on which $R_\rho$ is $C^k$ \eqref{eq:smoothnessset}. 
Since $\mathrm{supp}(R_\rho) \subset \overline{B_{4\rho}(0)}^X$, we moreover have
\begin{equation}
    \sup_{x \in X_\Sigma} |\phi_\rho(x)| < \rho
    \label{eq:phirhomax}
\end{equation}
for $\rho$ small enough.
With the cutoff function fixed, by the same argument as above in the proof of Remark~\ref{remark:strongmfds}, we have $\im(\mathrm{ev}_0 \circ \mathcal{W}_\rho) = W^\Sigma_\rho$ (which now holds globally), hence the notation $\phi_\rho$ is no mistake, it is indeed equal to the one in \eqref{eq:Wsigmarho}.

For $\eta \in [\widetilde{\eta},\overline{\eta}]$, let
\begin{displaymath}
    V^{-\eta} = \left\{ y \in BC^{-\eta}(\mathbb{R}^{\leq 0} ; X) \; \big| \; |P_{\Sigma'}y|_0 < \infty \right\}
\end{displaymath}
(recall $| \cdot |_0$ stands for the $BC^0$ norm).
$V^{-\eta}$ is a Banach space continuously embedded in $BC^{-\eta}(\mathbb{R}^{\leq 0} ; X)$ when equipped with the norm $|y|_{V^{-\eta}} = |P_\Sigma y|_{-\eta} + |P_{\Sigma'}y|_0$.
Consider the open subset
\begin{displaymath}
    V^{-\eta}_\rho = \left\{ y \in BC^{-\eta}(\mathbb{R}^{\leq 0} ; X) \; \big| \; |P_{\Sigma'}y|_0 < \rho \right\}.
\end{displaymath}
It follows from the invariance property \eqref{eq:invariance} that
\begin{displaymath}
    \mathrm{ev}_t \circ \mathcal{W}^{-\eta}_\rho \in \im \left( \mathrm{ev}_0 \circ \mathcal{W}_\rho^{-\eta} \right) = \gr(\phi_\rho), \qquad t \leq 0. 
\end{displaymath}
By \eqref{eq:phirhomax}, we have
\begin{equation}
    \im \left( \mathcal{W}_\rho^{-\eta}\right) \subset V^{-\eta}_\rho.
    \label{eq:image_contained_in_O0}
\end{equation}

The reason we cannot repeat the proof of the $\eta > 0$ case above is that now, according to Lemma~\ref{lemma:compositionBC}\ref{item_comp2}, the map $\widehat{R}_\rho$ and hence $\mathcal{F}_\rho$ are only smooth with differing domain and range spaces -- thus the implicit function theorem no longer directly applies.
Roughly speaking, this is a consequence of the spaces $BC^{-\eta}$ allowing growth (as $t \to -\infty$) for $\eta < 0$.
To show smoothness of $\mathcal{W}_\rho^{-\eta}$, we consider an abstract version of \eqref{eq:Frho}, and provide a version of implicit mapping that allows for such differences in domains and ranges (following \cite{vanderbauwhede1987center,diekmann2012delay}).

Let $(Y_0,|\cdot|_{Y_0})$, $(Y,|\cdot|_Y)$, $(Y_1,|\cdot|_{Y_1})$ and $(\Lambda,|\cdot|_{\Lambda})$ be Banach spaces such that 
\begin{displaymath}
    Y_0 \xhookrightarrow{J_0} Y \xhookrightarrow{J} Y_1
\end{displaymath}
with $J_0$, $J$ continuous embeddings.
Let $O_0$ denote an open subset of $Y_0$.
Consider the fixed point equation
\begin{equation}
   f(\lambda,y) = y
   \label{eq:f_fixedpoint}
\end{equation}
for some $ f : \Lambda \times Y \to Y$ satisfying the following hypotheses:
\begin{enumerate}[label =(Hf.\arabic*)]
    \item \label{Hf1} The function $g:(\lambda,y_0) \mapsto J \circ f(\lambda,J_0(y_0))$ is of class $C^1$ on $\Lambda \times O_0$.
    There exist mappings $f^{(1)} : \Lambda \times J_0(O_0) \to \mathcal{L}(Y)$ and $f^{(1)}_1 : \Lambda \times J_0(O_0) \to \mathcal{L}(Y_1)$ such that
    \begin{displaymath}
        D_2g(\lambda,y_0)[z_0] = Jf^{(1)} (\lambda,J_0(y_0))[J_0(z_0)] \qquad \text{for all } (\lambda,y_0,z_0) \in \Lambda \times O_0 \times Y_0
    \end{displaymath}
    and
    \begin{displaymath}
        Jf^{(1)} (\lambda,J_0(y_0))[y] = f^{(1)}_1(\lambda,J_0(y_0)) [J(y)] \qquad \text{for all } (\lambda,y_0,y) \in \Lambda \times O_0 \times Y.
    \end{displaymath}

    \item \label{Hf2} There exists some $\kappa \in [0,1)$ such that for all $y,\tilde{y} \in Y$ and $\lambda \in \Lambda$
    \begin{equation}
        |f(\lambda,y) - f(\lambda,\tilde{y}) |_Y \leq \kappa |y-\tilde{y}|_Y,
        \label{eq:kappacontraction}
    \end{equation}
    and for all $\overline{y} \in J_0(O_0)$ 
    \begin{displaymath}
        \Vert f^{(1)}(\lambda,\overline{y}) \Vert_{\mathcal{L}(Y)} \leq \kappa, \qquad \Vert f^{(1)}_1(\lambda,\overline{y}) \Vert_{\mathcal{L}(Y_1)} \leq \kappa.
    \end{displaymath}
    \item \label{Hf3} Let $\Psi : \Lambda \to Y$ denote the map associating to each $\lambda$ a solution of \eqref{eq:f_fixedpoint} (this exists by \eqref{eq:kappacontraction}).
    Suppose $\Psi = J_0 \circ \Phi$ for some continuous $\Phi: \Lambda \to Y_0$ with image contained in $O_0$.
    \item \label{Hf4} $f_0 :\Lambda \times Y_0 \to Y$, $(\lambda,y_0) \mapsto f(\lambda,J_0(y_0))$ has continuous partial derivative
    \begin{displaymath}
        D_1f_0: \Lambda \times Y_0 \to \mathcal{L}(\Lambda;Y).
    \end{displaymath}
    \item \label{Hf5} The mapping $(\lambda,y_0) \mapsto J \circ f^{(1)}(\lambda,J_0(y_0))$ is continuous on $\Lambda \times O_0$.
\end{enumerate}

\begin{lemma}[Lemma IX.6.7, \cite{diekmann2012delay}] \label{lemma:6.7}
    Assume \ref{Hf1}-\ref{Hf5} hold.
    Then $\Psi$ is locally Lipschitz continuous.
    Moreover, $J \circ \Psi$ is of class $C^1$ with $D( J \circ \Psi)(\lambda) = J \circ \mathcal{A}(\lambda)$, where $\mathcal{A}: \Lambda \to \mathcal{L}(\Lambda;Y)$ is the solution map of the equation 
    \begin{displaymath}
        A =  D_1f_0(\lambda,\Phi(\lambda)) + f^{(1)} (\lambda,\Psi(\lambda)) A 
    \end{displaymath}
    for $A$ (whose existence is guaranteed by \ref{Hf2}-\ref{Hf4}).
\end{lemma}

Let $Y_1,\ldots , Y_n$ and $Z$ be Banach spaces.  
By convention, we endow product spaces with the norm $|(y_1,\ldots,y_n)|_{Y_1 \times \ldots \times Y_n} = \sup \{|y_1|_{Y_1},\ldots,|y_n|_{Y_n}\}$.
We denote by $\mathrm{M}(Y_1,\ldots,Y_n;Z)$ the Banach space of bounded multilinear maps 
\begin{displaymath}
    Y_1 \times \ldots \times Y_n \longrightarrow Z
\end{displaymath}
equipped with the norm 
\begin{displaymath}
    \Vert N \Vert_{\mathrm{M}(Y_1,\ldots,Y_n;Z)} = \sup \big\{ | N(y_1, \ldots, y_n) |_Z \; \big| \; |(y_1, \ldots, y_n)|_{Y_1 \times \ldots \times Y_n} < 1 \big\}.
\end{displaymath}
We use the shorthand notation $\mathrm{M}_n(Y;Z)$ for the space $\mathrm{M}(Y,\ldots,Y;Z)$, where $Y$ appears $n \geq 0$ times, with the convention that $\mathrm{M}_0(Y;Z) = Z$.

Recall that $R_\rho$ is $C^k$ on $S$ (c.f.\ \eqref{eq:smoothnessset}).
Let $1 \leq j \leq k$, and  let $\xi_i < 0$ for $i =1,\ldots,j$, $\sum \xi_i = : \xi$, and $\eta < 0 $ denote real numbers.
We may define, if $\eta \leq \xi$, the map  $\widehat{R}_\rho^{(j)}: V^{-\sigma}_\rho \to M_j(BC^{-\xi_1}(\mathbb{R}^{\leq 0};X),\ldots,BC^{-\xi_j}(\mathbb{R}^{\leq 0};X);BC^{-\eta}(\mathbb{R}^{\leq 0};X^{\odot *}))$, $\sigma < 0$, by 
\begin{equation}
    \left( \widehat{R}_\rho^{(j)}(y)[h_1,\ldots,h_j] \right) (t):= D^jR_\rho(y(t))[h_1(t),\ldots,h_j(t)], \qquad t \leq 0.
    \label{eq:widehatR(j)}
\end{equation}
It is immediate that this is well defined and takes values in $M_j(X_{-\xi_1},\ldots,X_{-\xi_j};X^{\odot *}_{\eta})$ by $\eta \leq  \xi$.

\begin{lemma}\label{lemma:wideR} 
    Let $1 \leq j \leq k$ and $\xi_i < 0$ for $i =1,\ldots,j$, $\sum \xi_i = : \xi$. 
    Then we have
    \begin{enumerate}[label=\upshape{(\roman*)}]
        \item \label{wideR1} If $\eta \leq  \xi $, we have $\widehat{R}_\rho^{(j)}(y) \in M_j(X_{-\xi_1},\ldots,X_{-\xi_j};X^{\odot *}_{\eta})$ for all $y \in V^{-\sigma}_\rho$, $\sigma < 0$. The mapping $y \mapsto \widehat{R}_\rho^{(j)}(y)$ is continuous if $\eta < j \xi$.
        \item \label{wideR2} If $\eta < j \xi$, the map $\jmath_{-\xi}^{-\eta/j} \circ \widehat{R}_\rho|_{V_\rho^{-\xi}}: V_\rho^{-\xi} \to BC^{-\eta}(\mathbb{R}^{\leq};X^{\odot *})$ is of class $C^j$.
        \item \label{wideR3} Let $\Phi:X_\Sigma \to V_\rho^{-\sigma}$, $\sigma < 0$, be a $C^1$ map. Then the mapping $\widehat{R}_\rho^{(j)} \circ \Phi:X_\Sigma \to M_j(X_{-\xi_1},\ldots,X_{-\xi_j};X^{\odot *}_{\eta})$ is $C^1$, provided that $\eta < \sigma +  \xi$; its derivative is given as
        \begin{equation}
            D (\widehat{R}_\rho^{(j)} \circ \Phi)(u_1) [v_1,\ldots,v_j] (s)=  \widehat{R}_\rho^{(j+1)}(\Phi(u_1))[ v_1(s),\ldots,v_j(s),D\Phi(u_1) [\cdot ](s)]
            \label{eq:Rhatcompderiv}
        \end{equation}
    \end{enumerate}
\end{lemma}

\begin{proof}
    These can be inferred from the corresponding results in Section IX.7, \cite{diekmann2012delay}, as in the proof of Lemma~\ref{lemma:compositionBC}.
    For clarity, we provide a direct proof of
\end{proof}

Note that, by its very definition \eqref{eq:widehatR(j)}, $\widehat{R}_\rho^{(j)}$ is insensitive to the space $V^{-\sigma}_\rho$, in the sense that
\begin{equation}
    \widehat{R}_\rho^{(j)} (y)= \widehat{R}_\rho^{(j)} \circ \jmath_{-\sigma}^{-\tau} (y), \qquad \text{for } y \in  V^{-\sigma}_\rho, \text{ and } \tau<\sigma<0.
    \label{eq:domain_Rhat}
\end{equation}
For future reference, we also extend the definition of the $\jmath$ maps to $M_j$ spaces in the obvious fashion:
If $A \in M_j(X_\Sigma; BC^{-\xi}(\mathbb{R}^{\leq 0};X))$, then
\begin{equation}
    (\jmath_{-\xi}^{-\eta} \circ A)[v_1,\ldots,v_j] : = \jmath_{-\xi}^{-\eta} \big( A  [v_1,\ldots,v_j]\big).
    \label{eq:jmathMj}
\end{equation}
If $\widehat{R}_\rho^{(j)} : V_\rho^{-\sigma} \to M_j(X_{-\xi_1},\ldots,X_{-\xi_j};X^{\odot *}_{\eta})$
and $\zeta < \eta$, then the map
\begin{gather*}
    \jmath_{-\eta}^{-\zeta} \circ \widehat{R}_\rho^{(j)} : V_\rho^{-\sigma} \to M_j(X_{-\xi_1},\ldots,X_{-\xi_j};X^{\odot *}_{\zeta}) \\ 
    \jmath_{-\eta}^{-\zeta} \circ \widehat{R}_\rho^{(j)} (y)[v_1,\ldots,v_j] = \jmath_{-\eta}^{-\zeta} \left( \widehat{R}_\rho^{(j)} (y)[v_1,\ldots,v_j] \right) 
\end{gather*}
can be factored as
\begin{equation}
    \jmath_{-\eta}^{-\zeta} \circ \widehat{R}_\rho^{(j)} (y)[v_1,\ldots,v_j] = \widehat{R}_\rho^{(j)}(y)[ \jmath_{-\xi_1}^{-\zeta_1} \circ v_1,\ldots,  \jmath_{-\xi_j}^{-\zeta_j} \circ v_j]
    \label{eq:factorjmathRhat}
\end{equation}
for some choice of $\zeta_i$, $i = 1,\ldots,j$, with $\sum\zeta_i = \zeta$.

The proof of smoothness will be entirely analogous to Section IX.7 of \cite{diekmann2012delay} for center manifolds. The only difference in our setting is that we consider a different interval of $\eta$ possibly bounded away from $0$, and that we work with negative semiorbits only -- manifest in the spaces $BC(\mathbb{R}^{\leq 0};X)$.
It turns out that these changes barely modify the arguments required.

\begin{proof}[Proof of the smoothness assertion, Theorem~\ref{thm:main1}\ref{thm1st2}]
Let $\eta \in [\tilde{\eta},\overline{\eta})$.
To prove that $\phi_\rho$ is of class $C^1$, we apply Lemma~\ref{lemma:6.7} with
\begin{align*}
    &O_0 = V_\rho^{-\overline{\eta}}, \quad Y_0 = V^{-\overline{\eta}}, \quad Y = BC^{-\overline\eta}(\mathbb{R}^{\leq 0};X), \quad  Y_1 =  BC^{-\eta}(\mathbb{R}^{\leq 0};X), \\
    &\Lambda = X_\Sigma, \quad f(u_1,y) = \mathcal{F}_\rho(u_1,y) =T(\cdot)u_1 + \mathcal{K} \circ \widehat{R}_\rho(y), \quad u_1 \in \Lambda, \; y \in Y,\\
    & f^{(1)}(u_1,y) = \mathcal{K} \circ \widehat{R}^{(1)}_\rho(y) \in \mathcal{L}(Y) , \quad u_1 \in \Lambda, \; y \in J_0(O_0), \\
    & f^{(1)}_1(u_1,y) = \mathcal{K} \circ \widehat{R}^{(1)}_\rho(y) \in \mathcal{L}(Y_1), \quad u_1 \in \Lambda, \; y \in J_0(O_0), 
\end{align*}
with $J_0 $ being the embedding $V^{-\overline{\eta}} \xhookrightarrow{} BC^{-\overline\eta}(\mathbb{R}^{\leq 0};X)$ and $J = \jmath_{-\overline{\eta}}^{-\eta}$.
Note that we have obtained $\Phi$ from \ref{Hf3} as $\mathcal{W}_\rho^{-\overline{\eta}}$ in \eqref{eq:FrhoW}; it satisfies the requirements of \ref{Hf3} by \eqref{eq:image_contained_in_O0} (it is an elementary fact from point-set topology that $\mathcal{W}_\rho^{-\overline{\eta}}$ is continuous $X_\Sigma \to Y_0$, Theorem 18.2, \cite{munkres2000topology}). 
Let us verify the remaining assumptions of Lemma~\ref{lemma:6.7}.
For \ref{Hf1}, we use Lemma~\ref{lemma:wideR}\ref{wideR2} on the term $ \jmath_{-\overline{\eta}}^{-\eta} \circ \mathcal{K} \circ \widehat{R}_\rho \circ J_0 = \mathcal{K} \circ \jmath_{-\overline{\eta}}^{-\eta} \circ \widehat{R}_\rho \circ J_0$, which shows $g$ is $C^1$ -- differentiability in the $\Lambda = X_\Sigma$ direction is a consequence of $X_\Sigma \subset \dom(A)$ and $\eta < \alpha$ (c.f.\ Lemma~\ref{lemma:appendixdichotomy}).
We have
\begin{displaymath}
    D_2g (u_1,y_0)[z_0] =  \jmath_{-\overline{\eta}}^{-\eta} \circ \mathcal{K} \circ \widehat{R}_\rho^{(1)}(J_0(y_0))[ J_0(z_0)] = \jmath_{-\overline{\eta}}^{-\eta} \circ  f^{(1)}(u_1,J_0(y_0))[ J_0(z_0)]
\end{displaymath}
for $(u_1,y_0,z_0) \in \Lambda \times O_0 \times Y_0$; and similarly for the final requirement of \ref{Hf1}.
\ref{Hf2} holds by the choice of $\rho$ in \eqref{eq:choice_of_rho}.
\ref{Hf4} is once more a consequence of $X_\Sigma \subset \dom(A)$ and $\overline{\eta} < \alpha$.
For \ref{Hf5}, we may apply Lemma~\ref{lemma:wideR}\ref{wideR1} directly.

Applying Lemma~\ref{lemma:6.7}, we infer that $\mathcal{W}_\rho^{-{\eta}} = \jmath_{-\overline{\eta}}^{-\eta} \circ \mathcal{W}_\rho^{-\overline{\eta}} $ is of class $C^1$, the Fréchet  derivative of which,
$D\mathcal{W}_\rho^{-{\eta}}(u_1) \in \mathcal{L}(X_\Sigma; BC^{-\eta}(\mathbb{R}^{\leq 0};X))$,
is the unique solution to
\begin{displaymath}
     \mathcal{F}_\rho^1(u_1,A^{(1)} ) = A^{(1)} ,
\end{displaymath}
where $\mathcal{F}_\rho^1 : X_\Sigma \times \mathcal{L}(X_\Sigma; BC^{-\eta}(\mathbb{R}^{\leq 0};X)) \to \mathcal{L}(X_\Sigma; BC^{-\eta}(\mathbb{R}^{\leq 0};X))$ is given by
\begin{equation}
    T(\cdot) + \mathcal{K} \circ \widehat{R}_\rho^{(1)} \big(\mathcal{W}_\rho^{-\overline{\eta}}(u_1) \big) [A^{(1)}] .
    \label{eq:F1}
\end{equation}
Since $\eta $ was arbitrary, we infer this property holds over all $ \eta \in [\tilde{\eta},\overline{\eta})$.

Choose $\sigma_{k-1} <\cdots<\sigma_1 < \overline\eta$ such that $k\sigma_{k-1} > \tilde{\eta}$.
For higher degrees of smoothness, we proceed by induction  on $1 \leq j \leq k-1$.
We assume, as in the statement Theorem~\ref{thm:main1}\ref{thm1st2}, that the spectral gap satisfies $\beta < k \alpha < 0$ (this is no less general than the original statement with $\ell$). 
The induction hypotheses are:
\begin{enumerate}[label =(IH.\arabic*)]
    \item \label{IH1} For all $1 \leq \ell \leq j$, the map $\mathcal{W}_\rho^{-{\eta}}$ is of class $C^\ell$ for $\eta \in [\tilde{\eta},\ell \sigma_\ell]$. 
    \item \label{IH2} For $\eta \in [\tilde{\eta},j \sigma_j]$, $D^j \mathcal{W}_\rho^{-{\eta}}(u_1)$ is the unique solution of an equation of the form
    \begin{displaymath}
        \mathcal{F}_\rho^j(u_1,A^{(j)} ) = A^{(j)} ,
    \end{displaymath}
    where $\mathcal{F}_\rho^j : X_\Sigma \times M_j(X_\Sigma; BC^{- \eta}(\mathbb{R}^{\leq 0};X)) \to M_j(X_\Sigma; BC^{-  \eta}(\mathbb{R}^{\leq 0};X))$ is given by
    \begin{displaymath}
        \mathcal{F}_\rho^j(u_1,A^{(j)} ) =  \mathcal{K} \circ \widehat{R}_\rho^{(1)} \big(\mathcal{W}_\rho^{-\sigma_j}(u_1) \big) [A^{(j)}] + H_j(u_1),  \qquad \text{for } j \geq 1,
    \end{displaymath}
    where $H_j : X_\Sigma \to M_j(X_\Sigma; BC^{-\eta}(\mathbb{R}^{\leq 0};X)) $ is given by  $H_1(u_1) = T(\cdot)$ for all $u_1 \in X_\Sigma$ if $j = 1$;
    if $j \geq 2$, $H_j$ is given by a finite sum of terms of the form
    \begin{equation}
        \jmath_{-j \sigma_j}^{-\eta} \circ\mathcal{K} \circ \widehat{R}_\rho^{(\ell)} \big(\mathcal{W}_\rho^{-\sigma_j}(u_1) \big) \left[  D^{r_1} \mathcal{W}_\rho^{-r_1{\sigma_j}}(u_1) ,\ldots,D^{r_\ell} \mathcal{W}_\rho^{-r_\ell{\sigma_j}}(u_1)  \right]
        \label{eq:Hj}
    \end{equation}
    with $2 \leq \ell \leq j$, $1 \leq r_i < j$ for $1 \leq i \leq \ell$, and $r_1 + \ldots + r_\ell =j$.
    \item \label{IH3} $\mathcal{F}_\rho^j$ is a contraction on the second factor, uniformly in $\eta \in [\tilde{\eta},j \sigma_j]$.
\end{enumerate}

That \ref{IH1}-\ref{IH3} hold for $j = 1$ was shown above; noting that $\mathcal{W}_\rho^{-\overline{\eta}}(u_1) $ in the argument of $\widehat{R}_\rho^{(1)}$ may be replaced by $\mathcal{W}_\rho^{-\sigma_1}(u_1) $, using \eqref{eq:domain_Rhat}.
(Uniformity of the contraction in \ref{IH3} is satisfied by the choice of $\rho$, as specified in \eqref{eq:choice_of_rho}.)
We may hence suppose \ref{IH1}-\ref{IH3} hold for $2 \leq j < k-1$; we show they are then satisfied for $j+1$.
For this, fix $\eta \in [\tilde{\eta},(j+1) \sigma_j)$  and $\xi \in (\eta,(j+1)\sigma_j)$.

We shall apply Lemma~\ref{lemma:6.7} with
\begin{align*}
    &O_0 = Y_0 = M_j(X_\Sigma;BC^{-j \sigma_j}(\mathbb{R}^{\leq 0};X)), \quad Y = M_j(X_\Sigma;BC^{-\xi}(\mathbb{R}^{\leq 0};X)), \\  
    &Y_1 = M_j(X_\Sigma;BC^{-\eta}(\mathbb{R}^{\leq 0};X)), \quad
    \Lambda  = X_\Sigma, \quad f = \mathcal{F}_\rho^{j}, \\ 
    &f^{(1)}(u_1,\cdot) =  \mathcal{K} \circ \widehat{R}_\rho^{(1)} \big(\mathcal{W}_\rho^{- \sigma_j}(u_1) \big) \in \mathcal{L}(Y), \; u_1 \in X_\Sigma, \\
    & f^{(1)}_1(u_1,\cdot) =  \mathcal{K} \circ \widehat{R}_\rho^{(1)} \big(\mathcal{W}_\rho^{- \sigma_j}(u_1) \big) \in \mathcal{L}(Y_1), \; u_1 \in X_\Sigma,
\end{align*}
with $J_0 = \jmath_{-j \sigma_j}^{-\xi}$ and $J = \jmath_{-\xi}^{-\eta}$ defined as in \eqref{eq:jmathMj}.

We have, for $g$ as in \ref{Hf1},
\begin{displaymath}
    g(u_1,A^{(j)}) =\jmath_{-\xi}^{-\eta} \circ \mathcal{K} \circ \widehat{R}_\rho^{(1)} \big(\mathcal{W}_\rho^{- \sigma_j}(u_1) \big) [\jmath_{-j \sigma_j}^{-\xi} \circ  A^{(j)}] + H_j(u_1).
\end{displaymath}
Since $g$ is bounded linear in $A^{(j)}$, it is $C^1$ in its second argument if $u_1 \mapsto \jmath_{-\xi}^{-\eta} \circ \mathcal{K} \circ \widehat{R}_\rho^{(1)} \big(\mathcal{W}_\rho^{-\sigma_j}(u_1) \big) \circ \jmath_{-j \sigma_j}^{-\xi} $ is continuous.
The latter assertion follows from Lemma~\ref{lemma:wideR}\ref{wideR1}, $\eta <j\sigma_j$, and the continuity of $\mathcal{W}_\rho^{- \sigma_j}$, as $\sigma_j < \overline{\eta}$. 
With respect to its first argument, we may apply Lemma~\ref{lemma:wideR}\ref{wideR3} (to both terms) to conclude it is $C^1$, since $\eta <   (j+1)\sigma_j$ and $\mathcal{W}_\rho^{-\sigma_j}:X_\Sigma \to V_\rho^{-\sigma_j}$ is $C^1$.
The rest of \ref{Hf1} is obvious -- noting that $f^{(1)}$ and $f^{(1)}_1$ indeed take values in $\mathcal{L}(Y)$ and $\mathcal{L}(Y_1)$ by Lemma~\ref{lemma:wideR}\ref{wideR1}.
The same argument also shows \ref{Hf4}, given that $\xi < (j+1) \sigma_j$.
\ref{Hf2} follows from induction hypothesis \ref{IH3}.
\ref{Hf3} is a consequence of induction hypotheses \ref{IH1} and \ref{IH2}.
\ref{Hf5} is once more an application of Lemma~\ref{lemma:wideR}\ref{wideR1}, noting that $\eta < \xi$.

We may hence apply Lemma~\ref{lemma:6.7} to obtain that $u_1 \mapsto D^{j} \mathcal{W}^{-\eta}_\rho(u_1)$ is $C^1$; its derivative $D^{j+1}\mathcal{W}^{-\eta}_\rho(u_1) \in \mathcal{L}(X_\Sigma ; M_j(X_\Sigma;BC^{-\eta}(\mathbb{R}^{\leq 0};X))) \cong M_{j+1}(X_\Sigma;BC^{-\eta}(\mathbb{R}^{\leq 0};X)) $ satisfies 
\begin{equation}
     A^{(j+1)}  = \mathcal{K} \circ \widehat{R}_\rho^{(1)} \big(\mathcal{W}_\rho^{- \sigma_j}(u_1) \big)[A^{(j+1)}] + H_{j+1}(u_1),
     \label{eq:Aj+1}
\end{equation}
where, by \eqref{eq:Rhatcompderiv},
\begin{equation}
    H_{j+1}(u_1) = \jmath_{-(j+1) \sigma_j}^{-\eta} \circ  \mathcal{K} \circ \widehat{R}_\rho^{(2)} \big(\mathcal{W}_\rho^{- \sigma_j}(u_1) \big)[D^{j} \mathcal{W}^{-j\sigma_j}_\rho(u_1),D\mathcal{W}_\rho^{- \sigma_j}(u_1)] + DH_j(u_1).
    \label{eq:Hj+1}
\end{equation} 
Note that a sample term of $H_j(u_1)$ from \eqref{eq:Hj} corresponds to a term 
\begin{gather}
       \jmath_{-(j+1) \sigma_j}^{-\eta} \circ\mathcal{K} \circ \widehat{R}_\rho^{(\ell+1)} \big(\mathcal{W}_\rho^{-\sigma_j}(u_1) \big) \left[  D^{r_1} \mathcal{W}_\rho^{-{r_1\sigma_j}}(u_1) ,\ldots,D^{r_\ell} \mathcal{W}_\rho^{-r_\ell{\sigma_j}}(u_1),D \mathcal{W}_\rho^{-\sigma_j}(u_1) \right] \nonumber \\
       +\sum_{l = 1}^\ell \jmath_{-(j+1) \sigma_j}^{-\eta} \circ\mathcal{K} \circ \widehat{R}_\rho^{(\ell)} \big(\mathcal{W}_\rho^{-\sigma_j}(u_1) \big) \left[  D^{r_1} \mathcal{W}_\rho^{-r_1{\sigma_j}}(u_1) ,\ldots,D^{r_l+1} \mathcal{W}_\rho^{-{r_l\sigma_j}}(u_1),\ldots,D^{r_\ell} \mathcal{W}_\rho^{-r_\ell{\sigma_j}}(u_1) \right]
       \label{eq:Hjupdate}
\end{gather}
of $DH_j(u_1)$, using \eqref{eq:Rhatcompderiv} once more.
(The $\jmath$ inclusion terms appear as a consequence of considering \eqref{eq:Aj+1} on $M_{j+1}(X_\Sigma;BC^{-\eta}(\mathbb{R}^{\leq 0};X))$.)
This implies $\mathcal{W}^{-\eta}_\rho$ is of class $C^{j+1}$ for $\eta \in [\tilde{\eta},(j+1)\sigma_j)$.

Thus, in particular, $\mathcal{W}^{-\eta}_\rho$ is of class $C^{j+1}$ for $\eta \in [\tilde{\eta},(j+1)\sigma_{j+1}]$, showing \ref{IH1} for $j+1$.
On this interval of $\eta$, \ref{IH2} will be shown if we can modify all $\sigma_j$'s in \eqref{eq:Aj+1}, \eqref{eq:Hj+1} and \eqref{eq:Hjupdate} to $\sigma_{j+1}$.
For the $\sigma_j$'s appearing in the first argument of $\widehat{R}_\rho^{(j)}$, we may simply use \eqref{eq:domain_Rhat} -- this already treats \eqref{eq:Aj+1}.
For the rest, we separate the $\jmath_{-(j+1) \sigma_j}^{-\eta}$ terms into $\jmath_{-(j+1) \sigma_{j+1}}^{-\eta} \circ \jmath_{-(j+1) \sigma_j}^{-(j+1) \sigma_{j+1}}$ and move the $\jmath_{-(j+1) \sigma_j}^{-(j+1) \sigma_{j+1}}$ term to the terms within square brackets via \eqref{eq:factorjmathRhat}.
Note that for each term, there is a unique choice of $(\zeta_i)$ -- in the notation of \eqref{eq:factorjmathRhat} -- that transforms each term to their $\sigma_{j+1}$ counterpart.
For instance, the last term(s) in \eqref{eq:Hjupdate} have $\zeta = (r_1 \sigma_{j+1}, \ldots,(r_l+1)\sigma_{j+1},\ldots,r_\ell \sigma_{j+1})$.
The rest of the terms can be treated similarly.
Noting also that $\jmath^{-r_l \sigma_{j+1}}_{-r_l\sigma_j} D^{r_l} \mathcal{W}_\rho^{-r_l{\sigma_j}}(u_1) = D^{r_l} \mathcal{W}_\rho^{-r_l{\sigma_{j+1}}}(u_1) $, we arrive at the form \ref{IH2} for $j + 1$.
Finally, \ref{IH3} for $j+1$ follows from its version for $j$, noting the form of \eqref{eq:Aj+1}.
\end{proof}

Tangency at the origin can be deduced in the exact same way as for center manifolds.
(Note that this has no assumptions on the spectral gap, since neither does $C^1$ smoothness.)

\begin{corollary}[Corollary IX.7.10, \cite{diekmann2012delay}] \label{corollary:tangency}
    The map $\phi_\rho$ from \eqref{eq:phirho},
    \begin{displaymath}
        \phi_\rho = \widetilde{P}_{\Sigma'} \circ\mathrm{ev}_0 \circ \mathcal{W}^{-\eta}_\rho:X_\Sigma \to X_{\Sigma'}
    \end{displaymath}
    satisfies $D\phi_\rho (0) = 0$.
\end{corollary}

\begin{proof}
    Fix $\eta \in [\tilde{\eta},\overline{\eta})$.
    Then, $D \mathcal{W}_\rho^{-\eta}(0)$ must satisfy $\mathcal{F}^1_\rho (0,D \mathcal{W}_\rho^{-\eta}(0)) = D \mathcal{W}_\rho^{-\eta}(0)$, with $\mathcal{F}^1_\rho$ as in \eqref{eq:F1}.
    Hence, noting that $\widehat{R}_\rho^{(1)}(0) = 0$,
    \begin{equation}
        D \mathcal{W}_\rho^{-\eta}(0) = T(\cdot) \imath_\Sigma \in \mathcal{L}(X_\Sigma; BC^{-\eta}(\mathbb{R}^{\leq 0}; X)),
        \label{eq:DmathcalW0}
    \end{equation}
    where $\imath_\Sigma : X_\Sigma \xhookrightarrow{} X$ is the inclusion.
    Writing out the definition of $\phi_\rho$, we have
    \begin{displaymath}
        D \phi_\rho (0) = \widetilde{P}_{\Sigma'} \circ\mathrm{ev}_0 \circ D \mathcal{W}^{-\eta}_\rho (0) = \widetilde{P}_{\Sigma'} \circ \imath_\Sigma = 0. \qedhere
    \end{displaymath}
\end{proof}

\section{Proof of Lemma~\ref{lemma:expansions}}
\label{sect:expansionslemmaproof}

Recall that the manifold $W^\Sigma$ was obtained upon restricting $W^\Sigma_\rho$ to a small enough neighbourhood of the origin (Section~\ref{sect:CHT}). 
Since the conclusions of the lemma are local, we may as well consider $W^\Sigma_\rho$.
Recall also that $W^\Sigma_\rho$ can be obtained as the fixed point of the contraction mapping \eqref{eq:Frho}; explicitly (c.f.\ \eqref{eq:FrhoW} -- we dispose of the $\rho$ and $\eta$ sub/superscripts),
\begin{multline}
    [\mathcal{W}(u_1)](t) = T(t) u_1 - \imath^{-1} \int_t^0 T^{\odot *} (t-s) P_{\Sigma}^{\odot *} R(\mathcal{W}(u_1)(s)) \, ds \\
        + \imath^{-1} \int_{-\infty}^t T^{\odot *} (t-s) P_{\Sigma'}^{\odot *} R(\mathcal{W}(u_1)(s)) \, ds, \label{eq:mathcalW4exp}
\end{multline}
where $\mathcal{W}: X_\Sigma \to BC^{-\eta}(\mathbb{R}^{\leq 0};X)$ maps points of $X_\Sigma$ to negative semiorbits in $W^\Sigma_\rho$, and where  $\eta \in (\beta,\alpha)$ is chosen such that $\mathcal{W}$ is $C^\ell$ (as in Appendix~\ref{sect:mfdsmoothness}).

The desired embedding $K$ is simply (a restriction of) $\mathrm{ev}_0 \circ \mathcal{W}$,
\begin{displaymath}
    K(u_1) = u_1 + \imath^{-1}\int_0^\infty T^{\odot *} (s) P_{\Sigma'}^{\odot *} R(\mathcal{W}(u_1)(-s)) \, ds.
\end{displaymath}
Considering $\mathcal{W}(0) = 0$ and $DR(0) = 0$, we obtain $DK(0) = \imath_\Sigma$ (as in Corollary~\ref{corollary:tangency}).
Differentiating once more,
\begin{displaymath}
    D^2K(0) = \imath^{-1}\int_0^\infty T^{\odot *} (s) P_{\Sigma'}^{\odot *} D^2R(0)[D\mathcal{W}(0)(-s),D\mathcal{W}(0)(-s)] \, ds.
\end{displaymath}
Considering that $D\mathcal{W}(0) (s) = T(s) \imath_\Sigma$ (see \eqref{eq:DmathcalW0}), we arrive at the desired formula.

Similarly to Section~\ref{sect:proofmain2}, the reduced dynamics can be written as
\begin{equation}
    \psi_t (u_1) = \widetilde{P}_\Sigma \circ \mathrm{ev}_t \circ \mathcal{W}(u_1)
    \label{eq:psi_in_last_lemma}
\end{equation}
where defined.
To obtain the vector field $H$, we differentiate \eqref{eq:psi_in_last_lemma} in time (we may do this since $\psi$ is jointly $C^\ell$, as per Theorem~\ref{thm:main1}).
Explicitly, via \eqref{eq:mathcalW4exp}, we have
\begin{displaymath}
     \frac{d}{dt} \Big|_{t = 0} \psi_t(u_1) = A_\Sigma u_1  + \widetilde{P}_\Sigma \imath^{-1} P_\Sigma^{\odot *} R (\mathcal{W}(u_1)(0)),
\end{displaymath}
hence
\begin{displaymath}
    H = A_\Sigma +  \widetilde{P}_\Sigma^{\odot *} R \circ K .
\end{displaymath}
The expansion formula follows immediately, considering $DR(0) = 0$.

\section*{Acknowledgments}
We are grateful to Robert Szalai for introducing GB to several references on sun-star calculus and the work of Garay \cite{garay2005brief}.

\printbibliography

\end{document}